\documentclass[a4paper,reqno]{paper}

\usepackage[latin1]{inputenc}
\usepackage[T1]{fontenc}
\DeclareSymbolFont{rsfscript}{OMS}{rsfs}{m}{b}
\DeclareSymbolFontAlphabet{\mathrsfs}{rsfscript}
\usepackage{amsfonts,enumerate}
\usepackage{amsmath}
\usepackage{amssymb,pgf}
\usepackage{colortbl}
\usepackage{amsthm}
\usepackage[absolute]{textpos}
\usepackage{stmaryrd}
\usepackage{subcaption}
\usepackage{dsfont}
\usepackage{float}
\usepackage{pstricks,tikz}
\usepackage{graphicx}
\usepackage{pst-grad}
\usepackage{multido}

\definecolor{asse}{RGB}{3,172,94}

\def\sR{\mathsf{R}}
\def\fs{\mathfrak{s}}

\def\ff{\mathfrak{f}}

\newenvironment{psmallmatrix}
  {\left(\begin{smallmatrix}}
  {\end{smallmatrix}\right)}

\usepackage{geometry}
\def\finl{~$\SS{\blacksquare}$}
\def\rfinl{\begin{right}\finl\end{right}}

\newtheorem{Th}{Theorem}[section]
\newtheorem{Lem}[Th]{Lemma}
\newtheorem{Cor}[Th]{Corollary}
\newtheorem{Prop}[Th]{Proposition}
\newtheorem{Def-Prop}[Th]{Definition-Proposition}

\theoremstyle{definition}
\newtheorem{Def}[Th]{Definition}
\newtheorem{Exa}[Th]{Example}

\newtheorem{Rem}[Th]{Remark}
\numberwithin{equation}{section}

\definecolor{purple}{rgb}{0.8,0.12,0.8}
\definecolor{orange}{rgb}{1.0,0.7,0.0}
\definecolor{pink}{rgb}{1,0.5,0.8}
\definecolor{blackg}{rgb}{0.1,0.25,0.1}
\definecolor{ForestGreen}{cmyk}{0.91,0,0.88,0.42}
\definecolor{Turquoise}{cmyk}{0.85,0,0.20,0}

\definecolor{GreenYellow}{cmyk}{0.15,0,0.69,0} 
\definecolor{Yellow}{cmyk}{0,0,1.,0} 
\definecolor{Goldenrod}{cmyk}{0,0.10,0.84,0} 
\definecolor{Dandelion}{cmyk}{0,0.29,0.84,0} 
\definecolor{Apricot}{cmyk}{0,0.32,0.52,0} 
\definecolor{Peach}{cmyk}{0,0.50,0.70,0} 
\definecolor{Melon}{cmyk}{0,0.46,0.50,0} 
\definecolor{YellowOrange}{cmyk}{0,0.42,1.,0} 
\definecolor{Orange}{cmyk}{0,0.61,0.87,0} 
\definecolor{BurntOrange}{cmyk}{0,0.51,1.,0} 
\definecolor{Bittersweet}{cmyk}{0,0.75,1.,0.24} 
\definecolor{RedOrange}{cmyk}{0,0.77,0.87,0} 
\definecolor{Mahogany}{cmyk}{0,0.85,0.87,0.35} 
\definecolor{Maroon}{cmyk}{0,0.87,0.68,0.32} 
\definecolor{BrickRed}{cmyk}{0,0.89,0.94,0.28} 
\definecolor{Red}{cmyk}{0,1.,1.,0} 
\definecolor{OrangeRed}{cmyk}{0,1.,0.50,0} 
\definecolor{RubineRed}{cmyk}{0,1.,0.13,0} 
\definecolor{WildStrawberry}{cmyk}{0,0.96,0.39,0} 
\definecolor{Salmon}{cmyk}{0,0.53,0.38,0} 
\definecolor{CarnationPink}{cmyk}{0,0.63,0,0} 
\definecolor{Magenta}{cmyk}{0,1.,0,0} 
\definecolor{VioletRed}{cmyk}{0,0.81,0,0} 
\definecolor{Rhodamine}{cmyk}{0,0.82,0,0} 
\definecolor{Mulberry}{cmyk}{0.34,0.90,0,0.02} 
\definecolor{RedViolet}{cmyk}{0.07,0.90,0,0.34} 
\definecolor{Fuchsia}{cmyk}{0.47,0.91,0,0.08} 
\definecolor{Lavender}{cmyk}{0,0.48,0,0} 
\definecolor{Thistle}{cmyk}{0.12,0.59,0,0} 
\definecolor{Orchid}{cmyk}{0.32,0.64,0,0} 
\definecolor{DarkOrchid}{cmyk}{0.40,0.80,0.20,0} 
\definecolor{Purple}{cmyk}{0.45,0.86,0,0} 
\definecolor{Plum}{cmyk}{0.50,1.,0,0} 
\definecolor{Violet}{cmyk}{0.79,0.88,0,0} 
\definecolor{RoyalPurple}{cmyk}{0.75,0.90,0,0} 
\definecolor{BlueViolet}{cmyk}{0.86,0.91,0,0.04} 
\definecolor{Periwinkle}{cmyk}{0.57,0.55,0,0} 
\definecolor{CadetBlue}{cmyk}{0.62,0.57,0.23,0} 
\definecolor{CornflowerBlue}{cmyk}{0.65,0.13,0,0} 
\definecolor{MidnightBlue}{cmyk}{0.98,0.13,0,0.43} 
\definecolor{NavyBlue}{cmyk}{0.94,0.54,0,0} 
\definecolor{RoyalBlue}{cmyk}{1.,0.50,0,0} 
\definecolor{Blue}{cmyk}{1.,1.,0,0} 
\definecolor{Cerulean}{cmyk}{0.94,0.11,0,0} 
\definecolor{Cyan}{cmyk}{1.,0,0,0} 
\definecolor{ProcessBlue}{cmyk}{0.96,0,0,0} 
\definecolor{SkyBlue}{cmyk}{0.62,0,0.12,0} 
\definecolor{Turquoise}{cmyk}{0.85,0,0.20,0} 
\definecolor{TealBlue}{cmyk}{0.86,0,0.34,0.02} 
\definecolor{Aquamarine}{cmyk}{0.82,0,0.30,0} 
\definecolor{BlueGreen}{cmyk}{0.85,0,0.33,0} 
\definecolor{Emerald}{cmyk}{1.,0,0.50,0} 
\definecolor{JungleGreen}{cmyk}{0.99,0,0.52,0} 
\definecolor{SeaGreen}{cmyk}{0.69,0,0.50,0} 
\definecolor{Green}{cmyk}{1.,0,1.,0} 
\definecolor{ForestGreen}{cmyk}{0.91,0,0.88,0.12} 
\definecolor{PineGreen}{cmyk}{0.92,0,0.59,0.25} 
\definecolor{LimeGreen}{cmyk}{0.50,0,1.,0} 
\definecolor{YellowGreen}{cmyk}{0.44,0,0.74,0} 
\definecolor{SpringGreen}{cmyk}{0.26,0,0.76,0} 
\definecolor{OliveGreen}{cmyk}{0.64,0,0.95,0.40} 
\definecolor{RawSienna }{cmyk}{0,0.72,1.,0.45} 
\definecolor{Sepia}{cmyk}{0,0.83,1.,0.70} 
\definecolor{Brown}{cmyk}{0,0.81,1.,0.60} 
\definecolor{Tan}{cmyk}{0.14,0.42,0.56,0} 
\definecolor{Gray}{cmyk}{0,0,0,0.50} 
\definecolor{Black}{cmyk}{0,0,0,1.} 
\definecolor{White}{cmyk}{0,0,0,0} 

\colorlet{pi0}{red!5!yellow!42!}
\colorlet{pi1}{NavyBlue}
\colorlet{pi2}{PineGreen}
\colorlet{pi3}{green!80!black!70!}
\colorlet{pi4}{lightgray}
\colorlet{pi5}{RedOrange}
\colorlet{pi7}{BurntOrange} 
\colorlet{pi8}{Tan}
\colorlet{pi9}{VioletRed} 
\definecolor{pi6}{cmyk}{0,0.53,0.38,0} 
\definecolor{pi10}{cmyk}{0.86,0.91,0,0.04}
\colorlet{pi11}{BlueGreen} 
\colorlet{pi12}{orange}
\colorlet{nonirred}{BrickRed}

\definecolor{pp}{RGB}{215,25,28}
\definecolor{pm}{RGB}{253,174,97}
\definecolor{mp}{RGB}{171,221,164}
\definecolor{mm}{RGB}{43,131,186}

\definecolor{green1}{RGB}{153,216,201}
\definecolor{green2}{RGB}{44,162,95}

\definecolor{blue1}{RGB}{158,202,225}
\definecolor{blue2}{RGB}{49,130,189}

\newcommand{\cB}{\mathcal{B}}

\newcommand{\cD}{\mathcal{D}}

\newcommand{\cH}{\mathcal{H}}

\newcommand{\cL}{\mathcal{L}}
\newcommand{\cLR}{\mathcal{LR}}
\newcommand{\cM}{\mathcal{M}}

\newcommand{\cP}{\mathcal{P}}
\newcommand{\cQ}{\mathcal{Q}}
\newcommand{\cR}{\mathcal{R}}

\newcommand{\cU}{\mathcal{U}}

\newcommand{\cZ}{\mathcal{Z}}

\newcommand{\ba}{\mathbf{a}}

\newcommand{\bx}{\mathbf{x}}

\newcommand{\fc}{\mathfrak{c}}

\newcommand{\sw}{\mathsf{w}}

\newcommand{\su}{\mathsf{u}}

\newcommand{\sv}{\mathsf{v}}

\newcommand{\sa}{\mathsf{a}}

\newcommand{\st}{{\mathsf t}}

\newcommand{\sq}{\mathsf{q}}
\newcommand{\sQ}{\mathsf{Q}}
\newcommand{\sC}{\mathsf{C}}

\newcommand{\sJ}{\mathsf{J}}

\newcommand{\sB}{\mathsf{B}}

\newcommand{\tsc}{\mathsf{\Lambda}}
\newcommand{\sfp}{\mathsf{\Pi}}

\newcommand{\nZ}{\mathbb{Z}}
\newcommand{\nR}{\mathbb{R}}
\newcommand{\nN}{\mathbb{N}}
\newcommand{\nQ}{\mathbb{Q}}

\newcommand{\si}{\sigma}
\newcommand{\la}{\lambda}
\newcommand{\ga}{\gamma}
\newcommand{\eps}{\varepsilon}

\newcommand{\Ga}{\Gamma}
\newcommand{\De}{\Delta}

\newcommand{\Up}{\Upsilon}

\newcommand{\tba}{\tilde{\ba}}

\newcommand{\ra}{\rightarrow}
\newcommand{\Lra}{\Longrightarrow}

\newcommand{\eq}{\Longleftrightarrow}

\newcommand{\barr}{\begin{array}{cccccccccc}}
\newcommand{\ear}{\end{array}}

\newcommand{\pmat}{\begin{pmatrix}}
\newcommand{\emat}{\end{pmatrix}}

\DeclareMathOperator{\Tr}{Tr}

\newcommand{\ben}{\begin{enumerate}}
\newcommand{\een}{\end{enumerate}}

\newcommand{\bit}{\begin{itemize}}
\newcommand{\eit}{\end{itemize}}

{ \begin{list}%
	{$\rightarrow$}%
	{\setlength{\labelwidth}{30pt}%
	 \setlength{\leftmargin}{20pt}%
	 \setlength{\itemsep}{.05cm}}}%
{ \end{list} }

\newcommand{\ov}{\overline}

\newcommand{\sg}{\langle}
\newcommand{\sd}{\rangle}

\newcommand{\conj}[1]{{\bf P#1}}

\newcommand{\B}[1]{{\bf B#1}}

\newcommand{\quand}{\quad\text{ and }\quad}

\def\SS{\scriptstyle}

\def\finl{~$\SS{\blacksquare}$}
\def\rfinl{\begin{right}\finl\end{right}}

\newcommand{\irreps}{\mathrm{Irrep}(\overline{\cH}_{\mathbb{C}})}

\newcommand{\tbu}{{\tiny $\bullet$}}

\newenvironment{maliste}%
{ \begin{list}%
	{\tbu}%
	{\setlength{\labelwidth}{30pt}%
	 \setlength{\leftmargin}{20pt}%
	 \setlength{\itemsep}{.05cm}}}%
{ \end{list} }

\newcommand{\bem}{\begin{maliste}}
\newcommand{\eem}{\end{maliste}}

\newcommand{\gen}{\sJ}
\DeclareMathOperator{\Irr}{Rep}

\begin{document}

\title{Balanced representations, the asymptotic Plancherel formula, and Lusztig's conjectures for $\tilde{C}_2$}

\author{J\'er\'emie Guilhot, James Parkinson}
\date{\today}
\maketitle

\def\finl{~$\SS{\blacksquare}$}
\def\rfinl{\begin{right}\finl\end{right}}

\parindent=0mm

\abstract{We prove Lusztig's conjectures ${\bf P1}$--${\bf P15}$ for the affine Weyl group of type $\tilde{C}_2$ for all choices of positive weight function. Our approach to computing Lusztig's $\mathbf{a}$-function is based on the notion of a ``balanced system of cell representations''. Once this system is established roughly half of the conjectures ${\bf P1}$--${\bf P15}$ follow. Next we establish an ``asymptotic Plancherel Theorem'' for type $\tilde{C}_2$, from which the remaining conjectures follow. Combined with existing results in the literature this completes the proof of Lusztig's conjectures for all rank~$1$ and~$2$ affine Weyl groups for all choices of parameters.} 

%



\section*{Introduction}

The theory of Kazhdan-Lusztig cells plays a fundamental role in the representation theory of Coxeter groups and Hecke algebras. In their celebrated paper~\cite{KL1} Kazhdan and Lusztig introduced the theory in the equal parameter case, and in \cite{Lus1p} Lusztig generalised the construction to the case of arbitrary parameters. A very specific feature in the equal parameter case is the geometric interpretation of Kazhdan-Lusztig theory, which implies certain ``positivity properties'' (such as the positivity of the structure constants with respect to the Kazhdan-Lusztig basis). This was proved in the finite and affine cases by Kazhdan and Lusztig in~\cite{KL2}, and the case of arbitrary Coxeter groups was settled only very recently by Elias and Williamson in~\cite{EW:14}. However, simple examples show that these positivity properties no longer hold for unequal parameters, hence the need to develop new methods to deal with the general case. 
\medskip

A major step in this direction was achieved by Lusztig in his book on Hecke algebras with unequal parameters \cite[Chapter~14]{bible} where he introduced 15 conjectures \conj{1}--\conj{15} which capture essential properties of cells for all choices of parameters. In the case of equal parameters these conjectures can be proved for finite and affine types using the above mentioned geometric interpretation (see \cite{bible}). For arbitrary parameters the existing state of knowledge is much less complete.

\medskip

Recently in \cite{GP:17} we developed an approach to proving~$\conj{1}$--$\conj{15}$ and applied it to the case $\tilde{G}_2$ with arbitrary parameters. This provided the first irreducible affine Coxeter group, apart from the infinite dihedral group, where Lusztig's conjectures have been established for arbitrary parameters. Indeed outside of the equal parameter case {\bf P1}--{\bf P15} are only known to hold in the following very limited number of cases (see \cite[Part VII]{Bon:17}): 
\bem
\item the \textit{quasisplit} case where a geometric interpretation is available~\cite[Chapter~16]{bible};
\item finite dihedral type~\cite{Geck:11} and infinite dihedral type~\cite[Chapter~17]{bible} for arbitrary parameters;
\item universal Coxeter groups for arbitrary parameters~\cite{SY:15};
\item finite type $B_n$ in the ``asymptotic'' parameter case \cite{B-I,Geck:11};
\item finite type $F_4$ for arbitrary parameters \cite{Geck:11};
\item affine type $\tilde{G}_2$ for arbitrary parameters \cite{GP:17}.
\eem

Our approach in \cite{GP:17} hinges on two main ideas: (a) the notion of a \textit{balanced system of cell representations} for the Hecke algebra, and (b) the \textit{asymptotic Plancherel formula}. In the present paper we develop these ideas in type $\tilde{C}_2$. This three parameter case turns out to be considerably more complicated than the two parameter $\tilde{G}_2$ case, and this additional complexity requires us to take a somewhat more conceptual approach here.
\medskip

We now briefly describe the ideas (a) and (b) above. Let~$(W,S)$ be a Coxeter system with weight function $L:W\to\mathbb{N}_{>0}$ and associated multi-parameter Hecke algebra $\cH$ defined over $\nZ[\sq,\sq^{-1}]$. Let $\tsc$ be the set of two-sided cells of $W$ with respect to~$L$, and recall that there is a natural partial order $\leq_{\cLR}$ on the set~$\tsc$. Let $(C_w)_{w\in W}$ denote the Kazhdan-Lusztig basis of~$\cH$.
\medskip

One of the main challenges in proving Lusztig's conjectures is to compute Lusztig's $\ba$-function since, in principle, it requires us to have information on all the structure constants with respect to the Kazhdan-Lusztig basis. In \cite{GP:17} we showed that the existence of a balanced system of cell representations is sufficient to compute the $\ba$-function. Such a system is a family $(\pi_{\Ga})_{\Ga\in\tsc}$ of representations of $\cH$, each equipped with a distinguished basis, satisfying various axioms including (1) $\pi_{\Ga}(C_w)=0$ for all $w\in\Gamma'$ with $\Ga'\not\geq_{\cLR}\Ga$, (2) the maximal degree of the coefficients that appear in the matrix $\pi_{\Ga}(C_w)$ is bounded by a constant~$\ba_{\pi_{\Ga}}$, (3) this bound is attained if and only if $w\in \Ga$. This concept is inspired by the work of Geck~\cite{Geck:11} in the finite dimensional case.
\medskip
%
%

Thus a main part of the present paper is devoted to establishing a balanced system of cell representations in type $\tilde{C}_2$ for each choice of parameters. For this purpose we use the explicit partition of $W$ into Kazhdan-Lusztig cells that was obtained by the first author in~\cite{guilhot4}. It turns out that the representations associated to finite cells naturally give rise to balanced representations and so most of our work is concerned with the infinite cells. In type $\tilde{C}_2$ there are either 3 or 4 such two-sided cells depending on the choice of parameters. To each of these two-sided cells we associate a natural finite dimensional representation admitting an elegant combinatorial description in terms of alcove paths. Using this description we are able to give a combinatorial proof of the balancedness of these representations. In fact we study these representations as representations of the ``generic'' affine Hecke algebra of type $\tilde{C}_2$, thereby effectively analysing all possible choices of parameters simultaneously. 
\medskip

Once a balanced system of cell representations is established for each choice or parameters we are able to compute Lusztig's $\ba$-function for type $\tilde{C}_2$, and combined with the explicit partition of $W$ into cells the conjectures $\conj{4}$, $\conj{8}$, $\conj{9}$, $\conj{10}$, $\conj{11}$, $\conj{12}$, and $\conj{14}$ readily follow.
\medskip

The second main part of this paper is establishing an ``asymptotic'' Plancherel formula for type $\tilde{C}_2$, with our starting point being the explicit formulation of the Plancherel Theorem in type $\tilde{C}_2$ obtained by the second author in~\cite{Par:14} (this is in turn a very special case of Opdam's general Plancherel Theorem~\cite{Op:04}). In particular we show that in type $\tilde{C}_2$ there is a natural correspondence, in each parameter range, between two-sided cells appearing in the cell decomposition and the representations appearing in the Plancherel Theorem (these are the \textit{tempered} representations of~$\cH$). Moreover we define a \textit{$\sq^{-1}$-valuation} on the Plancherel measure, and show that in type $\tilde{C}_2$ the $\sq^{-1}$-valuation of the mass of a tempered representation is twice the value of Lusztig's $\ba$-function on the associated cell. This observation allows us to introduce an \textit{asymptotic Plancherel measure}, giving a descent of the Plancherel formula to Lusztig's asymptotic algebra~$\mathcal{J}$. In particular we obtain an inner product on~$\mathcal{J}$, giving a satisfying conceptual proof of~$\conj{1}$ and $\conj{7}$. Moreover we are able to determine the set $\cD$ of Duflo involutions, and conjectures $\conj{2}$, $\conj{3}$, $\conj{5}$, $\conj{6}$, and $\conj{13}$ follow naturally. 
\medskip

The remaining conjecture \conj{15} is of a slightly different flavour. In \cite{Xie:15} Xie has proved this conjecture under an assumption on Lusztig's $\ba$-function. We are able to verify this assumption using our calculation of the $\ba$-function and the asymptotic Plancherel formula, hence proving~$\conj{15}$ and completing the proof of all conjectures $\conj{1}$--$\conj{15}$.
\medskip

We conclude this introduction with an outline of the structure of the paper. In Section~\ref{sec:KLtheory} we recall the basics of Kazhdan-Lusztig theory, and we recall the axioms of a balanced system of cell representations from \cite{GP:17}. Section~\ref{sec:WeylGroups} provides background on affine Weyl groups, root systems, the affine Hecke algebra, and the combinatorics of alcove paths. In Section~\ref{sec:3} we recall the partition of $\tilde{C}_2$ into cells for all choices of parameters from~\cite{guilhot4}, and introduce some notions such as the generating set of a two-sided cell, cell factorisation and the $\tba$-function. In Section~\ref{sec:4} we define various representations of the affine Hecke algebra in preparation for the important Sections~\ref{sec:5} and~\ref{sec:infinite} where we establish the existence of the a balanced system of cell representations for each choice of parameters. The main work here is in Section~\ref{sec:infinite}, where we conduct a detailed combinatorial analysis of certain representations associated to the infinite two-sided cells. In Section~\ref{sec:7} we establish connections between the Plancherel Theorem and the decomposition into cells, hence establishing the asymptotic Plancherel Theorem for type $\tilde{C}_2$. The proofs of $\conj{1}$--$\conj{15}$ are given progressively throughout the paper (see Corollaries~\ref{cor:conj14}, \ref{cor:afn}, \ref{cor:P8}, \ref{cor:P6}, \ref{cor:P71}, and Theorems~\ref{thm:P1} and~\ref{thm:conj15}).


\section{Kazhdan-Lusztig theory and balanced cell representations}\label{sec:KLtheory}

In this section we recall the definition of the generic Hecke algebra and the setup of Kazhdan-Lusztig theory, including the Kazhdan-Lusztig basis, Kazhdan-Lusztig cells, and the Lusztig's conjectures ${\bf P1}$--${\bf P15}$. In this section $(W,S)$ denotes an arbitrary Coxeter system (with $|S|<\infty$) with length function $\ell:W\to\mathbb{N}=\{0,1,2,\ldots\}$. For $I\subseteq S$ let $W_I$ be the standard parabolic subgroup generated by~$I$.

\subsection{Generic Hecke algebras and their specialisations}

Let $(\sq_s)_{s\in S}$ be a family of commuting invertible indeterminates with the property that $\sq_s=\sq_{s'}$ whenever $s$ and $s'$ are conjugate in~$W$. Let $\sR_g=\mathbb{Z}[(\sq_s^{\pm 1})_{s\in S}]$. The \textit{generic Hecke algebra} of type $(W,S)$ is the $\sR_g$-algebra $\cH_g$ with basis $\{T_w\mid w\in W\}$ and multiplication given by (for $w\in W$ and $s\in S$)
\begin{align}\label{eq:generic}
T_wT_{s}=\begin{cases}
T_{ws}&\text{if $\ell(ws)=\ell(w)+1$}\\
T_{ws}+(\sq_s-\sq_s^{-1})T_w&\text{if $\ell(ws)=\ell(w)-1$}.
\end{cases}
\end{align}
We set $\sq_w:=\sq_{s_1}\cdots \sq_{s_n}$ where $w=s_1\ldots s_n\in W$ is a reduced expression of $w$. This can easily be seen to be independent of the choice of reduced expression (using Tits' solution to the Word Problem).

\medskip

Let $L:W\to\mathbb{N}$ be a \textit{positive weight function} on~$W$. Thus $L(w)>0$ for all $w\in W$ different from the identity and $L(ww')=L(w)+L(w')$ whenever $\ell(ww')=\ell(w)+\ell(w')$. Let $\sq$ be an invertible indeterminate and let $\sR=\nZ[\sq,\sq^{-1}]$ be the ring of Laurent polynomials in~$\sq$. The \textit{Hecke algebra} of type $(W,S,L)$ is the $\sR$-algebra $\cH=\cH_L$ with basis $\{T_w\mid w\in W\}$ and multiplication given by (for $w\in W$ and $s\in S$)
\begin{align}\label{eq:definingrelations}
T_wT_{s}=\begin{cases}
T_{ws}&\text{if $\ell(ws)=\ell(w)+1$}\\
T_{ws}+(\sq^{L(s)}-\sq^{-L(s)})T_w&\text{if $\ell(ws)=\ell(w)-1$}.
\end{cases}
\end{align}
We refer to $(T_w)_{w\in W}$ as the ``standard basis'' of $\cH$. Of course $\cH$ is obtained from $\cH_g$ via the specialisation $\sq_s\mapsto \sq^{L(s)}$, with the multiplicative property of weight functions ensuring that this specialisation compatible with the fact that $\sq_s=\sq_{s'}$ whenever $s$ and $s'$ are conjugate in~$W$.  For a given weight function $L$, we denote the above specialisation by $\Theta_L:\cH_g\ra \cH$.

\medskip

While Kazhdan-Lusztig theory is setup in terms of the specialised algebra $\cH=\cH_L$, we will also need the generic algebra $\cH_g$ at times in this paper (particularly in Section~\ref{sec:infinite}). We sometimes write $\sQ_s=\sq_s-\sq_s^{-1}$, or $\sQ_s=\sq^{L(s)}-\sq^{-L(s)}$ depending on context (particularly in matrices for typesetting purposes). If $S=\{s_0,\ldots,s_n\}$ we will also often write, for example, $0121$ as shorthand for $s_0s_1s_2s_1$, and thus in the Hecke algebra $T_{0121}=T_{s_0s_1s_2s_1}$. In particular, note that $1$ is shorthand for $s_1$, and therefore to avoid confusion we denote the identity of $W$ by~$e$. 

\subsection{The Kazhdan-Lusztig basis}\label{sec:1.1}

Let $L$ be a positive weight function and let $\cH=\cH_L$. The involution $\bar{\ }$ on $\sR$ which sends $\sq$ to $\sq^{-1}$ can be extended to an involution on $\cH$ by setting 
$$\ov{\sum_{w\in W} a_w T_w}=\sum_{w\in W} \ov{a_w}\, T_{w^{-1}}^{-1}.$$
In \cite{KL1}, Kazhdan and Lusztig proved that there exists a unique basis $\{C_w\mid w\in W\}$ of $\cH$ such that, for all $w\in W$,
$$
\ov{C_w}=C_w\quad\text{and}\quad C_w=T_w+\sum_{y<w} P_{y,w}T_y\quad\text{where $P_{y,w}\in \sq^{-1}\nZ[\sq^{-1}]$}. 
$$
This basis is called the \textit{Kazhdan-Lusztig basis} (KL basis for short) of~$\cH$. The polynomials $P_{y,w}$ are called the \textit{Kazhdan-Lusztig polynomials}, and to complete the definition we set $P_{w,w}=1$ and $P_{y,w}=0$ whenever $y\not<w$ (here $\leq$ denotes Bruhat order on~$W$) and $P_{w,w}=1$ for all $w\in W$. We note that the Kazhdan-Lusztig polynomials, and hence the elements $C_w$, depend on the the weight function $L$ (see the following example). 

\begin{Exa}
\label{Exa:KL-element}
Let $(W,S,L)$ be a Coxeter group and let $J\subseteq S$ be such that the group $W_J$ generated by $J$ is finite. Let~$w_J$ be the longest element of $W_J$. The Kazhdan-Lusztig element $C_{w_J}$ is equal to $\sum_{w\in W_J} \sq^{L(w)-L(w_J)}T_w$. Indeed, this element has the required triangularity with respect to the standard basis and it is stable under the bar involution. Further, if we set $\sC_{w_J}:=\sum_{w\in W} \sq_{w}\sq_{w_J}^{-1}T_w\in \cH_g$  then we have $\Theta_L(\sC_{w_J})=C_{w_J}$ for all positive weight functions $L$ on $W$.  

\medskip
Now assume that $S$ contains two elements $s_1,s_2$ such that $(s_1s_2)^4=e$. If we set $a=L(s_1)$ and $b=L(s_2)$ then we have 
$$C_{212}=
\begin{cases}
T_{212}+\sq^{-b}\left(T_{12}+T_{21}\right)
+\left(\sq^{-b-a}-\sq^{-b+a}\right)T_{2}
+\sq^{-2b}T_{1}+\left(\sq^{-2b-a}-\sq^{-2b+a}\right)T_{e}& \mbox{if $b>a$,}\\
T_{212}+\sq^{-a}\left(T_{21}+T_{12}\right)
+\sq^{-2a}\left(T_{1}+T_{2}\right)
+\sq^{-3a}T_{e}& \mbox{if $a=b$,}\\
T_{212}+\sq^{-b}\left(T_{12}+T_{21}\right)
+\left(\sq^{-a-b}-\sq^{-a+b}\right)T_{2}
+\sq^{-2b}T_{1}+\left(\sq^{-a}-\sq^{-a-2b}\right)T_{e}& \mbox{if $b<a$.}
\end{cases}
$$
Indeed, the expressions on the right-hand side are stable under the bar involution and since they have the required triangularity property, they have to be the Kazhdan-Lusztig element associated to $212$. Unlike the case where~$w=w_J$, there is no generic element in $\cH_g$ that specialises to $C_{212}\in \cH(W,S,L)$ for all positive weight functions $L$. 
 We also note that when $b>a$ we have $P_{2,212}=\sq^{-b-a}-\sq^{-b+a}$, showing that the Kazhdan-Lusztig polynomials can have negative coefficients in the unequal parameter case.
\end{Exa}

Let $x,y\in W$. We denote by $h_{x,y,z}\in \sR$ the structure constants associated to the Kazhdan-Lusztig basis:
$$
C_{x}C_y=\sum_{z\in W} h_{x,y,z}C_z.
$$

\begin{Def}[{\cite[Chapter 13]{bible}}]
The \textit{Lusztig $\ba$-function} is the function $\ba:W\to\mathbb{N}$ defined by
$$\ba(z):=\min\{n\in \nN\mid \sq^{-n}h_{x,y,z}\in \nZ[\sq^{-1}]\text{ for all $x,y\in W$}\}.$$
\end{Def}

When $W$ is infinite it is, in general, unknown whether the $\sa$-function is well-defined. However in the case of affine Weyl groups it is known that $\ba$ is well-defined, and that $\ba(z)\leq L(\sw_0)$ where $\sw_0$ is the longest element of the underlying finite Weyl group $W_0$ (see \cite{bible}). The $\ba$-function is a very important tool in the representation theory of Hecke algebras, and plays a crucial role in the work of Lusztig on the unipotent characters of reductive groups.

\begin{Def}
For $x,y,z\in W$ let $\ga_{x,y,z^{-1}}$ denote the constant term of $\sq^{-\ba(z)}h_{x,y,z}$.
\end{Def}

The coefficients $\gamma_{x,y,z^{-1}}$ are the structure constants of the \textit{asymptotic algebra} $\mathcal{J}$ introduced by Lusztig in \cite[Chapter~18]{bible}.

\subsection{Kazhdan-Lusztig cells and associated representations}\label{sec:1.2}

Define preorders $\leq_{\cL},\leq_{\cR},\leq_{\cLR}$ on $W$ extending the following by transitivity:
\begin{align*}
x&\leq_{\cL}y&&\Longleftrightarrow&&\text{there exists $h\in\cH$ such that $C_x$ appears in the decomposition in the KL basis of $hC_y$,}\\
x&\leq_{\cR}y&&\Longleftrightarrow&&\text{there exists $h\in\cH$ such that $C_x$ appears in the decomposition in the KL basis of $C_yh$,}\\
x&\leq_{\cLR}y&&\Longleftrightarrow&&\text{there exists $h,h'\in\cH$ such that $C_x$ appears in the decomposition in the KL basis of $hC_yh'$.}
\end{align*}
We associate to these preorders equivalence relations $\sim_{\cL}$, $\sim_{\cR}$, and $\sim_{\cLR}$ by setting (for $*\in \{\cL,\cR,\cLR\}$)
$$\text{$x\sim_{*} y$ if and only if $x\leq_{*} y$ and $y\leq_{*} x$}.$$
The equivalence classes of $\sim_{\cL}$, $\sim_{\cR}$, and $\sim_{\cLR}$ are called \textit{left cells}, \textit{right cells}, and \textit{two-sided cells}.  

\begin{Exa} 
\label{exa:prec}
For $y,w\in W$ we write $y\preceq w$ if and only if there exists $x,z\in W$ such that $w=xyz$ and $\ell(w)=\ell(x)+\ell(y)+\ell(y)$.  In this case it is not hard to see, using the unitriangularity of the change of basis matrix from the standard basis to the Kazhdan-Lusztig basis, that $T_xC_y T_z\in C_{w}+\sum_{z<w} a_zC_z$ and therefore  $w\leq_{\cLR} y$. 
\end{Exa}

We denote by $\tsc$ the set of all two-sided cells (note that of course $\tsc$ depends on the choice of weight function). Given any cell $\Ga$ (left, right, or two-sided) we set
$$\Ga_{\leq_\ast}:=\{w\in W\mid\text{there exists $x\in \Ga$} \text{ such that } w\leq_{\ast} x\}$$
and we define $\Ga_{\geq_\ast}$, $\Ga_{>_\ast}$ and $\Ga_{<_\ast}$ similarly.

\medskip

To each right cell $\Up$ of $W$ there is a natural right $\cH$-module $\cH_{\Up}$ constructed as follows. The $\sR$-modules
\begin{align*}
\cH_{\leq_\cR \Up}&:=\sg C_{x}\mid x\in \Up_{\leq_{\cR}}\sd\quand \cH_{<_\cR\Up}:=\sg C_{x}\mid  x\in \Up_{<_{\cR}}\sd
\end{align*}
are  right $\cH$-modules by definition and therefore the quotient
$$\cH_{\Up}:=\cH_{\leq_\cR \Up}\slash \cH_{<_\cR\Up}$$
is a right $\cH$-module with basis $\{\ov{C}_w\mid w\in \Up\}$ where $\ov{C}_w$ is the class of $C_w$ in $\cH_{\Up}$. Given a left cell (respectively a two-sided cell) we can follow a similar construction to produce left $\cH$-modules (respectively $\cH$-bimodules).

\subsection{Lusztig conjectures}\label{sec:1.3}

Define $\De:W\to \nN$ and $n_z\in \mathbb{Z}\backslash\{0\}$ by the relation
$$P_{e,z}=n_z\sq^{-\De(z)}+\text{ strictly smaller powers of $\sq$.}$$
This is well defined because $P_{x,y}\in \sq^{-1}\nZ[\sq^{-1}]$ for all $x,y\in W$. Let
$$\cD=\{w\in W\mid \De(w)=\ba(w)\}.$$
The elements of $\cD$ are called \textit{Duflo elements} (or, somewhat prematurely, \textit{Duflo involutions}; see \conj{6} below). 
\medskip

In \cite[Chapter 13]{bible},  Lusztig has formulated  the following 15 conjectures, now known as ${\bf P1}$--${\bf P15}$.
\begin{itemize}
\item[\bf P1.] For any $z\in W$ we have $\ba(z)\leq \Delta(z)$.
\item[\bf P2.] If $d \in \cD$ and $x,y\in W$ satisfy $\gamma_{x,y,d}\neq 0$,
then $y=x^{-1}$.
\item[\bf P3.] If $x\in W$ then there exists a unique $d\in \cD$ such that
$\gamma_{x,x^{-1},d}\neq 0$.
\item[\bf P4.] If $z'\leq_{\cLR} z$ then $\ba(z')\geq \ba(z)$. In particular the $\ba$-function is constant on two-sided cells.
\item[\bf P5.] If $d\in \cD$, $x\in W$, and $\gamma_{x,x^{-1},d}\neq 0$, then
$\gamma_{x,x^{-1},d}=n_d=\pm 1$.
\item[\bf P6.] If $d\in \cD$ then $d^2=e$ (the identity).
\item[\bf P7.] For any $x,y,z\in W$, we have $\gamma_{x,y,z}=\gamma_{y,z,x}$.
\item[\bf P8.] Let $x,y,z\in W$ be such that $\gamma_{x,y,z}\neq 0$. Then
$x^{-1}\sim_{\cR} y$, $y^{-1} \sim_{\cR} z$, and $z^{-1}\sim_{\cR} x$.
\item[\bf P9.] If $z'\leq_{\cL} z$ and $\ba(z')=\ba(z)$, then $z'\sim_{\cL}z$.
\item[\bf P10.] If $z'\leq_{\cR} z$ and $\ba(z')=\ba(z)$, then $z'\sim_{\cR}z$.
\item[\bf P11.] If $z'\leq_{\cLR} z$ and $\ba(z')=\ba(z)$, then
$z'\sim_{\cLR}z$.
\item[\bf P12.] If $I\subseteq S$ then the $\ba$-function of~$W_I$ is the restriction to $W_I$ of the $\ba$-function of $W$.
\item[\bf P13.] Each right cell $\Up$ of $W$ contains a unique element
$d\in \cD$, and we have $\gamma_{x,x^{-1},d}\neq 0$ for all $x\in \Up$.
\item[\bf P14.] For each $z\in W$ we have $z \sim_{\cLR} z^{-1}$.
\item[\bf P15.] If $x,x',y,w\in W$ are such that $\ba(w)=\ba(y)$ then 
$$\sum_{y'\in W} h_{w,x',y'}\otimes h_{x,y',y}=\sum_{y'\in W} h_{y',x',y}\otimes h_{x,w,y'} \text{ in } \sR\otimes_{\nZ} \sR.$$
\end{itemize}

\subsection{Balanced system of cell representations}\label{sec:balanced}

In~\cite{GP:17} we introduced the notion of a \textit{balanced system of cell representations}, inspired by the work of Geck~\cite{Geck:02,Geck:11} in the finite case. We recall this theory here. 
\medskip

If $\mathsf{S}$ is an $\sR$-polynomial ring (including the possibility $\mathsf{S}=\sR$), we write $\mathsf{S}^{\leq 0}$ and $\mathsf{S}^0$ for the associated $\mathbb{Z}[\sq^{-1}]$-polynomial and $\mathbb{Z}$-polynomial subrings of $\mathsf{S}$, respectively. In particular $\sR^{\leq 0}=\mathbb{Z}[\sq^{-1}]$ and $\sR^0=\mathbb{Z}$. Let 
$$
\text{sp}_{|_{\sq^{-1}=0}}:\mathsf{S}^{\leq 0}\to \mathsf{S}^0\quad\text{denote the specialisation at $\sq^{-1}=0$}.
$$ 
\newpage

By a \textit{matrix representation} of $\cH$ we shall mean a triple $(\pi,\cM,\sB)$ where $\cM$ is a right $\cH$-module over an $\sR$-polynomial ring $\mathsf{S}$, and $\sB$ is a basis of~$\cM$. We write (for $h\in\cH$ and $u,v\in\sB$)
$$
\pi(h;\sB)\quad\text{and}\quad [\pi(h;\sB)]_{u,v}
$$
for the matrix of $\pi(h)$ with respect to the basis $\sB$, and the $(u,v)^{th}$ entry of $\pi(h;\sB)$. 
\medskip

Let $\deg(f(\sq))$ denote the degree of the Laurent polynomial $f(\sq)\in\mathsf{S}$ (note that degree here refers to degree in $\sq$, not degree in the indeterminates of the polynomial ring $\mathsf{S}$). 
A matrix representation $(\pi,\cM,\sB)$ is called \textit{bounded} if 
$
\deg([\pi(C_w;\sB)]_{u,v})
$ is bounded from above (for all $u,v\in \sB$ and all $w\in W$). In this case we call the integer
\begin{align}\label{eq:replace1}
\ba_\pi:=\max\{\deg([\pi(C_w;\sB)]_{u,v})\mid  u,v\in \sB,\,w\in W\}
\end{align}
the \textit{bound} of the matrix representation and we define  the \textit{leading matrices} by
\begin{align}\label{eq:replace2}
\displaystyle{\fc_{\pi}(w;\sB):=\text{sp}_{|_{\sq^{-1}=0}}\left(\sq^{-\ba_{\pi}}\pi(C_w;\sB)\right)}\quad\text{for $w\in W$}.
\end{align}

\begin{Def}\label{def:balanced}
We say that $\cH$ admits a \textit{balanced system of cell representations} if for each two-sided cell $\Gamma\in\tsc$ there exists a matrix representation $(\pi_{\Gamma},\cM_{\Gamma},\sB_{\Ga})$ defined over an $\sR$-polynomial ring $\sR_\Ga$ (where we could have $\sR_\Ga=\sR$) such that the following properties hold:
\begin{enumerate}
\item[\B{1}\textbf{.}] If $w\notin\Gamma_{\geq_{\mathcal{LR}}}$ then $\pi_{\Gamma}(C_w;\sB_{\Ga})=0$. 
\item[\B{2}\textbf{.}] The matrix representation $(\pi_{\Gamma},\cM_{\Gamma},\sB_{\Ga})$ is bounded. Let $\ba_{\pi_\Ga}$ denote the bound. 
\item[\B{3}\textbf{.}] We have $\fc_{\pi_{\Ga}}(w;\sB_{\Ga})\neq 0$ if and only if $w\in \Ga$. 
\item[\B{4}\textbf{.}] The leading matrices $\fc_{\pi_\Ga}(w;\sB_{\Ga})$ ($w\in \Ga$) are free over $\nZ$.
\item[\B{5}\textbf{.}] For each $z\in \Ga$ there exists $x,y\in \Ga$ such that $\tilde{\ga}_{x,y,z^{-1}}\neq 0$, where $\tilde{\gamma}_{x,y,z^{-1}}\in\nZ$ is the coefficient of $\sq^{\ba_{{\pi}_{\Ga}}}$ in $h_{x,y,z}$.
\item[\B{6}\textbf{.}] If $\Gamma'\leq_{\cLR}\Gamma$ then $\ba_{\pi_{\Ga'}}\geq \ba_{\pi_{\Gamma}}$. 
\end{enumerate}
The natural numbers $(\ba_{\pi_{\Gamma}})_{\Ga\in \tsc}$ are called the \textit{bounds} of the balanced system of cell representations. 
\end{Def}

\begin{Rem}\label{rem:checking} We make the following remarks:
\begin{enumerate}
\item We note that \B{1} does not depend on the basis $\sB_{\Ga}$. A representation with property \B{1} is called a \textit{cell representation} for the two-sided cell $\Gamma$. It is clear that the representations associated to cells that we introduced in Section \ref{sec:1.2} are cell representations (see \cite[Section~2.1]{GP:17}).
\item If the basis $\sB_{\Gamma}$ of $\cM_{\Gamma}$ is clear from context we will sometimes write $\fc_{\pi_{\Gamma}}(w)$ in place of $\fc_{\pi_{\Ga}}(w;\sB_{\Ga})$. 
\item By \cite[Corollary~2.4]{GP:17} the axioms $\B{1}$--$\B{4}$ and $\B{6}$ alone imply that the $\mathbb{Z}$-span $\mathcal{J}_{\Ga}$ of the matrices $\fc_{\pi_{\Ga}}(w;\sB_{\Ga})$ with $w\in \Ga$ is a $\nZ$-algebra, and that
$$
\fc_{\pi_\Ga}(x;\sB_{\Ga})\fc_{\pi_\Ga}(y;\sB_{\Ga})=\sum_{z\in \Ga} \tilde{\ga}_{x,y,z^{-1}}\fc_{\pi_\Ga}(z;\sB_{\Ga})\quad\text{for $x,y\in\Gamma$}
$$
with $\tilde{\gamma}_{x,y,z^{-1}}$ as defined in \B{5}. Hence these integers are the structure constants of the algebra $\mathcal{J}_{\Ga}$. 
\item We note that in~(\ref{eq:replace1}) and~(\ref{eq:replace2}) it is equivalent to replace $C_w$ by $T_w$, because $C_w=T_w+\sum_{v<w}P_{v,w}T_v$ with $P_{v,w}\in\sq^{-1}\mathbb{Z}[\sq^{-1}]$. However in \B{1} one cannot replace $C_w$ by $T_w$. 
\item Finally we note that we have slightly changed the numbering from \cite{GP:17}, where $\B{5}$ was denoted $\B{4}'$, and $\B{6}$ was denoted $\B{5}$. 
\end{enumerate}
\end{Rem}

In \cite{GP:17} we showed that the existence of a balanced system of cell representations is sufficient to compute Lusztig's $\ba$-function. In particular, we have:

\begin{Th}[{\cite[Theorem~2.5 and Corollary~2.6]{GP:17}}] \label{thm:afn}
Suppose that $\cH$ admits a balanced system of cell representations. Then $\ba(w)=\ba_{\pi_\Ga}$ for all $w\in \Ga$. Moreover, for each $\Ga\in\tsc$ the $\mathbb{Z}$-algebra $\mathcal{J}_{\Ga}$ spanned by the matrices $\{\fc_{\pi_{\Ga}}(w;\sB_{\Ga})\mid w\in\Ga\}$ is isomorphic to Lusztig's asymptotic algebra associated to~$\Ga$, and $\tilde{\gamma}_{x,y,z}=\gamma_{x,y,z}$. 
\end{Th}

Note that the first part of this theorem implies that the bounds $\ba_{\pi_{\Ga}}$ in Definition~\ref{def:balanced} are in fact unique. That is, if there exist two balanced systems of cell representations then their bounds coincide.


\section{Affine Weyl groups, affine Hecke algebras, and alcove paths}\label{sec:WeylGroups}

We begin this section with some basic facts about root systems and Weyl groups. We then recall the combinatorial language of alcove paths from~\cite{Ram:06}, and the concept of alcove paths confined to strips from~\cite{GP:17}. We also discuss the combinatorics of the affine Hecke algebra (and extended affine Hecke algebra) of type $\tilde{C}_2$.

\subsection{Root systems and Weyl groups}\label{sec:root}

Let $\Phi$ be the non-reduced root system of type $BC_2$ in the vector space $\mathbb{R}^2$. Thus $\Phi$ consists of vectors
$$
\Phi=\Phi^+\cup(-\Phi^+),\quad\text{where}\quad \Phi^+=\{\alpha_1,\alpha_2,\alpha_1+\alpha_2,\alpha_1+2\alpha_2,2\alpha_2,2(\alpha_1+\alpha_2)\},
$$
with $\|\alpha_1\|=\sqrt{2}$, $\|\alpha_2\|=1$, and $\langle \alpha_1,\alpha_2\rangle=-1$. Let $\Phi_0$ and $\Phi_1$ be the subsystems
$$
\Phi_0=\pm\{\alpha_1,\alpha_2,\alpha_1+\alpha_2,\alpha_1+2\alpha_2\}\quad\text{and}\quad 
\Phi_1=\pm\{\alpha_1,2\alpha_2,\alpha_1+2\alpha_2,2\alpha_1+2\alpha_2\}
$$
of types $B_2$ and $C_2$, respectively. 
\medskip

Let $\alpha^{\vee}=2\alpha/\langle\alpha,\alpha\rangle$. The dual root system is
$$
\Phi^{\vee}=\pm\{\alpha_1^{\vee},\alpha_2^{\vee}/2,\alpha_1^{\vee}+\alpha_2^{\vee}/2,\alpha_1^{\vee}+\alpha_2^{\vee},\alpha_2^{\vee},2\alpha_1^{\vee}+\alpha_2^{\vee}\}.
$$
The corrot lattice is the $\mathbb{Z}$-lattice $Q$ spanned by $\Phi^{\vee}$. Thus
$$
Q=\{m\alpha_1^{\vee}+n\alpha_2^{\vee}/2\mid m,n\in\mathbb{Z}\}.
$$
The fundamental coweights $\omega_1$ and $\omega_2$ are defined by $\langle\omega_i,\alpha_j\rangle=\delta_{i,j}$, and thus 
\begin{align*}
\omega_1=\alpha_1^{\vee}+\alpha_2^{\vee}/2\quad\text{and}\quad \omega_2=\alpha_1^{\vee}+\alpha_2^{\vee}.
\end{align*}
In particular, note that $\omega_1,\omega_2\in Q$. Let $Q^+$ be the cone $\mathbb{Z}_{\geq 0}\omega_1+\mathbb{Z}_{\geq 0}\omega_2$ (note that this notation is non-standard).
\medskip

For each $\alpha\in \Phi$ let $s_{\alpha}$ be the orthogonal reflection in the hyperplane $H_{\alpha}=\{x\in\mathbb{R}^2\mid \langle x,\alpha\rangle=0\}$ orthogonal to~$\alpha$, and for $i\in \{1,2\}$ let $s_i=s_{\alpha_i}$. The \textit{Weyl group} of $\Phi$ is the subgroup $W_0$ of $GL(\mathbb{R}^2)$ generated by the reflections $s_1$ and $s_2$ (this is a Coxeter group of type $B_2=C_2$). The Weyl group $W_0$ acts on $Q$ and the \textit{affine Weyl group} is
$
W=Q\rtimes W_0
$
where we identify $\lambda\in Q$ with the translation $t_{\lambda}(x)=x+\lambda$. The affine Weyl group is a Coxeter group with generating set $S=\{s_0,s_1,s_2\}$, where $s_0=t_{\varphi^{\vee}}s_{\varphi}$, with $\varphi=2\alpha_1+2\alpha_2$ the highest root of $\Phi$.  
\medskip

For each $\alpha\in\Phi$ and $k\in\mathbb{Z}$ let $H_{\alpha,k}=\{x\in \mathbb{R}^2\mid \langle x,\alpha\rangle=k\}$, and let $s_{\alpha,k}$ be the orthogonal reflection in the affine hyperplane~$H_{\alpha,k}$. Explicitly, $s_{\alpha,k}(x)=x-(\langle x,\alpha\rangle-k)\alpha^{\vee}$. Each affine hyperplane $H_{\alpha,k}$ with $\alpha\in \Phi^+$ and $k\in\mathbb{Z}$ divides $\mathbb{R}^2$ into two half spaces, denoted
$$
H_{\alpha,k}^+=\{x\in \mathbb{R}^2\mid \langle x,\alpha\rangle\geq k\}\quad\text{and}\quad H_{\alpha,k}^-=\{x\in \mathbb{R}^2\mid\langle x,\alpha\rangle\leq k\}.
$$
This ``orientation'' of the hyperplanes is called the \textit{periodic orientation} (see Figure~\ref{fig:rootsystem}).
\medskip

If $w\in W$ we define the \textit{linear part} $\theta(w)\in W_0$ and the \textit{translation weight} $\mathrm{wt}(w)\in Q$ by the equation 
$$
w=t_{\mathrm{wt}(w)}\theta(w).
$$
Let $\mathcal{F}$ denote the union of the hyperplanes $H_{\alpha,k}$ with $\alpha\in \Phi$ and $k\in \mathbb{Z}$. The closures of the open connected components of $\mathbb{R}^2\backslash\mathcal{F}$ are called \textit{alcoves} (these are the closed triangles in Figure~\ref{fig:rootsystem}). The \textit{fundamental alcove} is given by 
$$
A_0=\{x\in \mathbb{R}^2\mid 0\leq \langle x,\alpha\rangle\leq 1\text{ for all $\alpha\in\Phi^+$}\}.
$$
The hyperplanes bounding $A_0$ are called the \textit{walls} of $A_0$. Explicitly these walls are $H_{\alpha_i,0}$ with $i=1,2$ and $H_{\varphi,1}$. We say that a \textit{face} of $A_0$ (that is, a codimension~$1$ facet) has \textit{type} $s_i$ for $i=1,2$ if it lies on the wall $H_{\alpha_i,0}$ and of type $s_0$ if it lies on the wall $H_{\varphi,1}$. 

\medskip

The affine Weyl group $W$ acts simply transitively on the set of alcoves, and we use this action to identify the set of alcoves with $W$ via $w\leftrightarrow wA_0$. Moreover, we use the action of $W$ to transfer the notions of walls, faces, and types of faces to arbitrary alcoves. Alcoves $A$ and $A'$ are called \textit{$s$-adjacent}, written $A\sim_s A'$, if $A\neq A'$ and $A$ and $A'$ share a common type $s$ face. Thus under the identification of alcoves with elements of $W$, the alcoves $w$ and $ws$ are $s$-adjacent. 

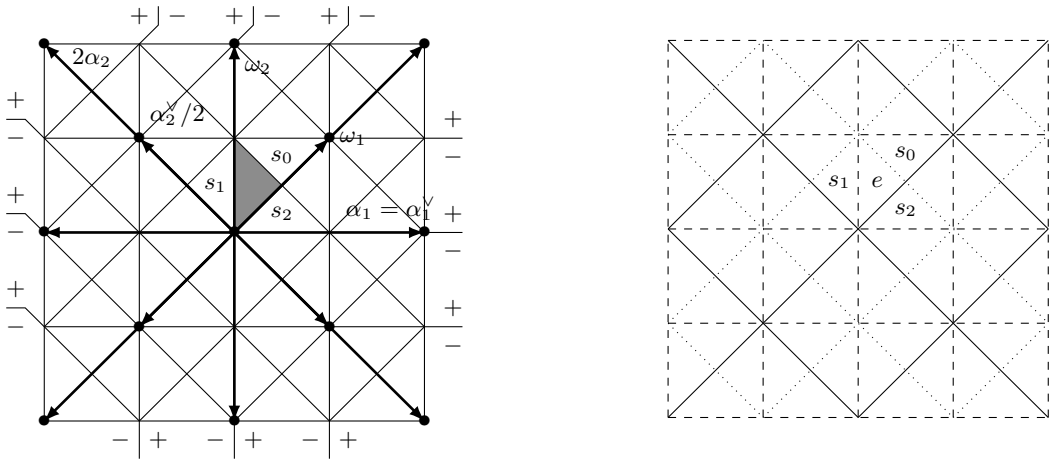
\begin{figure}[H]
\begin{center}
\begin{subfigure}{.4\textwidth}
\begin{center}
\begin{tikzpicture}[scale=1.25]
\path [fill=gray!90] (0,0) -- (0.5,0.5) -- (0,1) -- (0,0);
\draw (-2,-2)--(2,-2);
\draw (-2,-1)--(2.4,-1);
\draw (-2,0)--(2.4,0);
\draw (-2,1)--(2.4,1);
\draw (-2,2)--(2,2);
\draw (-2,-2)--(-2,2);
\draw (-1,-2.4)--(-1,2);
\draw (0,-2.4)--(0,2);
\draw (1,-2.4)--(1,2);
\draw (2,-2)--(2,2);
\draw (-2,1)--(-0.8,2.2)--(-0.8,2.4);
\draw (-2,0)--(0.2,2.2)--(0.2,2.4);
\draw (-2,-1)--(1.2,2.2)--(1.2,2.4);
\draw (-2,-2)--(2,2);
\draw (-1,-2)--(2,1);
\draw (0,-2)--(2,0);
\draw (1,-2)--(2,-1);
\draw (-2.4,-0.8)--(-2.2,-0.8)--(-1,-2);
\draw (-2.4,0.2)--(-2.2,0.2)--(0,-2);
\draw (-2.4,1.2)--(-2.2,1.2)--(1,-2);
\draw (-2,2)--(2,-2);
\draw (-1,2)--(2,-1);
\draw (0,2)--(2,0);
\draw (1,2)--(2,1);
\node at (1.65,0.25) {\small{$\alpha_1=\alpha_1^{\vee}$}};
\node at (0.25,1.75) {\small{$\omega_2$}};
\node at (-0.6,1.25) {\small{$\alpha_2^{\vee}/2$}};
\node at (-1.5,1.85) {\small{$2\alpha_2$}};
\node at (1.25,1) {\small{$\omega_1$}};
\draw [latex-latex,line width=1pt] (-2,0)--(2,0);
\draw [latex-latex,line width=1pt] (0,-2)--(0,2);
\draw [latex-latex,line width=1pt] (-1,-1)--(1,1);
\draw [latex-latex,line width=1pt] (-1,1)--(1,-1);
\draw [latex-latex,line width=1pt] (-2,-2)--(2,2);
\draw [latex-latex,line width=1pt] (-2,2)--(2,-2);
\node at (-2,-2) {$\bullet$};
\node at (-2,0) {$\bullet$};
\node at (-2,2) {$\bullet$};
\node at (-1,-1) {$\bullet$};
\node at (-1,1) {$\bullet$};
\node at (0,-2) {$\bullet$};
\node at (0,0) {$\bullet$};
\node at (0,2) {$\bullet$};
\node at (2,-2) {$\bullet$};
\node at (2,0) {$\bullet$};
\node at (2,2) {$\bullet$};
\node at (1,-1) {$\bullet$};
\node at (1,1) {$\bullet$};
\node at (-0.2,0.5) {\small{$s_1$}};
\node at (0.5,0.2) {\small{$s_2$}};
\node at (0.5,0.8) {\small{$s_0$}};
\node at (2.3,0.2) {\small{$+$}};
\node at (2.3,-0.2) {\small{$-$}};
\node at (2.3,1.2) {\small{$+$}};
\node at (2.3,0.8) {\small{$-$}};
\node at (2.3,-0.8) {\small{$+$}};
\node at (2.3,-1.2) {\small{$-$}};
\node at (-0.2,-2.2) {\small{$-$}};
\node at (0.2,-2.2) {\small{$+$}};
\node at (-1.2,-2.2) {\small{$-$}};
\node at (-0.8,-2.2) {\small{$+$}};
\node at (0.8,-2.2) {\small{$-$}};
\node at (1.2,-2.2) {\small{$+$}};
\node at (-2.3,0.4) {\small{$+$}};
\node at (-2.3,0) {\small{$-$}};
\node at (-2.3,1.4) {\small{$+$}};
\node at (-2.3,1) {\small{$-$}};
\node at (-2.3,-0.6) {\small{$+$}};
\node at (-2.3,-1) {\small{$-$}};
\node at (-0.6,2.3) {\small{$-$}};
\node at (-1,2.3) {\small{$+$}};
\node at (0.4,2.3) {\small{$-$}};
\node at (0,2.3) {\small{$+$}};
\node at (1.4,2.3) {\small{$-$}};
\node at (1,2.3) {\small{$+$}};
\end{tikzpicture}
\end{center}
\end{subfigure}\qquad\qquad
\begin{subfigure}{.4\textwidth}
\begin{center}
\begin{tikzpicture}[scale=1.25]
\draw[style=dashed] (-2,-2)--(2,-2);
\draw[style=dashed] (-2,-1)--(2,-1);
\draw[style=dashed] (-2,0)--(2,0);
\draw[style=dashed] (-2,1)--(2,1);
\draw[style=dashed] (-2,2)--(2,2);
\draw[style=dashed] (-2,-2)--(-2,2);
\draw[style=dashed] (-1,-2)--(-1,2);
\draw[style=dashed] (0,-2)--(0,2);
\draw[style=dashed] (1,-2)--(1,2);
\draw[style=dashed] (2,-2)--(2,2);
\draw[style=dotted] (-2,1)--(-1,2);
\draw (-2,0)--(0,2);
\draw[style=dotted] (-2,-1)--(1,2);
\draw (-2,-2)--(2,2);
\draw[style=dotted] (-1,-2)--(2,1);
\draw (0,-2)--(2,0);
\draw[style=dotted] (1,-2)--(2,-1);
\draw[style=dotted] (-2,-1)--(-1,-2);
\draw (-2,0)--(0,-2);
\draw[style=dotted] (-2,1)--(1,-2);
\draw (-2,2)--(2,-2);
\draw[style=dotted] (-1,2)--(2,-1);
\draw (0,2)--(2,0);
\draw[style=dotted] (1,2)--(2,1);
\node at (0.2,0.5) {\small{$e$}};
\node at (-0.2,0.5) {\small{$s_1$}};
\node at (0.5,0.2) {\small{$s_2$}};
\node at (0.5,0.8) {\small{$s_0$}};
\end{tikzpicture}
\end{center}
\end{subfigure}
\caption{Root system of type $BC_2$, periodic orientation, and adjacency types (dotted, dashed, solid $=$ 0,1,2)}
\label{fig:rootsystem}
\end{center}
\end{figure}

\subsection{Alcove paths}

For any sequence $\vec{w}=(s_{i_{1}},s_{i_2},\ldots,s_{i_{\ell}})$ of elements of $S$ we have
$$
e\sim_{s_{i_1}}s_{i_1}\sim_{s_{i_2}} s_{i_1}s_{i_2}\sim_{s_{i_3}}\cdots\sim_{s_{i_{\ell}}}s_{i_{1}}s_{i_2}\cdots s_{i_\ell}.
$$
In this way, sequences $\vec{w}$ of elements of $S$ determine \textit{alcove paths} (also called \textit{alcove walks}) of \textit{type} $\vec{w}$ starting at the fundamental alcove~$e=A_0$. We will typically abuse notation and refer to alcove paths of type $\vec{w}=s_{i_1}s_{i_2}\cdots s_{i_{\ell}}$ rather than $\vec{w}=(s_{i_{1}},s_{i_2},\ldots,s_{i_{\ell}})$. Thus ``the alcove path of type $\vec{w}=s_{i_1}s_{i_2}\cdots s_{i_{\ell}}$'' is the sequence $(v_0,v_1,\ldots,v_{\ell})$ of alcoves, where $v_0=e$ and $v_k=s_{i_1}\cdots s_{i_k}$ for $k=1,\ldots,\ell$. 
\medskip

Let $\vec w=s_{i_1}s_{i_2}\cdots s_{i_{\ell}}$ be an expression for $w\in W$, and let~$v\in W$. A \textit{positively folded alcove path of type~$\vec w$ starting at $v$} is a sequence $p=(v_0,v_1,\ldots,v_{\ell})$ with $v_0,\ldots,v_{\ell}\in W$ such that
\begin{enumerate}
\item $v_0=v$,
\item $v_k\in\{v_{k-1},v_{k-1}s_{i_k}\}$ for each $k=1,\ldots,\ell$, and
\item if $v_{k-1}=v_k$ then $v_{k-1}$ is on the positive side of the hyperplane separating $v_{k-1}$ and $v_{k-1}s_{i_k}$. 
\end{enumerate}
The \textit{end} of $p$ is $\mathrm{end}(p)=v_{\ell}$. Let $\mathrm{wt}(p)=\mathrm{wt}(\mathrm{end}(p))$ and $\theta(p)=\theta(\mathrm{end}(p))$. Let 
$$
\mathcal{P}(\vec{w},v)=\{\text{all positively folded alcove paths of type $\vec{w}$ starting at $v$}\}.
$$

Less formally, a \textit{positively folded alcove path of type~$\vec w$ starting at $v$} is a sequence of steps from alcove to alcove in $W$, starting at $v$, and made up of the symbols (where the $k$th step has $s=s_{i_k}$ for $k=1,\ldots,\ell$):

\begin{center}
\begin{pspicture}(-6,-1)(6,1)
\psset{unit=.5cm}
\psline(-8.5,-1)(-8.5,1)
\psline{->}(-9,0)(-8,0)
\rput(-9.5,1){{ $-$}}
\rput(-9.5,.2){{ $x$}}
\rput(-7.5,.2){{$xs$}}
\rput(-7.9,1){{ $+$}}
\rput(-8.5,-1.5){{ \text{(positive $s$-crossing)}}}

%
%
\psline(-2,-1)(-2,1)
\psline(-1.5,0)(-2,0)
\psline{<-}(-1.5,-.15)(-2,-.15)
\rput(-3,1){{ $-$}}
\rput(-3,.2){{ $xs$}}
\rput(-1,.2){{ $x$}}
\rput(-1.4,1){{ $+$}}
\rput(-2,-1.5){{ \text{(positive $s$-fold)}}}

\psline(4.5,-1)(4.5,1)
\psline{->}(5,0)(4,0)
\rput(5,1){{ $+$}}
\rput(5.5,.2){{$x$}}
\rput(3.5,.2){{ $xs$}}
\rput(3.6,1){{ $-$}}
\rput(4.5,-1.5){{ \text{(negative $s$-crossing)}}}
\end{pspicture}
\end{center}

If $p$ has no folds we say that $p$ is \textit{straight}. Note that, by definition, there are no ``negative'' folds. 

\medskip

If $p$ is a positively folded alcove path we define, for each $s_j\in S$, 
\begin{align*}
f_j(p)&=\#\textrm{(positive $s_j$-folds in $p$)}.
\end{align*}

\subsection{Alcove paths confined to strips}\label{sec:stripwalks}

Let $\alpha_1'=\alpha_1$ and $\alpha_2'=2\alpha_2$ (these are the simple roots of $\Phi_1$). For $i\in\{1,2\}$ let
$$
\mathcal{U}_i=\{x\in \mathbb{R}^2\mid 0\leq \langle x,\alpha_i'\rangle\leq 1\}
$$
be the region between the hyperplanes $H_{\alpha_i',0}$ and $H_{\alpha_i',1}$. It is also convenient to define $\cU_3=\cU_2$. 
\medskip

Let $\vec{w}=s_{i_1}\cdots s_{i_{\ell}}$ be an expression for $w\in W$. Let $i\in\{1,2,3\}$. An \textit{$i$-folded alcove path of type $\vec{w}$ starting at $v\in \mathcal{U}_i$} is a sequence $p=(v_0,v_1,\ldots,v_{\ell})$ with $v_0,\ldots,v_{\ell}\in\mathcal{U}_i$ such that
\begin{enumerate}
\item $v_0=v$, and $v_k\in \{v_{k-1},v_{k-1}s_{i_k}\}$ for each $k=1,\ldots,\ell$, and
\item if $v_{k-1}=v_k$ then either:
\begin{enumerate}
\item $v_{k-1}s_{i_k}\notin \mathcal{U}_i$, or
\item $v_{k-1}$ is on the positive side of the hyperplane separating $v_{k-1}$ and $v_{k-1}s_{i_k}$. 
\end{enumerate}
\end{enumerate}

We note that condition 2)(a) can only occur if $v_{k-1}$ and $v_{k-1}s_{i_k}$ are separated by either $H_{\alpha_i',0}$ or $H_{\alpha_i',1}$. 
\medskip

The \textit{end} of the $i$-folded alcove path $p=(v_0,\ldots,v_{\ell})$ is $\mathrm{end}(p)=v_{\ell}$.  Let 
$$
\mathcal{P}_i(\vec{w},v)=\{\text{all $i$-folded alcove paths of type $\vec{w}$ starting at $v$}\}.
$$

Less formally, $i$-folded alcove paths are made up of the following symbols, where $x\in \mathcal{U}_i$ and $s\in S$:

\begin{figure}[H]
\begin{subfigure}{.6\textwidth}
\begin{center}
\begin{pspicture}(-5,-1.5)(6,1)
\psset{unit=.5cm}
\psline(-6,-1)(-6,1)
\psline{->}(-6.5,0)(-5.5,0)
\rput(-6.8,1){{ $-$}}
\rput(-7,.2){{ $x$}}
\rput(-5,.2){{ $xs$}}
\rput(-5.4,1){{ $+$}}
\rput(-6,-1.5){\footnotesize{(positive $s$-crossing)}}
\psline(-0,-1)(-0,1)
\psline(.5,0)(-0,0)
\psline{<-}(.5,-.15)(-0,-.15)
\rput(-.8,1){{ $-$}}
\rput(-1,.2){{ $xs$}}
\rput(1,.2){{ $x$}}
\rput(.6,1){{ $+$}}
\rput(0,-1.5){\footnotesize{($s$-fold)}}

\psline(6,-1)(6,1)
\psline{->}(6.5,0)(5.5,0)
\rput(6.5,1){{ $+$}}
\rput(7,.2){{ $x$}}
\rput(5,.2){{ $xs$}}
\rput(5.2,1){{ $-$}}
\rput(6,-1.5){\footnotesize{(negative $s$-crossing)}}
\end{pspicture}
\caption{When the alcoves $x$ and $xs$ both belong to $\mathcal{U}_i$}
\end{center}
\end{subfigure}
\begin{subfigure}{0.4\textwidth}
\begin{center}
\begin{pspicture}(-3,-1.5)(3,1)
\psset{unit=.5cm}
\psline(3,-1)(3,1)
\psline(2.5,0)(3,0)
\psline{<-}(2.5,-.15)(3,-.15)
\rput(3.5,1){{ $+$}}
\rput(4,.2){{ $xs$}}
\rput(2,.2){{ $x$}}
\rput(2.2,1){{ $-$}}
\rput(3,-1.5){\footnotesize{($s$-bounce)}}
\psline(-3,-1)(-3,1)
\psline(-2.5,0)(-3,0)
\psline{<-}(-2.5,-.15)(-3,-.15)
\rput(-3.8,1){{ $-$}}
\rput(-4,.2){{ $xs$}}
\rput(-2,.2){{ $x$}}
\rput(-2.4,1){{ $+$}}
\rput(-3,-1.5){\footnotesize{($s$-bounce)}}
\end{pspicture}
\caption{When $xs$ lies outside of $\cU_{i}$}
\end{center}
\end{subfigure}
\end{figure}

We refer to the two symbols in (b) as ``$s$-bounces'' rather than folds, since they play a different role in the theory. It turns out that there is no need to distinguish between ``positive'' and ``negative'' $s$-bounces. We note that bounces only occur on the hyperplanes $H_{\alpha_i',0}$ and $H_{\alpha_i',1}$. Moreover, note that there are no folds or crossings on the walls $H_{\alpha_i',0}$ and $H_{\alpha_i',1}$ -- the only interactions with these walls are bounces. In the case $i=1$ every bounce has type $1$. In the case $i=2,3$ the bounces on $H_{\alpha_2',0}$ have type $2$, and those on $H_{\alpha_2',1}$ have type $0$ (see Figures~\ref{fig:rootsystem} and~\ref{fig:paths}).
\medskip

Let $p$ be an $i$-folded alcove path. For each $j\in\{0,1,2\}$ let
$$
f_j(p)=\#(\text{$s_j$-folds in $p$})\quad\text{and}\quad g_j(p)=\#(\text{$s_j$-bounces in $p$}).
$$
For $i\in\{1,2\}$ let $W_i=\langle s_i\rangle$ and let $W_0^i$ denote the set of minimal length coset representatives for cosets in $W_{i}\backslash W_0$. Define 
$$\theta^i(p)=\psi_i(\theta(p))\quad\text{and}\quad \mathrm{wt}^i(p)=\langle\mathrm{wt}(p),\omega_i\rangle,$$ 
where $\psi_i:W_0\to W^i_0$ is the natural projection map taking $u\in W_0$ to the minimal length representative of $W_{i}u$, and $\omega_1,\omega_2$ are as defined in Section~\ref{sec:root}. For later use, we also set
$$
\theta^3=\theta^2\quad\text{and}\quad \mathrm{wt}^3=\mathrm{wt}^2.
$$
Thus if $\mathrm{wt}(p)=m\alpha_1^{\vee}+n\alpha_2^{\vee}/2$ then $\mathrm{wt}^1(p)=m$ and $\mathrm{wt}^2(p)=\mathrm{wt}^3(p)=n$.
\medskip

 Let
\begin{align*}
\tau_1=t_{\omega_1}s_1=s_0s_1s_2\quad\text{and}\quad 
\tau_2=t_{\omega_1}=s_0s_1s_2s_1,
\end{align*}
and let $\tau_3=\tau_2$. Observe that $\tau_i$ preserves $\cU_i$. It is not hard to see that for each $p\in \mathcal{P}_i(\vec{w},u)$ the path $\tau_i(p)$ obtained by applying $\tau_i$ to each part of $p$ is again a valid $i$-folded alcove path starting at $\tau_iu$ (the main point here is that in the case $i=1$ the reflection part of $\tau_1$ is in the simple root direction $\alpha_1$, and thus sends $\Phi^+\backslash\{\alpha_1\}$ to itself; in the cases $i=2,3$ the element $\tau_2=\tau_3$ is a pure translation, and so the result is obvious). Moreover $\theta^i(p)$ is preserved under the application of $\tau_i$, and a direct calculation shows that $\mathrm{wt}^i(\tau_i^{k}(p))=k+\mathrm{wt}^i(p)$. 
\medskip

Note that $W_0^i$ is a fundamental domain for the action of $\langle \tau_i\rangle$ on $\cU_i$. Let $\sB$ be any other fundamental domain for this action. For $w\in \cU_i$ we define $\mathrm{wt}^i_{\sB}(w)\in\mathbb{Z}$ and $\theta^i_{\sB}(w)\in \sB$ by the equation
$$
w=\tau_i^{\mathrm{wt}^i_{\sB}(w)}\theta^i_{\sB}(w),
$$
and for $i$-folded alcove paths $p$ we define 
$$
\mathrm{wt}_{\sB}^i(p)=\mathrm{wt}_{\sB}^i(\mathrm{end}(p))\quad\text{and}\quad \theta^i_{\sB}(p)=\theta_{\sB}^i(\mathrm{end}(p)).
$$ 
It is easy to see that in the case $\sB=W_0^i$ these definitions agree with those for $\mathrm{wt}^i(p)$ and $\theta^i(p)$ made above.
\medskip

\begin{Exa}
Figure~\ref{fig:paths} shows three examples of $i$-folded alcove paths, with $i=1$ in the first two cases, and $i=2$ or $i=3$ in the third case. In each case the identity alcove is shaded in dark green. The first and second paths have type $\vec{w}=210121012120$ and start at $u=e$, and the third path has type $\vec{w}=121021210120120$ and starts at $u=12$. 

\setcounter{figure}{1}
\psset{linewidth=.1mm,unit=1cm}
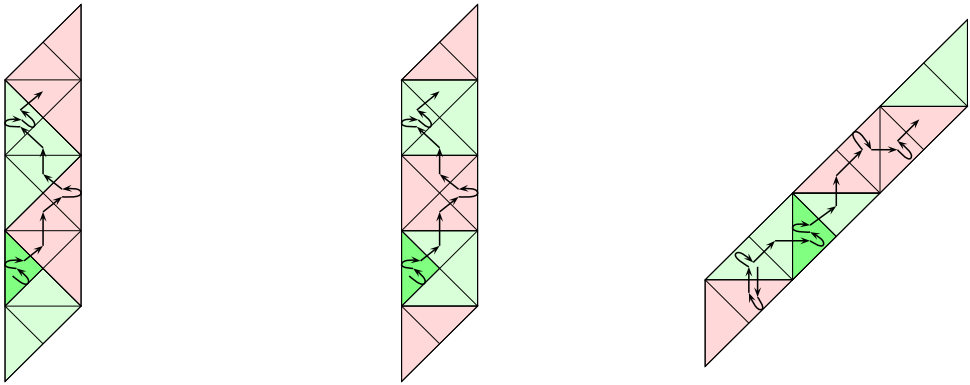
\begin{figure}[H]
\begin{subfigure}{.3\textwidth}
\begin{center}
\begin{pspicture}(0,0)(1,5)
\pspolygon[fillcolor=green!15!,fillstyle=solid](0,0)(1,1)(0,2)
\pspolygon[fillcolor=green!50!,fillstyle=solid](0,1)(0.5,1.5)(0,2)
\pspolygon[fillcolor=red!15!,fillstyle=solid](1,1)(0,2)(1,3)
\pspolygon[fillcolor=green!15!,fillstyle=solid](0,2)(1,3)(0,4)
\pspolygon[fillcolor=red!15!,fillstyle=solid](1,3)(0,4)(1,5)
\psline(0,0)(0,4)
\psline(1,1)(1,5)
\psline(0,1)(1,1)
\psline(0,2)(1,2)
\psline(0,3)(1,3)
\psline(0,4)(1,4)
\psline(0,4)(1,5)
\psline(0,3)(1,4)
\psline(0,2)(1,3)
\psline(0,1)(1,2)
\psline(0,0)(1,1)
\psline(0.5,4.5)(1,4)
\psline(0,4)(1,3)
\psline(0,3)(1,2)
\psline(0,2)(1,1)
\psline(0,1)(0.5,0.5)
\pscurve[linewidth=0.2mm]{->}(0.1,1.4)(0.3,1.3)(0.15,1.5)
\pscurve[linewidth=0.2mm]{->}(0.15,1.5)(0,1.55)(0.25,1.6)
\psline[linewidth=0.2mm]{->}(0.25,1.6)(0.5,1.8)
\psline[linewidth=0.2mm]{->}(0.5,1.8)(0.5,2.25)
\psline[linewidth=0.2mm]{->}(0.5,2.25)(0.75,2.45)
\pscurve[linewidth=0.2mm]{->}(0.75,2.45)(1,2.5)(0.75,2.55)
\psline[linewidth=0.2mm]{->}(0.75,2.55)(0.5,2.75)
\psline[linewidth=0.2mm]{->}(0.5,2.75)(0.5,3.1)
\psline[linewidth=0.2mm]{->}(0.5,3.1)(0.2,3.38)
\pscurve[linewidth=0.2mm]{->}(0.2,3.38)(0,3.425)(0.22,3.47)
\pscurve[linewidth=0.2mm]{->}(0.22,3.47)(0.38,3.38)(0.2,3.6)
\psline[linewidth=0.2mm]{->}(0.2,3.6)(0.5,3.85)
\end{pspicture}
\end{center}
\end{subfigure}
\begin{subfigure}{.3\textwidth}
\begin{center}
\begin{pspicture}(0,0)(1,5)
\pspolygon[fillcolor=green!15!,fillstyle=solid](0,1)(1,1)(1,2)(0,2)
\pspolygon[fillcolor=green!50!,fillstyle=solid](0,1)(0.5,1.5)(0,2)
\pspolygon[fillcolor=red!15!,fillstyle=solid](0,2)(1,2)(1,3)(0,3)
\pspolygon[fillcolor=green!15!,fillstyle=solid](0,3)(1,3)(1,4)(0,4)
\pspolygon[fillcolor=red!15!,fillstyle=solid](0,4)(1,4)(1,5)
\pspolygon[fillcolor=red!15!,fillstyle=solid](0,0)(1,1)(0,1)
\psline(0,0)(0,4)
\psline(1,1)(1,5)
\psline(0,1)(1,1)
\psline(0,2)(1,2)
\psline(0,3)(1,3)
\psline(0,4)(1,4)
\psline(0,4)(1,5)
\psline(0,3)(1,4)
\psline(0,2)(1,3)
\psline(0,1)(1,2)
\psline(0,0)(1,1)
\psline(0.5,4.5)(1,4)
\psline(0,4)(1,3)
\psline(0,3)(1,2)
\psline(0,2)(1,1)
\psline(0,1)(0.5,0.5)
\pscurve[linewidth=0.2mm]{->}(0.1,1.4)(0.3,1.3)(0.15,1.5)
\pscurve[linewidth=0.2mm]{->}(0.15,1.5)(0,1.55)(0.25,1.6)
\psline[linewidth=0.2mm]{->}(0.25,1.6)(0.5,1.8)
\psline[linewidth=0.2mm]{->}(0.5,1.8)(0.5,2.25)
\psline[linewidth=0.2mm]{->}(0.5,2.25)(0.75,2.45)
\pscurve[linewidth=0.2mm]{->}(0.75,2.45)(1,2.5)(0.75,2.55)
\psline[linewidth=0.2mm]{->}(0.75,2.55)(0.5,2.75)
\psline[linewidth=0.2mm]{->}(0.5,2.75)(0.5,3.1)
\psline[linewidth=0.2mm]{->}(0.5,3.1)(0.2,3.38)
\pscurve[linewidth=0.2mm]{->}(0.2,3.38)(0,3.425)(0.22,3.47)
\pscurve[linewidth=0.2mm]{->}(0.22,3.47)(0.38,3.38)(0.2,3.6)
\psline[linewidth=0.2mm]{->}(0.2,3.6)(0.5,3.85)
\end{pspicture}
\end{center}
\end{subfigure}
\psset{linewidth=.1mm,unit=1.15cm}
\begin{subfigure}{.3\textwidth}
\begin{center}
\begin{pspicture}(0,-0.5)(3,4.5)
\pspolygon[fillcolor=green!15!,fillstyle=solid](0,1)(1,1)(2,2)(1,2)
\pspolygon[fillcolor=green!50!,fillstyle=solid](1,1)(1.5,1.5)(1,2)
\pspolygon[fillcolor=red!15!,fillstyle=solid](1,2)(2,2)(3,3)(2,3)
\pspolygon[fillcolor=green!15!,fillstyle=solid](2,3)(3,3)(3,4)
\pspolygon[fillcolor=red!15!,fillstyle=solid](0,0)(0,1)(1,1)
\psline(0,0)(3,3)
\psline(3,3)(3,4)
\psline(3,4)(0,1)
\psline(0,1)(0,0)
\psline(0,1)(1,1)
\psline(1,2)(2,2)
\psline(2,3)(3,3)
\psline(0,1)(0.5,0.5)
\psline(0.5,1.5)(1,1)
\psline(1,2)(1.5,1.5)
\psline(1.5,2.5)(2,2)
\psline(2,3)(2.5,2.5)
\psline(2.5,3.5)(3,3)
\psline(1,2)(2,2)
\psline(1,1)(1,2)
\psline(2,2)(2,3)
\psline[linewidth=0.2mm]{->}(0.6,1.15)(0.6,0.8)
\pscurve[linewidth=0.2mm]{->}(0.6,0.8)(0.65,0.66)(0.5,0.85)
\psline[linewidth=0.2mm]{->}(0.5,0.85)(0.5,1.15)
\pscurve[linewidth=0.2mm]{->}(0.5,1.15)(0.35,1.3)(0.55,1.2)
\psline[linewidth=0.2mm]{->}(0.55,1.2)(0.8,1.45)
\psline[linewidth=0.2mm]{->}(0.8,1.45)(1.2,1.45)
\pscurve[linewidth=0.2mm]{->}(1.2,1.45)(1.35,1.4)(1.2,1.55)
\pscurve[linewidth=0.2mm]{->}(1.2,1.55)(1,1.6)(1.2,1.63)
\psline[linewidth=0.2mm]{->}(1.2,1.63)(1.5,1.85)
\psline[linewidth=0.2mm]{->}(1.5,1.85)(1.5,2.2)
\psline[linewidth=0.2mm]{->}(1.5,2.2)(1.8,2.5)
\pscurve[linewidth=0.2mm]{->}(1.8,2.5)(1.7,2.7)(1.9,2.5)
\psline[linewidth=0.2mm]{->}(1.9,2.5)(2.2,2.5)
\pscurve[linewidth=0.2mm]{->}(2.2,2.5)(2.35,2.4)(2.2,2.6)
\psline[linewidth=0.2mm]{->}(2.2,2.6)(2.45,2.85)
\end{pspicture}
\end{center}
\end{subfigure}
\caption{$i$-folded alcove paths}
\label{fig:paths}
\end{figure}
The first and second figures illustrate two choices of fundamental domain $\sB$ for the action of $\tau_1$ on $\cU_1$ (indicated by green and red shading). In the first example $\sB=W_0^1$, and we have $\mathrm{wt}_{\sB}^1(p)=3$ and $\theta_{\sB}^1(p)=21$. In the second example $\sB=\{e,2,0,20\}$, and we have $\mathrm{wt}_{\sB}^1(p)=2$ and $\theta_{\sB}^1(p)=0$. The third figure illustrates the fundamental domain $\sB=\{12,2,e,0\}$ for the action of $\tau_2=\tau_3$ on $\cU_2=\cU_3$. We have $\mathrm{wt}_{\sB}^2(p)=\mathrm{wt}_{\sB}^3(p)=1$ and $\theta_{\sB}^2(p)=\theta_{\sB}^3(p)=0$. 
\end{Exa}

\subsection{The affine Hecke algebra of type $\tilde{C}_2$}\label{sec:heckeC2}

Let $(W,S)$ be the $\tilde{C}_2$ Coxeter system and let $\cH_g$ be the associated generic affine Hecke algebra, as in~(\ref{eq:generic}). The algebra $\cH_g$ is generated by $T_0=T_{s_0}$, $T_1=T_{s_1}$ and $T_2=T_{s_2}$ subject to the relations (for $i=0,1,2$)
\begin{align*}
T_i^2=1+\sQ_iT_i,\quad T_0T_1T_0T_1=T_1T_0T_1T_0,\quad T_1T_2T_1T_2=T_2T_1T_2T_1,\quad\text{and}\quad T_0T_2=T_2T_0,
\end{align*}
where $\sq_i=\sq_{s_i}$ and $\sQ_i=\sq_i-\sq_i^{-1}$. 
\medskip

Let $v\in W$ and choose any expression $v=s_{i_1}\cdots s_{i_{\ell}}$ (not necessarily reduced). Consider the associated straight alcove path $(v_0,v_1\ldots,v_{\ell})$, where $v_0=e$ and $v_k=s_{i_1}\cdots s_{i_k}$. Let $\eps_1,\ldots,\eps_{\ell}$ be defined using the periodic orientation on hyperplanes as follows:
$$
\eps_k=\begin{cases}
+1&\text{if $v_{k-1}\,\,{^-}\hspace{-0.1cm}\mid^+\,\, v_k$\hspace{0.37cm} (that is, a positive crossing)}\\
-1&\text{if $v_{k}\,\,\hspace{0.34cm}{^-}\hspace{-0.1cm}\mid^+\,\,  v_{k-1}$ (that is, a negative crossing).} 
\end{cases}
$$
It is easy to check (using Tits' solution to the Word Problem) that the element
$$X_{v}=T^{\eps_{1}}_{s_{i_{1}}}\ldots T^{\eps_\ell}_{s_{i_{\ell}}}\in\cH_g$$
does not depend on the particular expression $v=s_{i_{1}}\cdots s_{i_\ell}$ we have chosen (see \cite{Goe:07}). If $\lambda\in Q$ we write 
$$
X^{\lambda}=X_{t_{\lambda}},
$$
and it follows from the above definitions that
\begin{align}\label{eq:split}
X_v=X_{t_{\mathrm{wt}(v)}\theta(v)}=X^{\mathrm{wt}(v)}X_{\theta(v)}=X^{\mathrm{wt}(v)}T_{\theta(v)^{-1}}^{-1}
\end{align}
(the second equality follows since $t_{\mathrm{wt}(v)}$ is on the positive side of every hyperplane through $\mathrm{wt}(v)$, and the third equality follows since $X_u=T_{u^{-1}}^{-1}$ for all $u\in W_0$). Moreover since $X_v=T_v+\text{(lower terms)}$ the set $\{X_v\mid v\in W\}$ is a basis of $\cH_g$, called the \textit{Bernstein-Lusztig basis}. 

\medskip

Let $\sR_g[Q]$ be the free $\sR_g$-module with basis $\{X^\la\mid \la\in Q\}$. We have a natural action of $W_0$ on $\sR_g[Q]$ given by $w X^\la =X^{w\la}$. We set 
$$\sR_g[Q]^{W_0}=\{p\in \sR_g[Q]\mid w\cdot p=p \text{ for all $w\in W_0$}\}.$$
It is a well-known result that the centre of $\cH_g$ is $\cZ(\cH_g)=\sR_g[Q]^{W_0}$.

\medskip

The combinatorics of positively folded alcove paths encode the change of basis from the standard basis $(T_w)_{w\in W}$ of $\cH_g$ to the Bernstein-Lusztig basis $(X_v)_{v\in W}$. This is seen by taking $u=e$ in the following proposition (see \cite[Theorem~3.3]{Ram:06}, or \cite[Proposition~3.2]{GP:17}).

\begin{Prop}\label{prop:basischange}(c.f. \cite[Theorem~3.3]{Ram:06}) Let $w,u\in W$, and let $\vec{w}$ be any reduced expression for $w$. Then
\begin{align*}
X_uT_w&=\sum_{p\in\mathcal{P}(\vec{w},u)}\mathcal{Q}(p)X_{\mathrm{end}(p)}\quad\text{where}\quad \cQ(p)=\prod_{j=0}^2(\sq_j-\sq_j^{-1})^{f_j(p)}.
\end{align*}
\end{Prop}

Let 
$$
X_1=X^{\alpha_1^{\vee}}\quad\text{and}\quad X_2=X^{\alpha_2^{\vee}/2}.
$$
We have $X_1=T_2^{-1}T_0T_1T_0^{-1}T_2T_1$ and $X_2=T_1^{-1}T_0T_1T_2$. Note that $X^{\omega_1}=X_1X_2$ and $X^{\omega_2}=X_1X_2^2$. 

\medskip

The \textit{Bernstein relations} are (for $\lambda\in Q$)
\begin{align*}
T_1^{-1}X^{\lambda}-X^{s_1\lambda}T_1^{-1}=\sQ_1\frac{X^{\lambda}-X^{s_1\lambda}}{X_1-1}\quad\text{and}\quad T_2^{-1}X^{\lambda}-X^{s_2\lambda}T_2^{-1}=(\sQ_2+\sQ_0X_2)\frac{X^{\lambda}-X^{s_2\lambda}}{X_2^2-1}.
\end{align*}
Note that $X^{\lambda}-X^{s_i\lambda}=X^{s_i\lambda}(X^{\langle\lambda,\alpha_i\rangle\alpha_i^{\vee}}-1)$ is indeed divisible by $X^{\alpha_i^{\vee}}-1$ because $\langle\lambda,\alpha_i\rangle\in\mathbb{Z}$ for all $\lambda\in Q$. 
\medskip

For later reference we record the following complete set of relations for $\cH_g$ in the Bernstein-Lusztig presentation. Let $Y_1=X^{\omega_1}$ and $Y_2=X^{\omega_2}$. Then
\begin{align*}
T_1^2&=1+(q^a-q^{-a})T_1&T_2^2&=1+(q^b-q^{-b})T_2&T_1T_2T_1T_2&=T_2T_1T_2T_1&Y_1Y_2&=Y_2Y_1\\
T_1^{-1}Y_1&=Y_1^{-1}Y_2T_1&T_2^{-1}Y_2&=Y_1^2Y_2^{-1}T_2+Q_0Y_1&T_1^{-1}Y_2&=Y_2T_1^{-1}&T_2^{-1}Y_1&=Y_1T_2^{-1}
\end{align*}

\begin{Rem}
Let $L:W\to\mathbb{N}_{>0}$ be the weight function with $L(s_1)=a$, $L(s_2)=b$, and $L(s_0)=c$, and let $\cH$ be the associated affine Hecke algebra, as in~(\ref{eq:definingrelations}). The results of the above section of course apply equally well to $\cH$ after applying the specialisation $\Theta_L$. For example, Proposition~\ref{prop:basischange} applies with the obvious modification $$\cQ(p)=(\sq^a-\sq^{-a})^{f_1(p)}(\sq^b-\sq^{-b})^{f_2(p)}(\sq^c-\sq^{-c})^{f_0(p)}.$$
\end{Rem}

\subsection{The extended affine Hecke algebra}\label{sec:extendedaff}

If $\sq_0=\sq_2$ (or, in the specialisation, $c=b$) one can slightly enlarge the affine Hecke algebra as follows. Let 
$$
P=\mathbb{Z}\omega_1+\mathbb{Z}\omega_2/2=\mathbb{Z}\alpha_1^{\vee}/2+\mathbb{Z}\alpha_2^{\vee}/2\quad\text{and}\quad P^+=\mathbb{Z}_{\geq 0}\omega_1+\mathbb{Z}_{\geq 0}\omega_2/2.
$$
The Weyl group $W_0$ acts on $P$, and the \textit{extended affine Weyl group} is
$$
\tilde{W}=P\rtimes W_0\cong W\rtimes (P/Q).
$$
Note that $P/Q\cong \mathbb{Z}_2$. Let $\sigma\in \tilde{W}$ be the nontrivial element of $P/Q$. Then $\sigma s_i\sigma^{-1}=s_{\sigma(i)}$ for each $i=0,1,2$, where $\sigma(i)$ denotes the nontrivial diagram automorphism of $(W,S)$. 
\medskip

The length function on $W$ is extended to $\tilde{W}$ by setting $\ell(w\sigma)=\ell(w)$ for all $w\in \tilde{W}$. Thus the length $0$ elements of $\tilde{W}$ are precisely the elements $e$ and $\sigma$. 
\medskip

Under the assumption $\sq_0=\sq_2$ we have $\sR_g=\mathbb{Z}[\sq_1,\sq_2,\sq_1^{-1},\sq_2^{-1}]$. The extended affine Hecke algebra is the algebra $\tilde{\cH}_g$ over $\sR_g$ with basis $\{T_w\mid w\in\tilde{W}\}$ and multiplication (for $u,v,w\in W$ and $s\in S$) given by 
\begin{align*}
T_uT_v&=T_{uv}&&\text{if $\ell(uv)=\ell(u)+\ell(v)$}\\
T_wT_s&=T_{ws}+(\sq_s-\sq_s^{-1})T_w&&\text{if $\ell(ws)=\ell(w)-1$}.
\end{align*}
The definition of the Bernstein-Lusztig basis $\{X_v\mid v\in \tilde{W}\}$ can be extended to $\tilde{\cH}_g$ by considering $\tilde{W}$ as $2$ sheets of $W$, and an alcove path of type $\vec{w}=s_{i_1}\cdots s_{i_{k}}\sigma$ consists of an ordinary alcove path of type $s_{i_1}\cdots s_{i_k}$ followed by a jump to the $\sigma$-sheet of $\tilde{W}$ (see \cite{Ram:06}). The centre of $\tilde{\cH}_g$ is $\sR_g[P]^{W_0}$. 
\medskip

The Hecke algebra $\cH_g$ (with $\sq_0=\sq_2$) is naturally a subalgebra of $\tilde{\cH}_g$. Indeed $\tilde{\cH}_g$ is generated by $T_0$, $T_1$, $T_2$, and the additional element 
$$
T_{\sigma}=X^{\omega_2/2}T_2^{-1}T_1^{-1}T_2^{-1}.
$$

\subsection{Schur functions}

The following Schur functions will play a role later. Let $\lambda\in Q$. The Schur function $\fs_{\lambda}(X)\in\mathbb{Z}[Q]^{W_0}$ is the polynomial
\begin{align}\label{eq:schur}
\fs_{\lambda}(X)=\sum_{w\in W_0}w\left(\frac{X^{\lambda}}{\prod_{\alpha\in \Phi_0^+}(1-X^{-\alpha^{\vee}})}\right).
\end{align}
Let $\lambda\in P$. The dual Schur function $\fs_{\lambda}'(X)\in \mathbb{Z}[P]^{W_0}$ is the polynomial
\begin{align}\label{eq:schurdual}
\fs_{\lambda}'(X)=\sum_{w\in W_0}w\left(\frac{X^{\lambda}}{\prod_{\alpha\in\Phi_1^+}(1-X^{-\alpha^{\vee}})}\right).
\end{align}
In particular we have
\begin{align*}
\fs_{\omega_1}(X)&=X^{\omega_1}+X^{-\omega_1}+X^{\omega_1-\omega_2}+X^{-\omega_1+\omega_2}&
\fs_{\omega_2}(X)&=1+X^{\omega_2}+X^{-\omega_2}+X^{2\omega_1-\omega_2}+X^{-2\omega_1+\omega_2}\\
\fs_{\omega_1}'(X)&=1+X^{\omega_1}+X^{-\omega_1}+X^{\omega_1-\omega_2}+X^{-\omega_1+\omega_2}&
\fs_{\omega_2/2}'(X)&=X^{\omega_2/2}+X^{-\omega_2/2}+X^{\omega_1-\omega_2/2}+X^{-\omega_1+\omega_2/2}.
\end{align*}


\section{Kazhdan-Lusztig cells in type $\tilde{C}_2$}\label{sec:3}

Let $W$ be a Coxeter group of type $\tilde{C}_2$ with weight diagram
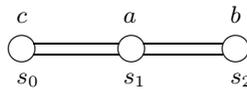
\begin{center}
\begin{picture}(150,32)
\put( 40, 10){\circle{10}}
\put( 44.5,  8){\line(1,0){31.75}}
\put( 44.5,  12){\line(1,0){31.75}}
\put( 81, 10){\circle{10}}
\put( 85.5,  8){\line(1,0){29.75}}
\put( 85.5,  12){\line(1,0){29.75}}
\put(120, 10){\circle{10}}
\put( 38, 20){$c$}
\put( 78, 20){$a$}
\put(118, 20){$b$}
\put( 38, -3){$s_{0}$}
\put( 78, -3){$s_{1}$}
\put(118,-3){$s_{2}$}
\end{picture}
\end{center}
That is, $L(s_1)=a$, $L(s_2)=b$, and $L(s_0)=c$. In this section we recall the decomposition of $W=\tilde{C}_2$ into cells for all choices of parameters $(a,b,c)\in \nN^3$. We then study the properties of this partition and introduce various notions such as the generating set of a two-sided cell, cell factorisations and the $\tba$-function. The $\tba$-function is defined using the values of Lusztig $\ba$-function in finite parabolic subgroups of $W$ and as a consequence of the main result of this paper, it turns out that $\ba=\tba$, and thus the table listed in Section \ref{sec:tba} in fact records the values of Lusztig's $\ba$-function (however, of course, this cannot be assumed at this stage).

 \subsection{Partition of $\tilde{C}_2$ into cells}
A positive weight function $L$ on $W$ is completly determined by its values $L(s_1)=a$, $L(s_2)=b$ and $L(s_0)=c$ on the set $S$ of generators. If the triplet $(a,b,c)\in\nN_{>0}$ admits a common divisor $d$ then the algebra $\cH$ defined with respect to $(a,b,c)$ is easily seen to be  isomorphic to the one defined with respect to~$(a/d,b/d,c/d)$. Therefore the Hecke algebra $\cH$ defined with respect to $(a,b,c)$ only depends on the ratios $b/a$ and $c/a$, and hence also the decomposition into cells depends only on these ratios. Thus we set
$$\fbox{\quad $r_1=\dfrac{b}{a} \quand r_2=\dfrac{c}{a}$\quad}$$

In this paper, many notions will depend on the choice of parameters and, as far as Kazhdan-Lusztig theory is concerned, it is equivalent to fix a weight function $L$, a triplet $(a,b,c)\in \nN^3$ or a pair $(r_1,r_2)\in \nQ^2$. Given $D\subset \nQ_{>0}^2$, we write 
\bem
\item $(a,b,c)\in D$ for $(a,b,c)\in \nN^3$ to mean $(b/a,c/a)\in D$;
\item $L\in D$ for a weight function $L$ to mean $(L(s_2)/L(s_1),L(s_0)/L(s_1))\in D$.
\eem
In a similar spirit, when considering a statistic $F$ that depends on the choice of parameters (for instance the partition into cells), we will write $F(L)$, or $F(a,b,c)$ or $F(r_1,r_2)$ to mean that we consider the statistic $F$ with respect to the weight function $L$, the triplet $(a,b,c)\in\nN^3$ or the pair $(r_1,r_2)\in \nQ_{>0}^2$. Furthermore, if $F(r_1,r_2)=F(r'_1,r'_2)$ whenever $(r_1,r_2)$ and $(r'_1,r'_2)$ belong to a subset $D\subset \nQ_{>0}^2$, we will also write~$F(D)$ to denote the common value of $F$ on $D$. 

\medskip

The partition of $W$ into cells has been obtained by the first author in~\cite{guilhot4}.  Even though there are an infinite number of positive weight functions on $W$, there are only a finite number of partitions of $W$ into cells (as conjectured by Bonnaf\'e in \cite{Bon:09}). In order to describe these partitions we first need to define a set $\cR$ of subsets of $\nQ_{>0}^2$ on which the partition into cells will be constant.  

\medskip

We define open subsets $A_1,\ldots,A_{10}$ of $\mathbb{Q}_{>0}^2$ in Figure~\ref{fig:regions}. Write $A_{i'}=A_i'$ for the region $A_i$ reflected in the line $r_1=r_2$ (we call this the ``dual'' region). For ``adjacent'' regions $A_i$ and $A_j$ (respectively $A_i$ and $A_i'$),  let $A_{i,j}$ (respectively $A_{i,i'}$) be the line segment $\overline{A}_i\cap \overline{A}_j$ (respectively $\overline{A}_i\cap \overline{A}_{i'}$) with the endpoints removed.
This partitions the set $\{(x_1,x_2)\in\mathbb{Q}_{>0}^2\mid x_2\leq x_1\}$ into $30$ regions:
\begin{enumerate}
\item[$\bullet$] $A_1$, $A_2$, $A_3$, $A_4$, $A_5$, $A_6$, $A_7$, $A_8$, $A_9$, $A_{10}$ (open subsets of $\mathbb{Q}^2$),
\item[$\bullet$] $A_{1,1'}$, $A_{2,2'}$, $A_{5,5'}$, $A_{1,2}$, $A_{2,3}$, $A_{3,4}$, $A_{4,5}$, $A_{3,6}$, $A_{6,7}$, $A_{4,7}$, $A_{7,8}$, $A_{5,8}$, $A_{7,9}$, $A_{9,10}$, $A_{8,10}$ (open intervals),
\item[$\bullet$] $P_1=(1/2,1/2)$, $P_2=(1,1)$, $P_3=(3/2,1/2)$, $P_4=(2,1)$, and $P_5=(3,1)$ (points).
\end{enumerate}
The set $\mathbb{Q}_{>0}^2$ is so partitioned into $55$ regions ($20$ open subsets, $27$ open intervals, and $8$ points). Let $\cR$ be the set of all such regions and let $\cR_\circ=\{A_i,A_i'\mid 1\leq i\leq 10\}$.

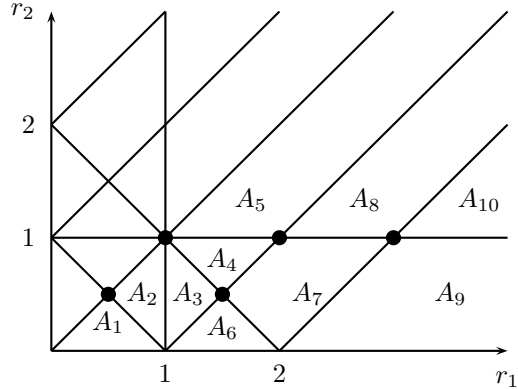
\begin{figure}[H]
\begin{center}
\psset{xunit=1.5cm}
\psset{yunit=1.5cm}
\begin{pspicture}(0,0)(3,3.5)

\psline{->}(0,0)(0,3)
\psline{->}(0,0)(4,0)

\psline(0,2)(1,3)
\psline(0,1)(2,3)
\psline(0,0)(3,3)
\psline(1,0)(4,3)
\psline(2,0)(4,2)
\psline(0,1)(4,1)
\psline(1,0)(1,3)
\psline(0,1)(1,0)
\psline(0,2)(2,0)

\rput(.5,.25){$A_{1}$}
\rput(.8,.5){$A_{2}$}
\rput(1.2,.5){$A_{3}$}
\rput(1.5,.8){$A_{4}$}
\rput(1.75,1.33){$A_{5}$}
\rput(1.5,.2){$A_{6}$}
\rput(2.25,.5){$A_{7}$}
\rput(2.75,1.33){$A_{8}$}
\rput(3.5,.5){$A_{9}$}
\rput(3.75,1.33){$A_{10}$}
\rput(4,-0.25){$r_1$}
\rput(-0.25,3){$r_2$}
\rput(1,-0.2){$1$}
\rput(-0.2,1){$1$}
\rput(2,-0.2){$2$}
\rput(-0.2,2){$2$}

\pscircle[fillstyle=solid,fillcolor=black](1,1){.1}
\pscircle[fillstyle=solid,fillcolor=black](3,1){.1}
\pscircle[fillstyle=solid,fillcolor=black](2,1){.1}
\pscircle[fillstyle=solid,fillcolor=black](1.5,.5){.1}
\pscircle[fillstyle=solid,fillcolor=black](.5,.5){.1}

\end{pspicture}
\end{center}
\caption{Regions of $\mathbb{R}^2$}
\label{fig:regions}
\end{figure}

\medskip

For any region $D\in \cR$, the decomposition of $W$ into right cells and two-sided cells is the same for all choices of parameters~$(r_1,r_2)\in D$.  In Figure~\ref{fig:partition}, we represent $\tsc(D)$ for all $D\in \cR$ such that  $D\subset \{(x_1,x_2)\in \nQ_{>0}^2\mid x_2\leq x_1\}$. The alcoves with the same colour lie in the same two-sided cell and the right cells in a given two-sided cell are the connected components.  The Hasse diagram on the right of each partition describes the two-sided order on the two-sided cells, going from the highest cell at the top to the lowest one at the bottom. Finally to obtain the decomposition and the two-sided order  for a region  included in $\{(x_1,x_2)\in \nQ_{>0}^2\mid x_2> x_1\}$ one simply applies the diagram automorphism $\sigma$ to the partition for the dual region. Hence the partition of $\tilde{C}_2$ into two-sided cells and right cells is known for all choices of parameters. 

\medskip 

\phantom{a}

\psset{linewidth=.05mm,unit=.4cm}

\begin{figure}[H]
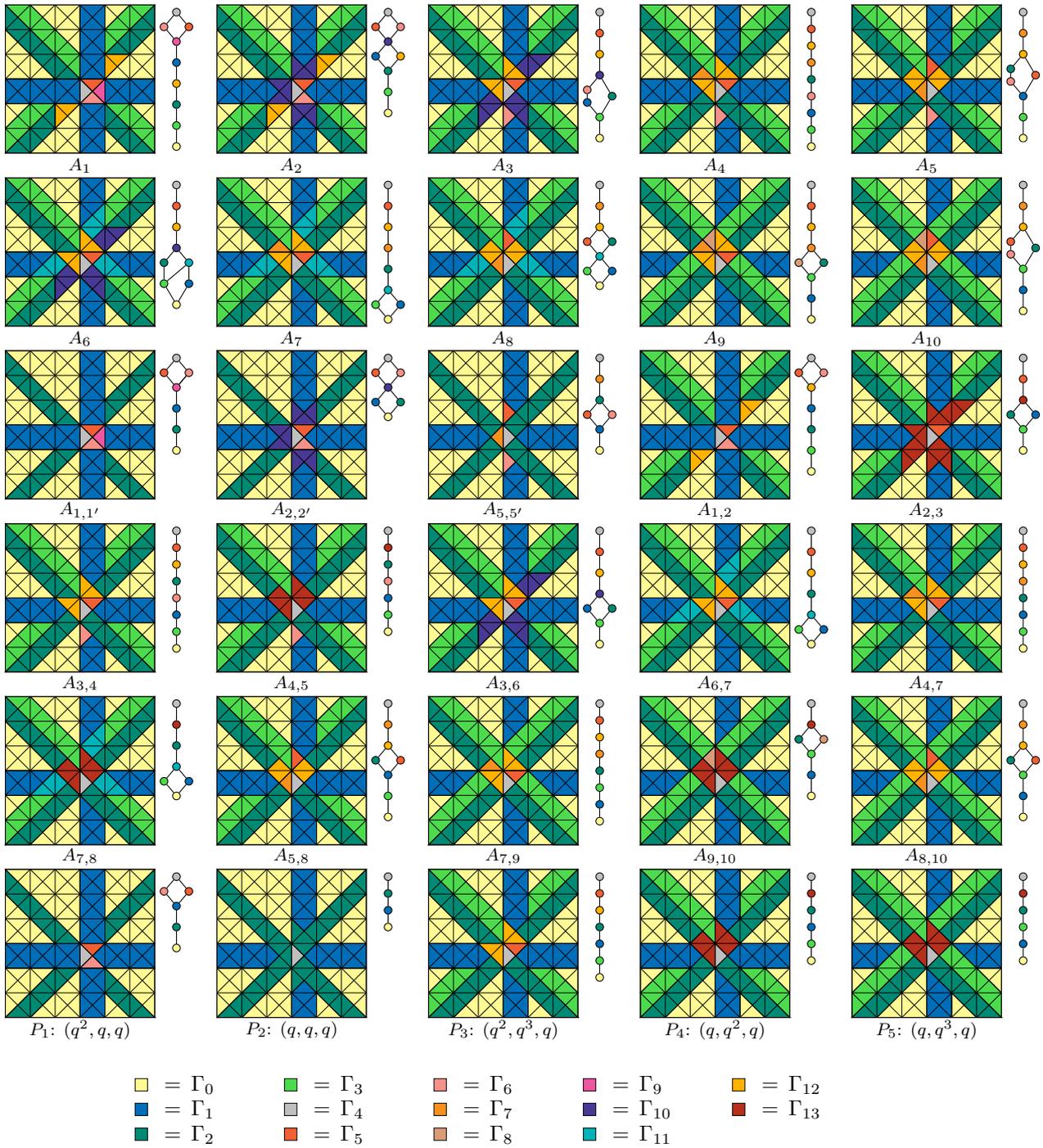

\begin{center}

\quad$=\,\Gamma_{13}$}
\end{pspicture}
\end{center}
\caption{Decomposition of $\tilde{C}_2$ into cells for $r_2\leq r_1$} 
\label{fig:partition}
\end{figure}

\begin{Cor}\label{cor:conj14}
Conjecture~$\conj{14}$ holds.
\end{Cor}

\begin{proof}
One directly checks that each two-sided cell is invariant under inversion. 
\end{proof}

\subsection{Semicontinuity conjecture}

The parameters $(r_1,r_2)\in \nQ^2_{>0}$ are called generic if there exists an open subset $O$ of $\nR^2$ that contains $(r_1,r_2)$ and such that for all $(r'_1,r'_2)\in O\cap \nQ_{>0}^2$ we have $\tsc(r_1,r_2)=\tsc(r'_1,r'_2)$. According to Figure 4, we see that the generic parameters for $W$ are exactly those that lie in some $A_i$ or $A'_i$. For $D\in \cR$ we set $\cR_D:=\{A\in \cR_\circ\mid D\subseteq \ov{A}\}$. For example,
$$
\cR_{P_2}=\{A_2,A_3,A_4,A_5,A_{2'},A_{3'},A_{4'},A_{5'}\}.
$$

In \cite{Bon:09}, Bonnaf\'e has conjectured that the partition of an arbitrary Coxeter group into cells satisfies certain ``semicontinuity properties''. The basic idea of his conjecture is that the partition for all parameters can be determined from the knowledge of the partition for generic parameters. More precisely  the partition $\tsc(D)$ for~$D\in \cR$ is the finest partition of $W$ that satisfies the following property:
$$\text{For all $A\in \cR_D$, and for all $\Ga\in \tsc(A)$, there exists a cell $\Ga'\in \tsc(D)$ such that $\Ga\subseteq \Ga'$}.
$$

 In the case of $\tilde{C}_2$ the conjecture is known to hold (by direct inspection using Figure~\ref{fig:partition}). Thus it is (retrospectively) sufficient to know $\tsc(A)$ for all $A\in \cR_\circ$ to determine $\tsc(D)$ for all $D\in \cR$ (in fact, using the diagram automorphism $\si$ it is enough to know $\tsc(A_i)$ for all $1\leq i\leq 10$). The most striking example of the semicontinuity phenomenon is when $D=P_{2}$ (the equal parameter case) where one has to look at the partition of $W$ into cells for parameters in the regions $A_2$, $A_{2'}$, $A_3$, $A_{3'}$, $A_4$, $A_{4'}$, $A_5$ and $A_{5'}$ to determine the partition into cells. As a result, all finite cells get absorbed into the infinite cells.

\subsection{Generating sets of two-sided cell}
Recall the definition of $\preceq$ in Example \ref{exa:prec}. Given a subset $C$ of $W$ we denote by $C^+$ the set that consists of all elements~$w\in W$ that satisfy $u\preceq w$ for some $u\in C$. 
By inspection of Figure \ref{fig:partition} we see that for all $D\in \cR$ and all~$\Ga\in \tsc(D)$ there exists a minimal subset $\gen_{\Ga}(D)$ of $W$ such that 
$$\Ga=\gen_{\Ga}(D)^+ - \bigcup_{\Ga'<_{\cLR} \Ga} \Ga'.$$
We call this set the generating set of $\Ga$.   We have for all $D\in \cR$ and all $\Ga\in \tsc(D)$
\ben[(1)]
 \item  $\gen_\Ga(D)\subseteq \bigcup_{I\subsetneq S} W_I$;
 \item the elements of $\sJ_{\Ga}$ are involutions;
 \item if $D\in \cR_\circ$ then  $|\gen_\Ga(D)|=1$;
 \item  we have 
 $$\displaystyle \gen_\Ga(D)\subseteq \bigcup_{A\in \cR_D} \bigcup_{\Ga'\in \tsc(A), \Ga'\cap \Ga\neq \emptyset}\gen_{\Ga'}(A)$$
 where the inclusion can be strict (see the example $D=P_2$ below); 
  \item the set $\{C_{\sw}\mid \sw\in \sJ_\Ga(D)\}$ generates the module $\cH_{{\leq_{\cLR}}\Ga}$; 
  \item $\Ga_1\leq_{\cLR} \Ga_2$ if and only if $J_{\Ga_2}(D)^+\cap \Ga_1\neq \emptyset$. 
 \een
 Of course, it is also possible to have $|\sJ_\Ga(D)|=1$ for some $D\notin \cR_\circ$. When  $|\sJ_\Ga(D)|=1$, we will denote by $\sw_{\Ga}$  the element of this set (or simply $\sw_i$ if $\Ga=\Ga_i$).  In the table below, we give the elements $\sw_i$ for all $A_j\in \cR$ and $\Ga_i\in \tsc(A_j)$. 
\begin{table}[H]
$${\footnotesize 
\begin{array}{|c||c|c|c|c|c|c|c|c|c|c|}\hline
&A_1&A_2&A_3&A_4&A_5&A_6&A_7&A_8&A_9&A_{10}\\\hline
\Ga_0& 1212& 1212& 1212& 1212& 1212& 1212& 1212& 1212& 1212& 1212 \\\hline
\Ga_1&1&20 &20 &20 &20 &212 &212 &212 &212&212  \\\hline
\Ga_2&101&101 &101 &2 &2 &101 &2 &2 &2 &2  \\\hline
\Ga_3&1010 &1010 &1010 &1010 &1010 &1010 &1010 &1010 &20 &20  \\\hline
\Ga_4&e &e &e &e &e &e &e &e &e &e  \\\hline
\Ga_5&0 &0 &0 &0 &010 &0 &0 &010 &0 &010  \\\hline
\Ga_6&2 &2 &212 &212 &212 &- &- &- &- &-  \\\hline
\Ga_7&- &- &- &101 & &- &101 &1 &101 &1  \\\hline
\Ga_8&- &- &- &- &- &- &- &- &1010 &1010  \\\hline
\Ga_9&20 &- &- &- &- &- &- &- &- &-  \\\hline
\Ga_{10}&- &1 &2 &- &- &2 &- &- &- &-  \\\hline
\Ga_{11}&- &- &- &- &- &20 &20 &20 &- &-  \\\hline
\Ga_{12}&121 &121  &1 &1 &0 &1 &1 &0 &1 &0  \\\hline
\end{array}}
$$
\caption{The set $\sJ_{\Ga_i}(A_j)=\{\sw_i\}$ for generic parameters}
\label{tab:generatingsets}
\end{table}
The set $\sJ_{\Ga}(D)$ when $\Ga\in \tsc(D)$ and $D\notin \cR_\circ$ can be obtained by first computing the right-hand side $J$ of (4) and then taking the minimal subset $J_{\mathrm{min}}$ such that $J\subset J_{\mathrm{min}}^+$. For instance, if $D=P_2$ and $\Ga=\Ga_2$ then the right-hand side of (4) is 
$$J=\{s_0,s_1,s_2,s_1s_2s_1,s_2s_1s_2,s_1s_0s_1,s_0s_1s_0\}$$
and thus $J_{\Ga_2}(P_2)=\{s_0,s_1,s_2\}$ since $s_1\prec s_1s_2s_1,s_2s_1s_2,s_0s_1s_0,s_0s_1s_0$.

\subsection{Cell factorisations}
 
 When the set $\gen_\Ga(D)$ contains a unique element then the two-sided cell $\Gamma$ admits a \textit{cell factorisation}. We refer to \cite[\S 4]{GP:17} for a detailed description of this concept in type $G_2$. To illustrate cell factorisation here, consider the lowest two-sided cell $\Gamma_0$  in the regime $r_2<r_1$. In this case we see that $\gen_{\Ga_0}(r_1,r_2)=\{\sw_0\}$ where $\sw_0=s_1s_2s_1s_2$.  By direct inspection of Figure~\ref{fig:partition} we have the following representation of elements of $\Gamma_0$:
\bem
\item Each right cell $\Up\subseteq \Gamma_0$ contains a unique element $\sw_{\Up}$ of minimal length. 
\item The element $\sw_0$ is a suffix of each $\sw_{\Up}$. Let $u_{\Up}=\sw_0\sw_{\Up}^{-1}$ and $\sB_0=\{u_{\Up}\mid \Up\subseteq \Ga_0\}$. 
\item We have
$$
\Gamma_0=\{u^{-1}\sw_0t_{\lambda}v\mid u,v\in\sB_0,\,\lambda\in Q^+\}.
$$
\eem
Moreover, each $w\in\Gamma_0$ has a unique expression in the form $w=u^{-1}\sw_0t_{\lambda}v$ with $u,v\in \sB_0$ and $\lambda\in Q^+$, and this expression is reduced (that is, $\ell(w)=\ell(u^{-1})+\ell(\sw_0)+\ell(t_{\lambda})+\ell(v)$). This expression is called the \textit{cell factorisation} of $w\in\Gamma_0$.

\medskip

In the infinite cells $\Ga=\Gamma_i$ with $i=1,2,3$ cell factorisation (if it exists) takes a similar form:
\bem
\item Each right cell  $\Up\subseteq \Gamma$ contains a unique element $\sw_{\Up}$ of minimal length.
\item The element $\sw_\Ga$ is a suffix of each $\sw_{\Up}$ and we set $u_{\Up}=\sw_\Ga \sw_{\Up}^{-1}$ and $\sB_{\Ga}=\{u_{\Up}\mid \Up\subseteq\Gamma\}$.
\item There exists $\st_{\Gamma}\in W$ such that 
$$
\Gamma=\{u^{-1}\sw_{\Gamma}\st_{\Gamma}^nv\mid u,v\in\sB_{\Gamma},n\in\mathbb{N}\},
$$
and moreover each $w\in\Gamma$ has a unique expression in this form, and this expression is reduced. 
\eem

The specific cell factorisations that we require will be introduced at the appropriate time. Here we give one example for illustration.  Consider $\Gamma=\Gamma_1(r_1,r_2)$ with $r_2<r_1-1$. Then the set $\gen_{\Ga}(r_1,r_2)$ contains a unique element $\sw_\Ga=s_2s_1s_2$. Therefore this cell admits a cell factorisation, and we have 
$$
 \st_{\Gamma}=012,\quad\text{and}\quad \sB_{\Ga}=\{e,0,01,010\}.
$$
We represent this factorisation in Figure \ref{factorisation-Ga1}. The set of grey alcoves together with the black alcove $A_0$ on the left hand side is $\sB^{-1}_{\Ga}$, and the small diagram on the right hand side illustrates $\sB_{\Ga}$. The connected sets of dark blue (respectively light blue) alcoves are the sets of the form $\{\su^{-1}\sw_{\Ga}t_{\Ga}^{n}\sv\mid \su,\sv\in \sB_\Ga\}$ where $n$
 is odd (respectively even).

\begin{figure}[H]
\psset{unit=.5cm}
\begin{subfigure}[b]{0.7\textwidth}
\begin{center}

\begin{pspicture}(-3,-3.5)(3,5)


\pspolygon[fillstyle=solid,fillcolor=black](0,0)(0,1)(.5,.5)
\pspolygon[fillstyle=solid,fillcolor=lightgray](.5,.5)(0,1)(1,1)
\pspolygon[fillstyle=solid,fillcolor=lightgray](-.5,.5)(0,1)(-1,1)
\pspolygon[fillstyle=solid,fillcolor=lightgray](0,1)(0,2)(.5,1.5)

%
%
%
%
%
%
%

\pspolygon[fillstyle=solid,fillcolor=blue1](0,0)(1,-1)(0,-2)
\pspolygon[fillstyle=solid,fillcolor=blue2](0,-2)(1,-3)(1,-1)
\pspolygon[fillstyle=solid,fillcolor=blue1](0,-2)(0,-3)(1,-3)

\pspolygon[fillstyle=solid,fillcolor=blue1](1,1)(2,0)(3,1)
\pspolygon[fillstyle=solid,fillcolor=blue2](3,1)(2,0)(4,0)
\pspolygon[fillstyle=solid,fillcolor=blue1](3,1)(4,1)(4,0)

\pspolygon[fillstyle=solid,fillcolor=blue1](-1,1)(-2,0)(-3,1)
\pspolygon[fillstyle=solid,fillcolor=blue2](-3,1)(-2,0)(-4,0)
\pspolygon[fillstyle=solid,fillcolor=blue1](-3,1)(-4,1)(-4,0)

\pspolygon[fillstyle=solid,fillcolor=blue1](0,2)(0,4)(1,3)
\pspolygon[fillstyle=solid,fillcolor=blue2](1,3)(1,5)(0,4)
\pspolygon[fillstyle=solid,fillcolor=blue1](0,4)(0,5)(1,5)
%


%
%
%
%
%

\psline(-4,-3)(4,-3)
\psline(-4,-2)(4,-2)
\psline(-4,-1)(4,-1)
\psline(-4,0)(4,-0)
\psline(-4,1)(4,1)
\psline(-4,2)(4,2)
\psline(-4,3)(4,3)
\psline(-4,4)(4,4)
\psline(-4,5)(4,5)

\psline(-4,-3)(-4,5)
\psline(-3,-3)(-3,5)
\psline(-2,-3)(-2,5)
\psline(-1,-3)(-1,5)
\psline(0,-3)(0,5)
\psline(1,-3)(1,5)
\psline(2,-3)(2,5)
\psline(3,-3)(3,5)
\psline(4,-3)(4,5)

\psline(-4,-2)(-3,-3)
\psline(-4,-1)(-2,-3)
\psline(-4,0)(-1,-3)
\psline(-4,1)(0,-3)
\psline(-4,2)(1,-3)
\psline(-4,3)(2,-3)
\psline(-4,4)(3,-3)
\psline(-4,5)(4,-3)

\psline(-3,5)(4,-2)
\psline(-2,5)(4,-1)
\psline(-1,5)(4,0)
\psline(0,5)(4,1)
\psline(1,5)(4,2)
\psline(2,5)(4,3)
\psline(3,5)(4,4)

\psline(4,-2)(3,-3)
\psline(4,-1)(2,-3)

\psline(4,0)(1,-3)
\psline(4,1)(0,-3)
\psline(4,2)(-1,-3)
\psline(4,3)(-2,-3)
\psline(4,4)(-3,-3)
\psline(4,5)(-4,-3)

\psline(3,5)(-4,-2)
\psline(2,5)(-4,-1)
\psline(1,5)(-4,0)
\psline(0,5)(-4,1)
\psline(-1,5)(-4,2)
\psline(-2,5)(-4,3)
\psline(-3,5)(-4,4)

\end{pspicture}

\end{center}
\end{subfigure}
\begin{subfigure}[b]{0.01\textwidth}
\begin{center}

\begin{pspicture}(0,-3.5)(1,5)

\pspolygon[fillstyle=solid,fillcolor=black](0,0)(0,1)(.5,.5)
\pspolygon[fillstyle=solid,fillcolor=lightgray](.5,.5)(0,1)(1,1)
\pspolygon[fillstyle=solid,fillcolor=lightgray](.5,1.5)(0,1)(1,1)
\pspolygon[fillstyle=solid,fillcolor=lightgray](0,1)(0,2)(.5,1.5)

\end{pspicture}

\end{center}
\end{subfigure}
 
\caption{Cell factorisation of $\Ga_1$ in the case $r_2<r_1-1$.}
\label{factorisation-Ga1}
\end{figure}
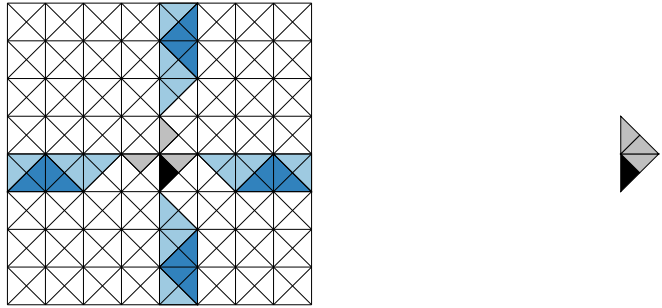

There are also cases where there is a kind of ``generalised'' cell factorisation that involves the extended affine Weyl group. Specifically, these cases are $\Ga_0$ with $r_2=r_1$, the cell $\Ga_2$ in the case $r_2=r_1$ and $r_2<1$, and the cell $\Ga_2$ in the case $r_2=r_1$ and $r_2>1$. We will discuss these factorisations at the appropriate time.
\medskip

All finite cells except for $\Ga_{13}$ admit a cell factorisation. In these cases $\st_{\Ga}=e$, and each element of the cell has a unique expression in the form $u^{-1}\sw_{\Gamma}v$ with $u,v\in\sB_{\Gamma}$ and $\sw_{\Ga}\in \gen_\Ga$. For example, if  $\Ga=\Ga_{12}$ with $(r_1,r_2)\in A_1\cup A_2\cup A_{1,2}$ then $\gen_\Ga=\{s_1s_2s_1\}$  and $\sB_{\Ga}=\{e,s_0\}$, and if  $\Ga=\Ga_{11}$ with $(r_1,r_2)\in A_6\cup A_7\cup A_8\cup A_{6,7}\cup A_{7,8}$ then  $\gen_\Ga=\{s_0s_2\}$ and $\sB_{\Ga}=\{e,s_1,s_1s_0\}$. 

\medskip

Suppose that $\Ga$ is a cell admitting a cell factorisation. If $w\in \Ga$ is written as $w=u^{-1}\sw_{\Ga}\st_{\Ga}^nv$ with $u,v\in\sB_{\Ga}$ and $n\in\mathbb{N}$ we write 
$$
\su_w=u,\quad \sv_w=v,\quad\text{and}\quad \tau_w=n
$$
(and in the case of $\Ga_0$ we have $w=u^{-1}\sw_{\Ga}t_{\lambda}v$ and $\tau_w=\lambda$). Let $x,y\in\Ga$. With these notations, we have for all generic parameters:
$$x\sim_{\cL} y \eq \sv_x=\sv_y\quand x\sim_{\cR} y \eq \su_x=\su_y.$$

\subsection{The $\tba$-function}
\label{sec:tba}
A useful auxiliary notion is the $\tba$-function, defined as follows. The values of  the $\ba$-function are explicitely known for finite dihedral groups (see, for example, \cite[Table~1]{GP:17}) and Lusztig's conjectures have been verified in this case (see \cite[Proposition 5.1]{Geck:11}). Therefore, for all choices of parameters, we can define $\ba$-functions $\ba_k:W_{I_k}\to\mathbb{N}$ $(k=0,1,2)$ where $I_k:=S\backslash\{k\}$, however we emphasise that it is not clear that $\ba_k$ is the restriction of $\ba$ to $W_{I_k}$; this is the content of~\conj{12}. It turns out, by direct observation, that if $u,v \in\Ga$ lie in a common two-sided cell, with $u\in W_{I_j}$ and $v\in W_{I_k}$ for $j,k\in\{0,1,2\}$, then 
$
\ba_j(u)=\ba_k(v).
$ These observations, together with the fact that every two-sided cell intersects a finite parabolic subgroup, allows us to define a function $\tba:W\to \mathbb{N}$ (for each choice of parameters) by
$$
\tba(w)=\ba_k(u)\quad\text{whenever $w\in\Gamma\in\tsc(r_1,r_2)$ and $u\in\Gamma\cap W_{I_k}$}.
$$
By definition $\tba$ is constant on each two-sided cell~$\Ga$, and  therefore we write $\tba(\Ga)$ for the value of $\tba$ on any element of $\Ga$, thereby considering $\tba$ as a function $\tba:\tsc(r_1,r_2)\to\mathbb{N}$.  We remark that $\tba$ is a deacreasing function on the set $\tsc$.  Indeed it is not hard to check that $\tba(\Ga)\geq \tba(\Ga')$ whenever $\Ga\leq_{\cLR} \Ga'$. Finally, the values of $\tba$ are ``generically invariant'' on the regions $D\in \cR$ as shown in the following proposition. 

\begin{Prop}\label{prop:bba}
Let $A\in \cR_{\circ}$ and $\Ga\in \tsc(A)$. There exists a unique triple $(x_1,x_2,x_3)\in \nZ^3$ such that 
$$\tba(\Ga)=x_1a+x_2b+x_3c \text{ for all $(a,b,c)\in A$}.$$
Furthermore, if $D\in \cR$ is such that $D\subseteq \ov{A}$, then for all  $\Ga'\in \tsc(D)$ such that $\Ga\subseteq \Ga'$ we have 
$$\tba(\Ga')=x_1a+x_2b+x_3c\text{ for all } (a,b,c)\in D.$$ 
\end{Prop}
\begin{proof}
This can deduced from the values of the $\ba$-function in dihedral groups: see, for example, \cite[Table~1]{GP:17}.
\end{proof}

Since  the values of  $\tba$-function will play a crucial role in the reminder of the paper, we record these values in the table below. 

\begin{table}[H]
$${\footnotesize 
\begin{array}{|c||c|c|c|c|c|c|c|c|c|c|}\hline
&A_1&A_2&A_3&A_4&A_5&A_6&A_7&A_8&A_9&A_{10}\\\hline
\Gamma_0& 2a+2b& 2a+2b& 2a+2b& 2a+2b& 2a+2b& 2a+2b& 2a+2b& 2a+2b& 2a+2b& 2a+2b \\\hline
\Gamma_1&a&b+c &b+c &b+c &b+c &-a+2b &-a+2b &-a+2b &-a+2b &-a+2b  \\\hline
\Ga_2&2a-c &2a-c &2a-c &b &b &2a-c &b &b &b &b  \\\hline
\Ga_3&2a+2c &2a+2c &2a+2c &2a+2c &2a+2c &2a+2c &2a+2c &2a+2c &b+c &b+c  \\\hline
\Ga_4&0 &0 &0 &0 &0 &0 &0 &0 &0 &0  \\\hline
\Ga_5&c &c &c &c &-a+2c &c &c &-a+2c &c &-a+2c  \\\hline
\Ga_{6}&b &b &-a+2b &-a+2b &-a+2b &- &- &- &- &-  \\\hline
\Ga_7&- &- &- &2a-c &a &- &2a-c &a &2a-c &a  \\\hline
\Ga_8&- &- &- &- &- &- &- &- &2a+2c &2a+2c  \\\hline
\Ga_9&b+c &- &- &- &- &- &- &- &- &-  \\\hline
\Ga_{10}&- &a &b &- &- &b &- &- &- &-  \\\hline
\Ga_{11}&- &- &- &- &- &b+c &b+c &b+c &- &-  \\\hline
\Ga_{12}&2a-b &2a-b &a &a &c &a &a &c &a &c  \\\hline
\end{array}}
$$
\caption{The values of $\tilde{\ba}(\Gamma_i)$ for $(b/a,c/a)\in A_j$}
\label{tab:afunction}
\end{table}

Table~\ref{tab:afunction} only lists the values of $\tba(\Ga_k)$ for $(a,b,c)$ such that $(r_1,r_2)\in A_i$ for some $1\leq i\leq 10$. The remaining cases can also be computed using Proposition~\ref{prop:bba}. However we now explain another method to deduce these values (essentially due to semicontinuity). 
\bem
\item Firstly, if $r_2>r_1$ then $\tba(\Ga_k(a,b,c))=\tba(\Ga_k(a,c,b))$. 
\item Secondly, suppose that $(r_1,r_2)\in D$ and $1\leq k\leq 13$. Let $A\in \cR_D$ and let $\Ga\in\tsc(A)$ be such that $\Ga\subseteq \Ga_k$. Then
$$
\tba(\Ga_k(a,b,c))=\underset{(b'/a',c'/a')\in A}{\lim_{(a',b',c')\to (a,b,c)}}\tba(\Ga(a',b',c'))
$$
\eem
Thus, for example, to compute $\tba(\Ga_2)$ in the equal parameter case $(r_1,r_2)=(1,1)$ we choose any $A\in \cR_{P_2}$ (for example, $A=A_2$) and any cell $\Gamma\in \tsc(A)$ with $\Ga\subseteq \Ga_2(1,1)$ (for example, $\Ga\in\{\Ga_2(A_2),\Ga_5(A_2),\Ga_6(A_2),\Ga_{10}(A_2),\Ga_{12}(A_2)\}$) and take the limit as $(a,b,c)\to(a,a,a)$ in the associated $\tba(\Ga)$ value from Table~\ref{tab:afunction}. Thus we conclude that $\tba(\Ga_2(1,1))=a$.


\section{Representations of $\cH$}\label{sec:4}

Let $(W,S)$ be the Coxeter group of type $\tilde{C}_2$ and let $L:W\to\mathbb{N}$ be a positive weight function. In this section we construct representations of $\cH$ that will ultimately be used to produce a balanced system of cell representations for each parameter regime. In fact it is convenient to define representations of the generic Hecke algebra $\cH_g$ of type $\tilde{C}_2$, from which representations of $\cH$ are obtained by the specialisation $\Theta_L$. In what follows we will use the same notations (eg, $\pi_i$) for the representations of $\cH_g$ and $\cH$. 

\subsection{The diagram automorphism}

Let $\sigma$ be the nontrivial diagram automorphism of $(W,S)$. Then $\sigma$ induces a ring automorphism of $\sR_g$ by swapping $\sq_0$ and $\sq_2$, and it is easy to check that the formula
$$
\bigg(\sum_{w\in W}a_wT_w\bigg)^{\sigma}=\sum_{w\in W}a_{w}^{\sigma}T_{w^{\sigma}}\quad\text{for $a_w\in\sR_g$}
$$
defines an involutive automorphism of $\cH_g$. 
\medskip

Suppose that $(\pi,\mathcal{M})$ be a right $\cH_g$-module over a ring $\mathsf{S}=\sR_g[\zeta_1^{\pm 1},\ldots,\zeta_n^{\pm 1}]$, where $\zeta_1,\ldots,\zeta_n$ are invertible pairwise commuting indeterminates. The diagram automorphism $\sigma$ of $(W,S)$ gives rise to a ``$\sigma$-dual'' representation $(\pi^{\sigma},\mathcal{M})$ of $\cH_g$ by
$$
\pi^{\sigma}(h)=\pi(h^{\sigma})^{\sigma},
$$
where the outer $\sigma$ is the homomorphism of $\mathrm{End}_{\mathsf{S}}(\mathcal{M})$ induced by $\sigma$.
\medskip

This construction will allow us to concentrate on the case $c\leq b$ for much of what follows, with the $c>b$ case dealt with by replacing each representation with its $\sigma$-dual.

\subsection{The principal series representation}

Let $\zeta_1$ and $\zeta_2$ be commuting indeterminates, and let $M_0$ be the $1$-dimensional right $\sR_g[Q]$ module over the ring $\sR_g[\zeta_1,\zeta_2,\zeta_1^{-1},\zeta_2^{-1}]$ with generator $\xi_0$ and $\sR_g[Q]$-action given by linearly extending 
$$
\xi_0\cdot X^{\mu}=\xi_0\,\zeta^{\mu}\quad\text{where $\zeta^{\mu}=\zeta_1^{m}\zeta_2^{n}$ if $\mu=m\alpha_1^{\vee}+ n\alpha_2^{\vee}/2$}. 
$$
Let $(\pi_0,\mathcal{M}_0)$ be the induced right $\cH_g$-module. That is,
$$
\mathcal{M}_0=\mathrm{Ind}_{\sR_g[Q]}^{\cH_g}(M_0)=M_0\otimes_{\sR_g[Q]}\cH_g.
$$
We sometimes write $\pi_0=\pi_0^{\zeta}$ when the dependence on $\zeta=(\zeta_1,\zeta_2)$ requires emphasis.
\medskip

Note that $\{\xi_0\otimes X_u\mid u\in W_0\}$ is a basis of $\mathcal{M}_0$. More generally, if $\sB$ is a fundamental domain for the action of $Q$ on $W$ then it is clear that 
$$
\cB=\{\xi_0\otimes X_u\mid u\in\sB\}
$$
is a basis of $\cM_0$. We will often write $\pi_0(T_w;\sB)$ in place of $\pi_0(T_w;\cB)$, even though strictly speaking $\sB$ is not a basis of $\mathcal{M}_0$ (c.f. notation in Section~\ref{sec:balanced}). 

\medskip

We have the following important alcove path interpretation of the matrix coefficients $[\pi_0(T_w;\sB)]_{u,v}$, as in~\cite{GP:17}.

\begin{Th}\label{thm:pi0}
Let $\sB$ be a fundamental domain for the action of $Q$ on $W$. For $u,v\in\sB$ we have
$$
[\pi_0(T_w;\sB)]_{u,v}=\sum_{\{p\in \mathcal{P}(\vec{w},u)\mid\theta_{\sB}(p)=v\}}\mathcal{Q}(p)\zeta^{\mathrm{wt}_{\sB}(p)},\quad\text{where}\quad \cQ(p)=\prod_{j=0}^2(\sq_j-\sq_j^{-1})^{f_j(p)}
$$
and where $\vec{w}$ is any reduced expression for $w$.
\end{Th}

For example, the matrices for $\pi_0(T_0)$ with respect to the ``standard basis'' $\sB=W_0$ and Lusztig's ``box basis'' $\sB=\sB_0$ are
\begin{align*}
\pi_0(T_0;W_0)&=\begin{psmallmatrix}
0&0&0&0&0&\zeta_1\zeta_2&0&0\\
0&0&0&0&\zeta_2&0&0&0\\
0&0&0&0&0&0&0&\zeta_1\zeta_2\\
0&0&0&0&0&0&\zeta_2&0\\
0&\zeta_2^{-1}&0&0&\sQ_0&0&0&0\\
\zeta_1^{-1}\zeta_2^{-1}&0&0&0&0&\sQ_0&0&0\\
0&0&0&\zeta_2^{-1}&0&0&\sQ_0&0\\
0&0&\zeta_1^{-1}\zeta_2^{-1}&0&0&0&0&\sQ_0
\end{psmallmatrix}\quad\text{and}\quad \pi_0(T_0;\sB_0)=
\begin{psmallmatrix}
0&1&0&0&0&0&0&0\\
1&\sQ_0&0&0&0&0&0&0\\
0&0&0&0&1&0&0&0\\
0&0&0&0&0&1&0&0\\
0&0&1&0&\sQ_0&0&0&0\\
0&0&0&1&0&\sQ_0&0&0\\
0&0&0&0&0&0&0&1\\
0&0&0&0&0&0&1&\sQ_0
\end{psmallmatrix},
\end{align*}
where we order $W_0=(e,1,2,12,21,121,212,1212)$ and $\sB_0=(e,0,01,012,010,0102,01021,010210)$. 

\begin{Rem}\label{rem:extend}
Suppose that $\sq_0=\sq_2$. The representation $\pi_0$ can be extended to the extended affine Hecke algebra $\tilde{\cH}_g$ as follows. Introduce an indeterminate $\zeta_1^{1/2}$ with $(\zeta_1^{1/2})^2=\zeta_1$. Let $M_0$ be the $1$-dimensional right $\sR_g[P]$ module with $\xi_0\cdot X^{\mu}=\xi_0\,\zeta^{\mu}$, where if $\mu=m\alpha_1^{\vee}/2+n\alpha_2^{\vee}/2$ then $\zeta^{\mu}=(\zeta_1^{1/2})^m\zeta_2^n$. Let $(\pi_0,\cM_0)$ be the induced right $\tilde{\cH}_g$ module. Then the restriction of $\pi_0$ to $\cH_g$ agrees with the representation defined above.  
\end{Rem}

\subsection{Induced representations}

Let $\mathcal{H}_i$ ($i=1,2$) be the subalgebra of $\cH_g$ generated by $T_i$, $X_1$ and $X_2$. Let $\zeta$ be an invertible indeterminate. Let $M_1$ be the $1$-dimensional (right) $\cH_1$-module over the ring $\sR_g[\zeta,\zeta^{-1}]$ generated by $\xi_1$ with 
\begin{align*}
\xi_1\cdot T_1&=\xi_1(-\sq_1^{-1})&\xi_1\cdot X_1&=\xi_1(\sq_1^{-2})&\xi_1\cdot X_2&=\xi_1(-\sq_1\zeta),
\end{align*}
and for $j\in \{2,3\}$ let $M_j$ be the $1$-dimensional (right) $\cH_2$-module over the ring $\sR_g[\zeta,\zeta^{-1}]$ generated by $\xi_j$ with
\begin{align*}
\xi_2\cdot T_2&=\xi_2(-\sq_2^{-1})&\xi_2\cdot X_1&=\xi_2(\sq_0\sq_2\zeta)&\xi_2\cdot X_2&=\xi_2(\sq_0^{-1}\sq_2^{-1})\\
\xi_3\cdot T_2&=\xi_3(-\sq_2^{-1})&\xi_3\cdot X_1&=\xi_3(-\sq_0^{-1}\sq_2\zeta)&\xi_3\cdot X_2&=\xi_3(-\sq_0\sq_2^{-1}).
\end{align*}
One uses the formulae in Section~\ref{sec:heckeC2} to check that the above formulae do indeed define representations of $\cH_1$ and $\cH_2$. 
\medskip

Let $(\pi_j,\mathcal{M}_j)$ with $j=1,2,3$ be the representations $\mathcal{M}_1=\mathrm{Ind}_{\cH_1}^{\cH_g}(M_1)$ and $\mathcal{M}_j=\mathrm{Ind}_{\cH_2}^{\cH_g}(M_j)$ for $j=2,3$. Then each $\mathcal{M}_j$ is a $4$-dimensional (right) $\cH_g$-module. Indeed $\{\xi_i\otimes X_u\mid u\in W_0^i\}$ is a basis of $\mathcal{M}_i$ (where we set $W_0^3=W_0^2$). More generally, if $\sB$ is a fundamental domain for the action of $\tau_i$ on $\cU_i$ (see Section~\ref{sec:stripwalks}) then 
$$
\cB=\{\xi_i\otimes X_u\mid u\in\sB\}
$$
is a basis of $\cM_i$.
\medskip

If $p$ is an $i$-folded alcove path we define
\begin{align}\label{eq:Qi}
\cQ_i(p)=\begin{cases}
(-\sq_1^{-1})^{g_1(p)}\prod_{j=0}^2(\sq_j-\sq_j^{-1})^{f_j(p)}&\text{if $i=1$}\\
(-\sq_2^{-1})^{g_2(p)}(-\sq_0^{-1})^{g_0(p)}\prod_{j=0}^2(\sq_j-\sq_j^{-1})^{f_j(p)}&\text{if $i=2$}\\
(-\sq_2^{-1})^{g_2(p)}\sq_0^{g_0(p)}\prod_{j=0}^2(\sq_j-\sq_j^{-1})^{f_j(p)}&\text{if $i=3$.}
\end{cases}
\end{align}
We note that the action of $\tau_i$ on the set of $i$-folded alcove paths preserves $\cQ_i$.

\medskip

We have the following analogue of Theorem~\ref{thm:pi0}, giving a combinatorial formula for the matrix entries of $\pi_i(T_w;\sB)$ ($i=1,2,3$) in terms of $i$-folded alcove paths. 

\begin{Th}\label{thm:pii}
Let $w\in W$, $i\in\{1,2,3\}$, and let $\sB$ be a fundamental domain for the action of $\tau_i$ on $\cU_i$. Then
$$
[\pi_i(T_w;\sB)]_{u,v}=\sum_{\{p\in\mathcal{P}_i(\vec{w},u)\mid \theta^i_{\mathsf{B}}(p)=v\}}\mathcal{Q}_i(p)\zeta^{\mathrm{wt}_{\sB}^i(p)},
$$
where $\vec{w}$ is any choice of reduced expression for~$w$.
\end{Th}

\begin{proof}
The proof is by induction, exactly as in \cite[Theorem~7.2, Corollary 7.3]{GP:17}. 
\end{proof}

For example, using the ``standard basis'' $\sB=W_0^i$ we have
\begin{align*}
\pi_1(T_0;W_0^1)&=\begin{psmallmatrix}
0&0&\zeta&0\\
0&0&0&\zeta\\
\zeta^{-1}&0&\sQ_0&0\\
0&\zeta^{-1}&0&\sQ_0
\end{psmallmatrix}&
\pi_1(T_1;W_0^1)&=
\begin{psmallmatrix}
-\sq_1^{-1}&0&0&0\\
0&\sQ_1&1&0\\
0&1&0&0\\
0&0&0&-\sq_1^{-1}
\end{psmallmatrix}&
\pi_1(T_2;W_0^1)&=
\begin{psmallmatrix}
\sQ_2&1&0&0\\
1&0&0&0\\
0&0&\sQ_2&1\\
0&0&1&0
\end{psmallmatrix}\\
\pi_{2}(T_0;W_0^2)&=\begin{psmallmatrix}
0&0&0&\zeta\\
0&-\sq_0^{-1}&0&0\\
0&0&-\sq_0^{-1}&0\\
\zeta^{-1}&0&0&\sQ_0
\end{psmallmatrix}&
\pi_{2}(T_1;W_0^2)&=\begin{psmallmatrix}
\sQ_1&1&0&0\\
1&0&0&0\\
0&0&\sQ_1&1\\
0&0&1&0
\end{psmallmatrix}&\pi_{2}(T_2;W_0^2)&=\begin{psmallmatrix}
-\sq_2^{-1}&0&0&0\\
0&\sQ_2&1&0\\
0&1&0&0\\
0&0&0&-\sq_2^{-1}
\end{psmallmatrix}\\
\pi_{3}(T_0;W_0^2)&=\begin{psmallmatrix}
0&0&0&\zeta\\
0&\sq_0&0&0\\
0&0&\sq_0&0\\
\zeta^{-1}&0&0&\sQ_0
\end{psmallmatrix}&
\pi_{3}(T_1;W_0^2)&=\begin{psmallmatrix}
\sQ_1&1&0&0\\
1&0&0&0\\
0&0&\sQ_1&1\\
0&0&1&0
\end{psmallmatrix}&\pi_{3}(T_2;W_0^2)&=\begin{psmallmatrix}
-\sq_2^{-1}&0&0&0\\
0&\sQ_2&1&0\\
0&1&0&0\\
0&0&0&-\sq_2^{-1}
\end{psmallmatrix}.
\end{align*}

\subsection{Square integrable representations}\label{sec:matrices}

The representations in this section will play a role in the analysis of the finite cells. It turns out that they are also ``square integrable representations'' (of certain natural $\mathbb{C}$-algebra specialisations of $\cH_g$), although this fact will not be particularly important in this paper. 
\medskip

Define $1$-dimensional representations $\pi_i$, $4\leq i\leq 9$, of $\cH_g$ by
\begin{align*}
(\pi_4(T_0),\pi_4(T_1),\pi_4(T_2))&=(-\sq_0^{-1},-\sq_1^{-1},-\sq_2^{-1})&(\pi_5(T_0),\pi_5(T_1),\pi_5(T_2))&=(\sq_0,-\sq_1^{-1},-\sq_2^{-1})\\
(\pi_{6}(T_0),\pi_6(T_1),\pi_6(T_2))&=(-\sq_0^{-1},-\sq_1^{-1},\sq_2)&
(\pi_7(T_0),\pi_7(T_1),\pi_7(T_2))&=(-\sq_0^{-1},\sq_1,-\sq_2^{-1})\\
(\pi_8(T_0),\pi_8(T_1),\pi_8(T_2))&=(\sq_0,\sq_1,-\sq_2^{-1})&(\pi_9(T_0),\pi_9(T_1),\pi_9(T_2))&=(\sq_0,-\sq_1^{-1},\sq_2)
\end{align*}

We now define $3$-dimensional representations $\pi_{10}$ and $\pi_{11}$. These representations were constructed as modules $\cH_{\Upsilon}$ for some right cell $\Upsilon$, however since we now consider them as representations of the generic Hecke algebra $\cH_g$ we will simply provide explicit matrices, from which the defining relations are easily checked. In the case $\pi_{10}$ we require two choices of basis for our applications, and we write the resulting matrices as $\pi_{10}(\,\cdot\,;A)$ and $\pi_{10}(\,\cdot\,;B)$. In the case $\pi_{11}$ we require three choices of basis, and we write the resulting matrices as $\pi_{11}(\,\cdot\,;A)$, $\pi_{11}(\,\cdot\,;B)$, and $\pi_{11}(\,\cdot\,;C)$. The third case only occurs for the specialised algebras with $\sq_0=\sq_1$, and indeed the matrices provided for this case below only give a representation of $\cH_g$ under the specialisation $\sq_0=\sq_1$.
\begin{align*}
\pi_{10}(T_0;A)&=
\begin{psmallmatrix}
-\sq_0^{-1}&0&0\\
1&\sq_0&0\\
0&0&-\sq_0^{-1}
\end{psmallmatrix}&
\pi_{10}(T_1;A)&=
\begin{psmallmatrix}
\sq_1&\mu_{0,1}&\mu_{1,2}\\
0&-\sq_1^{-1}&0\\
0&0&-\sq_1^{-1}
\end{psmallmatrix}&
\pi_{10}(T_2;A)&=\begin{psmallmatrix}
-\sq_2^{-1}&0&0\\
0&-\sq_2^{-1}&0\\
1&0&\sq_2
\end{psmallmatrix}\\
\pi_{10}(T_0;B)&=
\begin{psmallmatrix}
-\sq_0^{-1}&0&0\\
0&-\sq_0^{-1}&0\\
0&1&\sq_0
\end{psmallmatrix}&
\pi_{10}(T_1;B)&=\begin{psmallmatrix}
-\sq_1^{-1}&0&0\\
1&\sq_1&\mu_{0,1}\\
0&0&-\sq_1^{-1}
\end{psmallmatrix}&
\pi_{10}(T_2;B)&=\begin{psmallmatrix}
\sq_2&\mu_{1,2}&0\\
0&-\sq_2^{-1}&0\\
0&0&-\sq_2^{-1}
\end{psmallmatrix}\\
\pi_{11}(T_0;A)&=\begin{psmallmatrix}
\sq_0&\mu_{0,1}&0\\
0&-\sq_0^{-1}&0\\
0&1&\sq_0
\end{psmallmatrix}&
\pi_{11}(T_1;A)&=\begin{psmallmatrix}
-\sq_1^{-1}&0&0\\
1&\sq_1&0\\
0&0&-\sq_1^{-1}\end{psmallmatrix}&
\pi_{11}(T_2;A)&=\begin{psmallmatrix}
\sq_2&\mu_{1,2}&\nu\\
0&-\sq_2^{-1}&0\\
0&0&-\sq_2^{-1}\end{psmallmatrix}\\
\pi_{11}(T_0;B)&=\begin{psmallmatrix}
\sq_0&0&0\\
0&-\sq_0^{-1}&0\\
0&1&\sq_0
\end{psmallmatrix}&
\pi_{11}(T_1;B)&=\begin{psmallmatrix}
-\sq_1^{-1}&0&0\\
1&\sq_1&\mu_{0,1}\\
0&0&-\sq_1^{-1}\end{psmallmatrix}&
\pi_{11}(T_2;B)&=\begin{psmallmatrix}
\sq_2&\mu_{1,2}&\nu'\\
0&-\sq_2^{-1}&0\\
0&0&-\sq_2^{-1}\end{psmallmatrix}\\
\pi_{11}(T_0;C)&=\begin{psmallmatrix}\sq_1&1&0\\
0&-\sq_1^{-1}&0\\
0&1&\sq_1
\end{psmallmatrix}&
\pi_{11}(T_1;C)&=\begin{psmallmatrix}-\sq_1^{-1}&0&0\\
1&\sq_1&1\\
0&0&-\sq_1^{-1}
\end{psmallmatrix}&
\pi_{11}(T_2;C)&=\begin{psmallmatrix}\sq_2&\mu_{1,2}&\sq_1^2\sq_2^{-1}+\sq_1^{-2}\sq_2\\
0&-\sq_2^{-1}&0\\
0&0&-\sq_2^{-1}
\end{psmallmatrix}
\end{align*}
where $\mu_{i,j}=\sq_i\sq_j^{-1}+\sq_i^{-1}\sq_j$, and
\begin{align*}
\nu&=-\sq_0\sq_1^{-1}\sq_2^{-1}+\sq_0\sq_1\sq_2^{-1}+\sq_0^{-1}\sq_1^{-1}\sq_2-\sq_0^{-1}\sq_1\sq_2,&\nu'&=\sq_0^{-1}\sq_1\sq_2^{-1}+\sq_0\sq_1\sq_2^{-1}+\sq_0^{-1}\sq_1^{-1}\sq_2+\sq_0\sq_1^{-1}\sq_2.
\end{align*}

Similarly we define a $2$-dimensional representation $\pi_{12}$, equipped with two choices of basis, by
\begin{align*}
\pi_{12}(T_0;A)&=\begin{psmallmatrix}
-\sq_0^{-1}&0\\
1&\sq_0
\end{psmallmatrix}&\pi_{12}(T_1;A)&=\begin{psmallmatrix}
\sq_1&\mu_{0,1}\\
0&-\sq_1^{-1}
\end{psmallmatrix}&\pi_{12}(T_2;A)&=\begin{psmallmatrix}
-\sq_2^{-1}&0\\
0&-\sq_2^{-1}
\end{psmallmatrix}\\
\pi_{12}(T_0;B)&=\begin{psmallmatrix}
\sq_0&\mu_{0,1}\\
0&-\sq_0^{-1}\end{psmallmatrix}&
\pi_{12}(T_1;B)&=\begin{psmallmatrix}
-\sq_1^{-1}&0\\
1&\sq_1\end{psmallmatrix}&
\pi_{12}(T_2;B)&=\begin{psmallmatrix}-\sq_2^{-1}&0\\
0&-\sq_2^{-1}\end{psmallmatrix}.
\end{align*}
We will sometimes write $\pi_i^A$ in place of $\pi_i(\,\cdot\,;A)$, and similarly for $\pi_i^B$ and $\pi_i^C$. 

\subsection{A generic version of axiom $\B{1}$}

The aim of this section is to show that the representations $\pi_i$ defined above ``generically'' satisfy $\B{1}$ for the cell $\Ga_i$. Our first task is to define some specific elements in $\cH_g$ that specialise to Kazhdan-Lusztig elements. As we have seen in Example~\ref{Exa:KL-element}, this can easily be done when $w$ is the longest element of some parabolic subgroup. In this section, we extend this construction to all elements in the sets $\sJ_\Ga$.

\medskip
Let $D\in \cR$ and $\sw\in \sJ_{\Ga}(D)$ where $\Ga\in \tsc(D)$. Then either $\sw$ is the longest element of some parabolic subgroup~$W_I$ or it is of the form $\sw=sts$ where $L(s)>L(t)$ for all weight functions $L\in D$. In the first case we set 
$$\sC(\sw;D)=\sum_{y\in W_I} \sq^{-1}_{\sw}\sq_y T_y$$
and in the second case we set 
$$\sC(\sw;D)=T_{sts}+\sq_s^{-1}\left(T_{ts}+T_{st}\right)
+\left(\sq_{s}^{-1}\sq_{t}^{-1}-\sq_s^{-1}\sq_t\right)T_{s}
+\sq_s^{-2}T_{t}+\left(\sq_s^{-2}\sq_t^{-1}-\sq_s^{-2}\sq_t\right)T_{e}.
$$
\begin{Prop}
For all $D\in \cR$, $\Ga\in \tsc(D)$ and $\sw\in \sJ_{\Ga}(D)$ we have  
$$\Theta_{r_1,r_2}(\sC(\sw;D))=C_{\sw} \text{ for all } (r_1,r_2)\in D.$$
Here, the element $C_{\sw}$ on the right-hand side is computed with respect to the parameters $(r_1,r_2)$.
 \end{Prop}
\begin{proof}
This is a consequence of Example 2.12 in \cite{geck1}. 
\end{proof}

To $D\in \cR$ and $\Ga\in \tsc(D)$ we associate the set of representations $\Irr_{D}(\Ga)$ of $\cH_g$ defined by
$$\Irr_D(\Ga)=\{\pi_i\mid \exists A\in \cR_D\text{ such that }  \Ga_i\in \tsc(A) \text{ and } \Ga_i\cap \Ga\neq \emptyset\}.$$
Note that the condition $\Ga_i\cap \Ga\neq \emptyset$ is equivalent to $\Ga_i\subseteq \Ga$ by the semicontinuity conjecture.
\begin{Exa}
When $D$ lies in $\cR_\circ$, we get $\Irr_D(\Ga_i)=\{\pi_i\}$. Next assume that $D=P_2$ (the equal parameter case) and that $\Ga=\Ga_2(D)$. In this case we find that $\Irr_D(\Ga_2)=\{\pi_2,\pi_5,\pi_6,\pi_7,\pi_{10},\pi_{12}\}$. 
\end{Exa}
We prove the following theorem by explicit computations, however we note that the conceptual reason why such a result holds, at least for finite cells, is that the representations we constructed above are the natural cell modules of the specialised Hecke algebras (c.f. Section~\ref{sec:1.2}).
\begin{Th}
\label{generic-B1}
Let $D\in \cR$ and let $\Ga_i,\Ga_j\in \tsc(D)$. We have 
$$\Ga_i\not\geq_{\cLR} \Ga_{j}\Lra \pi(\sC(\sw;D))=0 \text{ for all $\pi\in \Irr_D(\Ga_j)$ and $\sw\in \sJ_{\Ga_i}(D)$}.$$
\end{Th}
\begin{proof}
The representations $\pi_i$, the cells, the two-sided order $\leq_{\cLR}$ and the sets~$\sJ_{\Ga}(D)$ are known explicitly. The proof of this theorem is therefore a matter of computations. Let us give some examples here. 
For all parameters in $A_1,\ldots,A_{10}$, the generating set of the lowest two-sided cell $\Ga_0$ is always $\{\sw_0\}$ and we have $\sC(\sw_0;D)=\sum_{y\in W_0} \sq^{-1}_{w}\sq_y T_y$. Next if $\pi_i$ is such that $i\neq 0$ then we can find parameters  $(r_1,r_2)\in A_j$ such that $\Ga_0\not\geq_{\cLR}\Ga_i$ and so we should have 
$$\pi_i(\sC(\sw_0;D))=0 \text{ for all $i\neq 0$}.$$
This is easily checked using the matrices of the representations $\pi_i$. Next let us look at the case $A_1\in \cR_\circ$. According to the two-sided order given in Figure \ref{fig:partition}, we should have 
\bem
\item $\pi_i(\sC(\sw_0;A_1))=0$ for all $i\in \{3,2,12,1,9,5,6,4\}$;
\item $\pi_i(\sC(s_1s_0s_1s_0;A_1))=0$ for all $i\in \{2,12,1,9,5,6,4\}$;
\item $\pi_i(\sC(s_1s_0s_1;A_1))=0$ for all $i\in \{12,1,9,5,6,4\}$;
\item $\pi_i(\sC(s_1s_2s_1;A_1))=0$ for all $i\in \{1,9,5,6,4\}$;
\item $\pi_i(\sC(s_1;A_1))=0$ for all $i\in \{9,5,6,4\}$;
\item $\pi_i(\sC(s_2s_0;A_1))=0$ for all $i\in \{5,6,4\}$;
\item $\pi_4(\sC(s_2;A_1))=\pi_4(\sC(s_0;A_1))=0$
\eem
and this can again be easily verified.
\end{proof}
From the properties (5) and (6) of the sets $\sJ_\Ga$, we see that  this theorem can be interpreted as a generic version of $\B{1}$.

\section{Finite cells}\label{sec:5}

In this section we construct balanced representations for each finite cell. Recall that constructing such a system requires us to associate not only a representation to each two-sided cell, but also a distinguished basis of that representation.

\begin{Th}\label{thm:finite}
Each finite two-sided cell~$\Gamma$ admits a representation~$\pi_{\Gamma}$ equipped with a basis $\cB$ satisfying $\B{1}$--$\B{5}$ with $\ba_{\pi_{\Gamma}}=\tilde{\ba}(\Gamma)$. Moreover, in all cases where the finite cell $\Gamma$ admits a cell factorisation we have
\begin{align}\label{eq:fact}
\fc_{\pi_{\Gamma}}(w;\cB)=\pm E_{\su_w,\sv_w}\quad\text{for all $w\in \Gamma$}.
\end{align}
\end{Th}

\begin{proof}
For the moment exclude the cell $\Ga_{13}$ from consideration. For all other finite cells we take $\pi_{\Gamma}$ to be the cell module $\cH_{\Upsilon}$ where $\Upsilon$ is any right cell contained in $\Gamma$, equipped with the natural Kazhdan-Lusztig basis. The matrices for $\pi_{\Gamma}$ have been computed using the {\sf CHEVIE} package~\cite{chevie2,chevie} in {\sf GAP3}~\cite{GAP}. For $4\leq i\leq 9$ we have $\pi_{\Gamma_i}=\pi_i$ (these representations are $1$-dimensional, and hence have unique representing matrices). For $i\in\{10,11,12\}$ we have the following explicit matrices:
\begin{align*}
\pi_{\Gamma_{10}}&=\begin{cases}
\pi_{10}^A&\text{if $(r_1,r_2)\in A_2\cup A_{2,2'}$}\\
\pi_{10}^B&\text{if $(r_1,r_2)\in A_3\cup A_6\cup A_{3,6}$}
\end{cases}&
\pi_{\Gamma_{11}}&=\begin{cases}
\pi_{11}^A&\text{if $(r_1,r_2)\in A_8$}\\
\pi_{11}^B&\text{if $(r_1,r_2)\in A_6\cup A_7\cup A_{6,7}$}\\
\pi_{11}^C&\text{if $(r_1,r_2)\in A_{7,8}$}
\end{cases}&
\pi_{\Gamma_{12}}&=\begin{cases}
\pi_{12}^A&\text{if $(r_1,r_2)\in X$}\\
\pi_{12}^B&\text{if $(r_1,r_2)\in Y$}
\end{cases}
\end{align*}
where $X=\{(r_1,r_2)\in\mathbb{R}_{>0}^2\mid r_2<r_1,\,r_2<1,\,r_1\neq 1\}$ and $Y=\{(r_1,r_2)\in\mathbb{R}_{>0}^2\mid r_2<r_1,\,r_2>1\}$ (note that if $(r_1,r_2)\in A_{7,8}$ then $c=a$ and hence $\pi_{11}^C$ is indeed a representation). It is then immediate that $\B{1}$ is satisfied. However we note that $\B{1}$ also follows from Theorem~\ref{generic-B1} (without needing to know that the representations above are the cell modules).
\medskip

Next we claim that $\B{2}$ and $\B{3}$ hold, with $\ba_{\Gamma}=\tilde{\ba}(\Gamma)$ (with the latter in Table~\ref{tab:afunction}). The basic approach is as follows. By $\B{1}$ we know that $\pi_{i}(C_w)=0$ whenever $w\notin(\Gamma_i)_{\geq\cLR}$. Thus it is sufficient to look at those $w$ with $w\in (\Gamma_i)_{\geq\cLR}$, and by Remark~\ref{rem:checking} we can work with the matrices $\pi_i(T_w)$ instead of $\pi_i(C_w)$. We use the Hasse diagrams in Figure~\ref{fig:partition} to compute the set $(\Gamma_i)_{\geq\cLR}$. In the case that $(\Gamma_i)_{\geq \cLR}$ is a union of finite cells (and hence is a finite set) we verify $\B{2}$ and $\B{3}$ directly by computing the matrices $\pi_i(T_w)$ for each $w\in (\Gamma_i)_{\geq \cLR}$. For example, consider the case $\Ga_{10}$ with $(r_1,r_2)\in A_3\cup A_6\cup A_{3,6}$. Then $\Ga_{\geq \cLR}=\Gamma_4\cup\Ga_5\cup\Ga_{12}\cup\Ga_{10}$, and by computing matrices we have
\begin{align*}
\max\{\deg[\pi_{10}^B(T_w)]_{i,j}\mid 1\leq i,j\leq 3\}\}=\begin{cases}
0&\text{if $w\in\Gamma_4$}\\
c&\text{if $w\in\Gamma_5$}\\
a&\text{if $w\in\Gamma_{12}$}\\
b&\text{if $w\in\Gamma_{10}$}.
\end{cases}
\end{align*}
Since $a<b$ and $c<b$ whenever $(r_1,r_2)\in A_3\cup A_6\cup A_{3,6}$ the axioms $\B{3}$ and $\B{4}$ follow. The case $(r_1,r_2)\in A_2\cup A_{2,2'}$ is similar.
\medskip

More interestingly, sometimes $(\Gamma_i)_{\geq\cLR}$ contains an infinite cell. These cases are outlined below (we note that this situation did not occur in type $\tilde{G}_2$; see \cite{GP:17}). 
\begin{enumerate}
\item Let $\Ga=\Ga_5(r_1,r_2)$. Then $\Ga_{\geq\cLR}$ contains the infinite cell $\Ga_2$ in the case $(r_1,r_2)\in A_{5,5'}$. The elements of $\Ga_2$ are 
$
\{u^{-1} 0 (1210)^k v\mid u,v\in\{e,1\}, \,k\geq 0\}\cup\{u^{-1} 2 (1012)^k v\mid u,v\in\{e,1\}, \,k\geq 0\}.
$
For $(r_1,r_2)\in A_{5,5'}$ we have $b=c$ and $c>a$, and thus if $u,v\in\{e,1\}$ we have
\begin{align*}
\deg \pi_5(u^{-1}2(1012)^kv)=\deg \pi_5(u^{-1}0(1210)^kv)&\leq c-2ak-bk+ck=c-2ak<c.
\end{align*}
Thus $\deg\pi_5(T_w)<c<2c-a=\tilde{\ba}_{\Gamma_5}$ for all $w\in\Ga_2$. The analysis for the cells $\Ga_i$ with $6\leq i\leq 9$ is similar. 
\item Let $\Ga=\Ga_{11}(r_1,r_2)$. Then $\Ga_{\geq\cLR}$ contains the infinite cell $\Gamma_2$ in the cases $(r_1,r_2)\in A_7\cup A_8\cup A_{6,7}\cup A_{7,8}$. In the regime $(r_1,r_2)\in A_7\cup A_8\cup A_{7,8}$ the cell $\Gamma_2(r_1,r_2)$ admits a cell factorisation with $\sB_{\Ga_2}=\{e,1,10,101\}$, $\st_2=\st_{\Ga_2}=1012$ and $\sw_2=\sw_{\Ga_2}=2$. If $(r_1,r_2)\in A_{6,7}$ we have $\Ga_2(r_1,r_2)=\Ga_2(A_7)\cup\{101\}$, and so we can use the cell factorisation in $A_7$ to describe all but one element of $\Ga_2$. 
\medskip

Let us consider one case in detail (the remaining cases are similar). Suppose that $(r_1,r_2)\in A_{7,8}$ (thus $c=a$ and $2a<b<3a$). Let $z=\sq^{4a-2b}$. By diagonalising $\pi_{11}^C(\st_2)$ we obtain 
\begin{align*}
\pi_{11}^C(\st_2^n)=(-1)^n\sq^{(-3a+b)n}\begin{psmallmatrix}
-z\phi_{n-1}(z)&-\sq^{3a-2b}\phi_n(z)&-\sq^{2a-2b}\phi_n(z)\\
0&\sq^{(-4a+2b)n}&0\\
\sq^{2a}\phi_n(z)&\sq^a\phi_n(z)&\phi_{n+1}(z)
\end{psmallmatrix}\quad\text{where}\quad \phi_n(z)=\frac{1-z^n}{1-z},
\end{align*}
with $\phi_{-1}(z)=-z^{-1}$. Since $4a-2b<0$ for $(r_1,r_2)\in A_{7,8}$ we have
$$
\phi_n(z)=1+z+\cdots+z^{n-1}\in\mathbb{Z}[\sq^{-1}]\quad\text{for $n\geq 0$}.
$$
It is then a straightforward (although somewhat tedious) exercise to show that the degrees of the matrix entries of $\pi_{11}^C(w)$ are strictly bounded by $a+b$ for all elements $w=u^{-1}\sw_2\st_2^nv\in \Ga_2$. 

\item Let $\Ga=\Ga_{12}(r_1,r_2)$. Then $\Ga_{\geq\cLR}$ contains the infinite cell $\Ga_1$ in the case $(r_1,r_2)\in A_1$. For $(r_1,r_2)\in A_1$ the cell $\Gamma_1$ admits a cell factorisation with $\sB_{\Ga_1}=\{e,0,2,02\}$, $\st_1=\st_{\Ga_1}=021$, and $\sw_1=\sw_{\Ga_1}=1$. We compute 
$$
\pi_{12}^A(\sw_{1}\st_{1}^{2n})=(-1)^n\sq^{-2nb}\begin{psmallmatrix}
\sq^a&\mu_{0,1}\\
0&-\sq^{-a}
\end{psmallmatrix}\quad\text{and}\quad \pi_{12}^A(\sw_{1}\st_{1}^{2n+1})=(-1)^n\sq^{-(2n+1)b}\begin{psmallmatrix}
-\sq^c&0\\
1&\sq^{-c}
\end{psmallmatrix}.
$$
It is then easy to compute $\pi_{12}^A(T_w)$ for all $w=u^{-1}\sw_{1}\st_{1}^nv$ with $n\in\mathbb{N}$ and $u,v\in\sB_{\Ga_1}$, and the result follows. 
\end{enumerate}

Thus $\B{1}$, $\B{2}$, and $\B{3}$ hold for all cells $\Ga_i$ with $4\leq i\leq 12$. Moreover, these cells admit cell factorisations, and the leading matrices are easily computed directly, verifying that~(\ref{eq:fact}) holds. For the cells $\Gamma_i$ with $4\leq i\leq 9$ the sign in~(\ref{eq:fact}) is easily computed (since the associated representations are $1$-dimensional). In the remaining cases we have the $+$ sign except for the case $\pi_{12}$ with $(r_1,r_2)\in A_1\cup A_2\cup A_{1,2}$ in which case we have the $-$~sign. 
\medskip

It is thus clear, from~(\ref{eq:fact}), that $\B{4}$ holds. To verify $\B{5}$ for the cell $\Ga=\Ga_i$ we note that if $w=u^{-1}\sw_{\Ga}v$ then 
$$
\fc_{\pi_{\Ga}}(u^{-1}\sw_{\Ga}u)\fc_{\pi_{\Ga}}(w)=\pm E_{u,u}E_{u,v}=\pm E_{u,v}=\pm \fc_{\pi_{\Ga}}(w).
$$
This completes the analysis for the finite cells $\Ga_i$ with $4\leq i\leq 12$. 
\medskip

We now consider the remaining cell $\Ga=\Gamma_{13}$. This cell appears for $(r_1,r_2)\in A_{2,3}\cup A_{4,5}\cup A_{7,8}\cup A_{9,10}\cup P_4\cup P_5$. We first consider the cases $(r_1,r_2)\in A_{4,5}\cup A_{7,8}\cup A_{9,10}\cup P_4\cup P_5$ (these are precisely the parameters with $r_2\leq r_1$ and $r_2=1$). In these cases $\Ga_{13}=\Up_1\cup\Up_2$ is a union of two right cells 
$\Up_1=\{0,01,010\}$ and $\Up_2=\{1,10,101\}$. Let
$$
\pi_{\Gamma}=\pi_5\oplus\pi_7\oplus\pi_{12}^B.
$$ 
By Theorem \ref{generic-B1}, we can see that $\pi_{\Ga}$ satisfies $\B{1}$. 

\medskip

Next we note that $\B{2}$ and $\B{3}$, with $\ba_{\pi_{\Ga}}=a$, hold by an easy direct calculation (note that $\Ga_{\geq\cLR}=\Ga_4\cup\Ga_{13}$ is finite). Moreover the leading matrices are computed directly as
\begin{align*}
\fc_{\pi_{\Ga}}(0)&=E_{11}+E_{33}&
\fc_{\pi_{\Ga}}(01)&=2E_{34}&
\fc_{\pi_{\Ga}}(010)&=-E_{11}+E_{33}\\
\fc_{\pi_{\Ga}}(1)&=E_{22}+E_{44}&
\fc_{\pi_{\Ga}}(10)&=E_{43}&
\fc_{\pi_{\Ga}}(101)&=-E_{22}+E_{44},
\end{align*}
and hence $\B{4}$ holds. Let $d_1,d_2\in \Ga_{13}$ be the elements $d_1=0$ and $d_2=1$ (these turn out to be the Duflo involutions; see Theorem~\ref{thm:Duflo}). Then the formulae above give
\begin{align*}
\fc_{\pi_{\Ga}}(d_i)\fc_{\pi_{\Ga}}(w)&=\fc_{\pi_{\Ga}}(w)\quad\text{for all $w\in \Up_i$, $i\in\{1,2\}$},
\end{align*}
and hence $\B{5}$ holds.

\medskip

Finally consider $(r_1,r_2)\in A_{2,3}$. In this case $\Ga=\Up_1\cup\Up_2\cup\Up_3$ is a union of right cells $\Up_1=\{1,10,12,121,1210\}$, $\Up_2=\{2,21,212,210\}$, and $\Up_3=\{01,010,012,0121,01210\}$. Let
$$
\pi_{\Gamma}=\pi_{6}\oplus\pi_{12}^A\oplus\pi_{10}^B.
$$
Once again, Theorem \ref{generic-B1} yields that $\pi_{\Gamma}$ satisfies $\B{1}$.  Moreover $\B{2}$ and $\B{3}$ hold by direct calculation with $\ba_{\pi_{\Gamma}}=a$, and the leading matrices are computed as
\begin{align*}
\fc_{\pi_{\Ga}}(1)&=E_{22}+E_{55}&
\fc_{\pi_{\Ga}}(10)&=E_{23}+E_{56}&
\fc_{\pi_{\Ga}}(12)&=E_{54}&
\fc_{\pi_{\Ga}}(121)&=-E_{22}+E_{55}\\
\fc_{\pi_{\Ga}}(1210)&=-E_{23}+E_{56}&
\fc_{\pi_{\Ga}}(2)&=E_{11}+E_{44}&
\fc_{\pi_{\Ga}}(21)&=2E_{45}&
\fc_{\pi_{\Ga}}(212)&=-E_{11}+E_{44}\\
\fc_{\pi_{\Ga}}(210)&=2E_{46}&
\fc_{\pi_{\Ga}}(01)&=E_{32}+E_{65}&
\fc_{\pi_{\Ga}}(010)&=E_{33}+E_{66}&
\fc_{\pi_{\Ga}}(012)&=E_{64}\\
\fc_{\pi_{\Ga}}(0121)&=-E_{32}+E_{65}&
\fc_{\pi_{\Ga}}(01210)&=-E_{33}+E_{66}
\end{align*}
and $\B{4}$ follows. Let $d_1=1$, $d_2=2$, and $d_3=01$ (again, these turn out to be the Duflo involutions; see Theorem~\ref{thm:Duflo}). Then the formulae above give
\begin{align*}
\fc(d_i)\fc(w)&=\fc(w)\quad\text{for all $w\in \Up_i$, $i\in\{1,2,3\}$}
\end{align*}
and hence $\B{5}$ holds, completing the proof. 
\end{proof}


\section{Infinite cells}\label{sec:infinite}

In this section we construct balanced representations for the infinite cells $\Gamma_i$ with $i\in\{0,1,2,3\}$ for all choices of parameters. The results of this section, along with Theorem~\ref{thm:finite}, give the following:

\begin{Th}\label{thm:balancedsystems}
For each choice of parameters $(a,b,c)\in\mathbb{Z}_{>0}^3$ there exists a balanced system of cell representations $(\pi_{\Ga})_{\Ga\in\tsc}$ for $\cH$ with bounds $\ba_{\pi_{\Ga}}=\tba(\Ga)$. 
\end{Th}

\begin{proof}
By Theorem~\ref{thm:finite} and Theorems~\ref{thm:Ga0}, \ref{thm:Ga0equal}, \ref{thm:gamma2part1}, \ref{thm:gamma2part2}, \ref{thm:gamma2part3}, \ref{thm:Ga1} and \ref{thm:Ga3} below we have a system $(\pi_{\Ga})_{\Ga\in\tsc}$ for each parameter range satisfying $\B{1}$--$\B{5}$ with $\ba_{\pi_{\Ga}}=\tba(\Ga)$. Then $\B{6}$ follows from the fact that $\tba(\Ga')\geq \tba(\Ga)$ whenever $\Ga'\leq_{\cLR}\Ga$ (see Table~\ref{tab:afunction}). 
\end{proof}

Thus, combined with Theorem~\ref{thm:afn} we can compute Lusztig's $\ba$-function. In fact, we have:

\begin{Cor}\label{cor:afn}
Table~\ref{tab:afunction} (and the discussion immediately following the table) gives the values of Lusztig's $\ba$-function for all choices of parameters. Moreover, the conjectures $\conj{4}$, $\conj{9}$, $\conj{10}$, $\conj{11}$, and $\conj{12}$ hold. 
\end{Cor}

\begin{proof}
It follows from Theorems~\ref{thm:afn} and~\ref{thm:balancedsystems} that Lusztig's $\ba$-function is given by Table~\ref{tab:afunction}. Conjectures $\conj{4}$, $\conj{9}$, $\conj{10}$, $\conj{11}$ and $\conj{12}$ are then easily checked using the explicit values of the $\ba$-function. In fact, due to the logical dependencies amongst the conjectures established in \cite[Chapter~14]{bible} it is sufficient to prove $\conj{4}$, $\conj{10}$, and $\conj{12}$, which are obvious from the explicit values of the $\ba$-function and the explicit decomposition of $W$ into right cells given in Figure~\ref{fig:partition}. Then $\conj{10}\Rightarrow\conj{9}$ and $\conj{4}+\conj{9}+\conj{10}\Rightarrow\conj{11}$. 
\end{proof}

Of course it remains to exhibit balanced systems for the infinite cells. We undertake this rather intricate task in the present section. Let us begin by noting the following immediate consequence of Theorem \ref{generic-B1}.

\begin{Cor}\label{cor:B1}
Let $i\in\{0,1,2,3\}$. The representation $\pi_i$ satisfies $\B{1}$ for the cell $\Ga_i$. 
\end{Cor} 

\subsection{The lowest two-sided cell}

Suppose first that $c\neq b$. It is sufficient to consider the case $c<b$, for if $c>b$ one can apply the diagram automorphism~$\sigma$. In the case $c<b$ the lowest two-sided cell $\Gamma_0$ admits a cell factorisation
$$
\Gamma_0=\{u^{-1}\sw_0t_{\lambda}v\mid u,v\in\sB_0,\,\lambda\in Q^+\}\quad\text{where}\quad \sB_0=\{e,0,01,012,010,0102,01021,010210\},
$$
and if $w=u^{-1}\sw_0 t_{\lambda}v$ is written in this form we define $\su_w=u$, $\sv_w=v$, and $\tau_w=\lambda$. 
\medskip

Since $\sB_0$ is a fundamental domain for the action of $Q$ on $W$ the set $\cB_0=\{\xi_0\otimes X_u\mid u\in\sB_0\}$ is a basis of $\mathcal{M}_0$.  The proof of the following theorem is very similar to \cite[Section~6]{GP:17} with only some minor adjustments, and so we will only sketch the argument.

\begin{Th}\label{thm:Ga0}
Let $c<b$. The representation $\pi_0$, equipped with the basis $\cB_0=\{\xi_0\otimes X_u\mid u\in\sB_0\}$, satisfies $\B{1}$--$\B{5}$ for the lowest two-sided cell~$\Gamma_0$, with $\ba_{\pi_0}=2a+2b$. Moreover, the leading matrices of $\pi_0$ are
$$
\fc_{\pi_0}(w;\sB_0)=\fs_{\tau_{w}}(\zeta)E_{\su_w,\sv_w}\quad\text{for $w\in\Gamma_0$},
$$
where $\fs_{\lambda}(\zeta)$ is the Schur function defined in~(\ref{eq:schur}). 
\end{Th}

\begin{proof}
We have already verified~$\B{1}$ in Corollary~\ref{cor:B1}. To verify $\B{2}$ we note that $\deg \cQ(p)\leq \max\{2a+2b,2a+2c\}$ for all positively folded alcove paths, and so for $c<b$ we have $\deg\cQ(p)\leq 2a+2b$ (see~\cite[Lemma~6.2]{GP:17}). Thus $\B{2}$ follows from Theorem~\ref{thm:pi0}. 
\medskip

Axiom~$\B{3}$ is verified as in \cite[Theorem~6.6]{GP:17}, with one additional ingredient: If $\deg(\cQ(p))=2a+2b$ then necessarily $p$ has no folds on type $0$-walls (for otherwise the degree is bounded by $2a+b+c<2a+2b$). The only simple hyperplane direction available in the ``box'' $\sB_0$ is a type $0$-wall, and thus if $p$ is a maximal path of type $u^{-1}\sw_0t_{\lambda}v$ with $u,v\in\sB_0$ then by the above observation there is no fold on this wall in the final $v$-part of the path (see \cite[Remark~6.4]{GP:17}). With this observation in hand the proof of \cite[Theorem~6.6]{GP:17} applies verbatim, including the calculation of the leading matrices. Linear independence of the Schur functions gives $\B{4}$, and to verify $\B{5}$ we note that if $w\in\Gamma_0$ then 
$$
\fc_{\pi_0}(\su_w^{-1}\sw_0\su_w;\sB_0)\fc_{\pi_0}(w;\sB_0)=\fs_0(\zeta)\fs_{\tau_w}(\zeta)E_{\su_w,\su_w}E_{\su_w,\sv_w}=\fs_{\tau_w}(\zeta)E_{\su_w,\sv_w}=\fc_{\pi_0}(w;\sB_0),
$$
and the proof is complete.
\end{proof}

Now suppose that $b=c$. In this case we will work in the extended affine Weyl group $\tilde{W}$ and the extended affine Hecke algebra $\tilde{\cH}$. See Remark~\ref{rem:extend} for the definition of the principal series representation $(\pi_0,\cM_0)$ in this case.  
\medskip

Let $\sB_{1/2}=\{e,0,01,012\}$ be the ``half box''. Each element $w\in W$ of the (non-extended) affine Weyl group can be written uniquely as 
\begin{align}\label{eq:ww}
w=t_{\lambda}v\quad\text{with either $\lambda\in Q$ and $v\in \sB_{1/2}$, or with $\lambda\in P\backslash Q$ and $v\in\sB_{1/2}\sigma$}.
\end{align}
We will work with the basis
\begin{align*}
\cB_0=\{\xi_0\otimes X_u\mid u\in\sB_{1/2}\cup\sB_{1/2}\sigma\}
\end{align*}
of the module $\mathcal{M}_0$. Then, as in Theorem~\ref{thm:pi0}, with respect to this basis we have
\begin{align}\label{eq:pathspi0equal}
[\pi_0(T_w;\cB_0)]_{u,v}=\sum_{\{p\in\mathcal{P}(\vec{w},u)\mid\theta(p)=v\}}\cQ(p)\zeta^{\mathrm{wt}(p)},
\end{align}
where, if $w=t_{\lambda}v$ as in~(\ref{eq:ww}), then $\mathrm{wt}(w)=\lambda$ and $\theta(w)=v$. 
\medskip

We have the following generalised cell factorisation: Each $w\in\Gamma_0$ has a unique expression as 
\begin{align}\label{eq:www}
w=u^{-1}\sw_0t_{\lambda}v\quad\text{with $u,v\in\sB_{1/2}\cup\sB_{1/2}\sigma$ and $\lambda\in P^+$}.
\end{align}
If $w\in\Gamma_0$ is written in the form~(\ref{eq:www}) then we define $\su_w=u$, $\sv_w=v$, and $\tau_w=\lambda$. 

\begin{Th}\label{thm:Ga0equal}
Let $c=b$. The representation $\pi_0$, equipped with the basis~$\cB_0$, satisfies $\B{1}$--$\B{5}$ for the lowest two-sided cell~$\Gamma_0$, with $\ba_{\pi_0}=2a+2b$. Moreover, the leading matrices of $\pi_0$ are
$$
\fc_{\pi_0}(w;\cB_0)=\fs_{\tau_{w}}'(\zeta)E_{\su_w,\sv_w}\quad\text{for $w\in\Gamma_0$},
$$
where $\fs_{\lambda}'(\zeta)$ is the Schur function defined in~(\ref{eq:schurdual}). 
\end{Th}

\begin{proof} 
The proof is again very similar to \cite[Theorem~6.6]{GP:17}. The choice of ``box'' $\sB_0'=\sB_{1/2}\cup\sB_{1/2}\sigma$ again implies that if $p$ is a maximal path of type $u^{-1}\sw_0t_{\lambda}v$ with $u,v\in\sB_0'$ then there are no folds in the  final $v$-part of the path. Moreover, a slight generalisation of \cite[Theorem~3.4]{GP:17} gives
$$
\fs_{\lambda}'(X)=\sum_{p\in\mathbb{P}(\vec{\sw_0}\cdot\vec{t}_{\lambda},e)}X^{\mathrm{wt}(p)}\quad\text{for $\lambda\in P^+$},
$$
and the proof of \cite[Theorem~6.6]{GP:17} now applies verbatim. 
\end{proof}

\subsection{Slices of the induced representations and folding tables}

In the following sections we analyse the remaining infinite cells $\Ga_i$ with $i\in\{1,2,3\}$. The basic idea is to use the combinatorial description of the matrix entries from Theorem~\ref{thm:pii} to show that the representation $\pi_i$ is balanced for the cell~$\Ga_i$. Thus we are primarily interested in the $i$-folded alcove paths that attain the maximal value of $\deg(\cQ_i(p))$, as these are the terms that contribute to the leading matrices. However the situation is complicated by the large number of distinct parameter regimes for the cells $\Ga_i$ as the $i$-folded alcove paths that attain the maximal value of $\deg(\cQ_i(p))$ vary with the parameter regimes. 
\medskip

Therefore it is desirable to be able to work with all parameter regimes simultaneously. To achieve this we work in the generic Hecke algebra~$\cH_g$. In this setting the degree of the multivariate polynomial $\cQ_i(p)$ (see~(\ref{eq:Qi})) is too crude for our purposes, and so we introduce a more refined statistic, which we call the \textit{exponent} of $\cQ_i(p)$, defined as follows. Firstly, if $\mathbf{x}=(x,y,z)\in\mathbb{Z}^3$ then the \textit{exponent} of the monomial $\sq^{\mathbf{x}}:=\sq_1^x\sq_2^y\sq_0^z$ is $\exp(\sq^{\mathbf{x}})=(x,y,z)\in\mathbb{Z}^3$. Let $\preceq$ denote the partial order on $\mathbb{Z}^3$ with $(x',y',z')\preceq (x,y,z)$ if and only if $x-x'\geq 0$, $y-y'\geq 0$, and $z-z'\geq 0$. 

\begin{Def}
Let $i\in\{1,2,3\}$ and let $p$ be an $i$-folded alcove path. Then, by direct inspection of the formula~(\ref{eq:Qi}), the multivariate polynomial $\cQ_i(p)$ has a unique monomial with exponent maximal with respect to $\preceq$. We denote this maximal exponent by $\exp(\cQ_i(p))$. Explicitly,
\begin{align*}
\exp(\cQ_i(p))=\begin{cases}
(f_1(p)-g_1(p),f_2(p),f_0(p))&\text{if $i=1$}\\
(f_1(p),f_2(p)-g_2(p),f_0(p)-g_0(p))&\text{if $i=2$}\\
(f_1(p),f_2(p)-g_2(p),f_0(p)+g_0(p))&\text{if $i=3$}.
\end{cases}
\end{align*}
\end{Def}

Note that if $\exp(\cQ_i(p))=(x,y,z)$ then on specialising $\sq_0\to\sq^c,\,\sq_1\to\sq^a,\,\sq_2\to \sq^b$ we have 
\begin{align}\label{eq:degreeconv}
\deg(\cQ_i(p))=xa+yb+zc.
\end{align}

\begin{Def} Let $\sB$ be a fundamental domain for the action of $\tau_i$ on $\cU_i$.  Let
$$
\mathbb{E}(\pi_i;\sB)=\{\mathbf{x}\in\mathbb{Z}^3\mid \text{$\sq^{\mathbf{x}}$ appears with nonzero coefficient in some matrix entry of $\pi_i(T_w;\sB)$ for some $w\in W$}\},
$$
where $\cB=\{\xi_i\otimes X_u\mid u\in\sB\}$ is the basis of $\cM_i$ associated to $\sB$. 
\end{Def}

\begin{Lem}\label{lem:indep}
If $\sB$ and $\sB'$ are fundamental domains for the action of $\tau_i$ on $\cU_i$ then $\mathbb{E}(\pi_i;\sB)=\mathbb{E}(\pi_i;\sB')$. 
\end{Lem}

\begin{proof}
We may write each $u\in \sB$ as $u=\tau_i^ku'$ for some $k\in\mathbb{Z}$ and $u'\in\sB'$. We claim that 
$$
\xi_i\otimes X_u=(\xi_i\otimes X_{u'})\zeta^k.
$$
Consider the case $i=2,3$. Then by (\ref{eq:split}) we have $X_u=X^{k\omega_1}X_{u'}$, and the result follows since $\xi_i\cdot X^{\omega_1}=\xi_i\,\zeta$ for $i=2,3$. If $i=1$ then we have $X_u=X_{\tau_1}^kX_{u'}$ (this follows from the fact that $\tau_1$ preserves the orientation of all hyperplanes except for the hyperplanes in the $\alpha_1$ parallelism class, and that this class is not encountered in $\cU_1$). Since 
$$
\xi_1\cdot X_{\tau_1}=\xi_1\cdot X^{\omega_1}T_1^{-1}=\xi_1\,(-\sq_1^{-1}\zeta)(-\sq_1)=\xi_1\,\zeta
$$
the claim follows. 
\medskip

Thus the change of basis matrix from the $\cB$ basis to the $\cB'$ basis is a monomial matrix with entries in $\mathbb{Z}[\zeta]$, and the lemma follows.
\end{proof}

Thus we can define 
$$
\mathbb{E}(\pi_i)=\mathbb{E}(\pi_i;\sB)\quad\text{for any fundamental domain $\sB$}.
$$
We will show below (in the course of the proof of Theorem~\ref{thm:Gamma2Paths}) that the elements of $\mathbb{E}(\pi_i)$ are bounded above in each component -- we will assume this fact for the moment. Let
$$
\mathbb{M}(\pi_i)=\{ \text{maximal elements of the partially ordered set }(\mathbb{E}(\pi_i),\preceq)\}.
$$

\begin{Def}
Let $\sB$ be a fundamental domain for the action of $\tau_i$ on $\cU_i$. For $\mathbf{x}=(x,y,z)\in \mathbb{Z}^3$ the \textit{$\mathbf{x}$-slice} of $\pi_i(T_w;\sB)$ is the matrix $\fc_{\pi_i}^{\mathbf{x}}(w;\sB)$ whose $(u,v)^{th}$ entry is the coefficient of $\sq^{\mathbf{x}}$ in $[\pi_i(T_w;\sB)]_{u,v}$. Thus $\fc_{\pi_i}^{\mathbf{x}}(w;\sB)$ is a matrix with entries in $\mathbb{Z}[\zeta,\zeta^{-1}]$. 
\end{Def}

The following key theorem shows that the slices $\fc_{\pi_i}^{\mathbf{x}}(w;\sB)$ with $\mathbf{x}\in \mathbb{M}(\pi_i)$ are sufficient to compute leading matrices in all parameter ranges.

\begin{Th}\label{thm:sumslices}
Let $(a,b,c)$ be a fixed choice of parameters, and suppose that property $\B{2}$ holds for $\pi_i(\cdot,\sB)$ with bound~$\ba_{\pi_i}$. Suppose that $xa+yb+zc\leq \ba_{\pi_i}$ for all $(x,y,z)\in \mathbb{M}(\pi_i)$. Then 
$$
\fc_{\pi_i}(w;\sB)=\sum\fc_{\pi_i}^{\mathbf{x}}(w;\sB),
$$
where the sum is over those $\mathbf{x}=(x,y,z)\in\mathbb{M}(\pi_i)$ with $xa+yb+zc=\ba_{\pi_i}$. 
\end{Th}

\begin{proof} By Theorem~\ref{thm:pii} the entry $[\fc_{\pi_i}(w;\sB)]_{u,v}$ of the leading matrix $\fc_{\pi_i}(w;\sB)$ is given as a sum over paths $p\in\mathcal{P}_i(\vec{w};u)$ with $\deg(\cQ_i(p))=\ba_{\pi_i}$. Thus it suffices to show that if $\exp(\cQ_i(p))\notin\mathbb{M}(\pi_i)$ then, after specialising, $\deg(\cQ_i(p))<\ba_{\pi_i}$. 
\medskip

Suppose that $p$ is an $i$-folded alcove path with $\exp(\cQ_i(p))=(x,y,z)\notin\mathbb{M}(\pi_i)$. Hence there is an $i$-folded alcove path $p'$ with $(x,y,z)\prec(x',y',z')=\exp(\cQ_i(p'))\in\mathbb{M}(\pi_i)$. Thus $x'-x$, $y'-y$ and $z'-z$ are all nonnegative with at least one being strictly positive. Thus
$(x'-x)a+(y'-y)b+(z'-z)c>0$, 
and so by~(\ref{eq:degreeconv}) we have, after specialising,
$$
\deg(\cQ_i(p))=xa+yb+zc<x'a+y'b+z'c\leq\ba_{\pi_i},
$$
and hence the result.
\end{proof}

Thus our approach in the following sections is to compute $\mathbb{M}(\pi_i)$ and the slices corresponding to these maximal exponents. In fact, the cell $\Ga_2$ turns out to be the most complicated, in part due to the intricate equal parameter regime. Thus we give complete details for $\Ga_2$, and we will only state the results for the easier cells $\Gamma_1$ and $\Ga_3$.

\begin{Rem}
The hypothesis $xa+yb+zc\leq \ba_{\pi_i}$ for all $(x,y,z)\in\mathbb{M}(\pi_i)$ in Theorem~\ref{thm:sumslices} is required because it is a priori possible that there exists and $i$-folded alcove path $p$ with $\exp(\cQ_i(p))=(x,y,z)$ and $xa+yb+zc>\ba_{\pi_i}$. The leading contributions from all such paths in Theorem~\ref{thm:pii} must cancel (after specialisation) for otherwise~$\B{2}$ is violated. While indeed cancellations can (and do) occur, it turns out that the condition $xa+yb+zc>\ba_{\pi_i}$ in fact never occurs. We will see this in the course of the calculations in the following sections.
\end{Rem}

We will use ``folding tables'' to analyse $i$-folded alcove paths ($i\in\{1,2,3\}$). We give a brief outline below, and we refer to \cite[\S 7.2]{GP:17} for further details. Let $v\in W^i_0$ and $x\in W$ with reduced expression $\vec{x}=s_{i_1}\ldots s_{i_n}$. We denote by $p(\vec{x},v)\in\cP_i(\vec{x},v)$ the unique $i$-folded alcove path of type $\vec{x}$ starting at $v$ with no folds. Of course $p(\vec{x},v)$ may still have bounces, because $i$-folded alcove paths are required to stay in the strip $\cU_i$. Nonetheless, we refer to $p(\vec{x},v)$ as the \textit{straight path} of type $\vec{x}$ starting at~$v$. Let
\begin{align*}
\mathcal{I}^-(\vec{x},v)&=\{k\in\{1,\ldots,n\}\mid \text{$p(\vec{x},v)$ makes a negative crossing at the $k$th step}\}\\
\mathcal{I}^+(\vec{x},v)&=\{k\in\{1,\ldots,n\}\mid \text{$p(\vec{x},v)$ makes a positive crossing at the $k$th step}\}\\
\mathcal{I}^{\,*}(\vec{x},v)&=\{k\in\{1,\ldots,n\}\mid \text{$p(\vec{x},v)$ bounces at the $k$th step}\}.
\end{align*}
Note that $\mathcal{I}^-\cup\mathcal{I}^+\cup\mathcal{I}^*=\{1,\ldots,n\}$. We define a function
$$
\varphi_{\vec{x}}^v:\mathcal{I}^-(\vec{x},v)\to W^i_0\times\mathbb{Z}
$$
as follows. For $k\in \mathcal{I}^-(\vec{x},v)$ let $p_k$ be the $i$-folded alcove path obtained from the straight path $p_0=p(\vec{x},v)$ by folding at the $k$th step (note that after performing this fold one may need to include bounces at places where the folded path $p_k$ attempts to exit the strip $\cU_i$). Let 
$$
\varphi_{\vec{x}}^v(k)=\text{the unique $(u,n)\in W_0^i\times\mathbb{Z}$ such that $p(\vec{x},\tau_i^nu)$ and $p_k$ agree after the $k$th step}.
$$
Equivalently, $(u,n)$ is the unique pair such that $\mathrm{end}(p(\vec{x},\tau_i^nu))=\mathrm{end}(p_k)$, and thus $\tau_i^nu$ is simply the end of the straight alcove path $p(\mathrm{rev}(\vec{x}),\mathrm{end}(p_k))$, where $\mathrm{rev}(\vec{x})$ is the expression $\vec{x}$ read backwards.  
\medskip

\begin{Def}[Folding tables, see {\cite{GP:17}}]
Fix the enumeration $y_1,y_2,y_3,y_4$ of $W_0^i$ with $\ell(y_j)=j-1$ for $j=1,\ldots,4$. For each $(j,k)$ with $1\leq j\leq 4$ and $1\leq k\leq \ell(x)$ define $f_{j,k}(\vec{x})\in\{-,*,1,2,3,4\}$ by
$$
f_{j,k}(\vec{x})=\begin{cases}
-&\text{if $k\in\mathcal{I}^+(\vec{x},y_j)$}\\
*&\text{if $k\in\mathcal{I}^{\,*}(\vec{x},y_j)$}\\
j'&\text{if $k\in\mathcal{I}^-(\vec{x},y_j)$ and $\varphi_{\vec{x}}^{y_j}=(y_{j'},n)$ for some $n\in\mathbb{Z}$.}
\end{cases}
$$
The \textit{$i$-folding table} of $\vec{x}$ is the $4\times \ell(x)$ array $\mathbb{F}_i(\vec{x})$ with $(j,k)^{th}$ entry equal to $f_{j,k}(\vec{x})$.
\end{Def}

 \begin{Rem}\label{prefix}
If $\vec{y}$ is a prefix of $\vec{y}$ then $\mathbb{F}_i(\vec{y})$ is the subarray of $\mathbb{F}(\vec{x})$ consisting of the first $\ell(y)$ columns. Also note that of course any other enumeration of $W_0^i$ can be used in the definition. 
\end{Rem}

\begin{Exa}\label{ex:ft}
Let $t_i=t_{\omega_i}$ for $i=1,2$. The $2$-folding tables for the elements $\vec{t}_1=0121$ and $\vec{t}_2=010212$ are shown in Table~\ref{tab:fold2}, where the rows are indexed by $W_0^2$ in the order $e,1,12,121$, and the $\vec{t}_2$ table excludes the final column. Note that we have appended a $0$-row and $0$-column to the table for convenience. The $0$-row is called the ``header'' of the table. The folding tables for the elements $v\in\sB_0$ are also given by these tables, because the reduced expressions for the elements of $\sB_0$ are the strict prefixes of $\vec{t}_2$, along with $010210$ (which is given in the $\vec{t}_2$ table by removing the penultimate column), along with $012$ (which is a subexpression of $\vec{t}_1$).
\begin{table}[H]
\renewcommand{\arraystretch}{1.2}  
\begin{subfigure}{.5\linewidth}
\centering
$\begin{array}{|c||c|c|c|c|}\hline
&0&1&2&1 \\\hline\hline
1&-&-&-&- \\\hline
2&*&-&*&1 \\\hline
3&*&1&*&- \\\hline
4&1&2&1&3 \\\hline
\end{array}$
\caption{$\vec{t}_1=0121$}
\end{subfigure}
\begin{subfigure}{.5\linewidth}
   \centering
   $\begin{array}{|c||c|c|c|c|c|c||c|}
   \hline
   & 0 &1 &0 &2&1&2 &0\\\hline
   \hline
1&-&-&*&-&-&*&-\\\hline
2&*&-&-&*&-&-&*\\\hline
3&*&1&2&*&1&2&*\\\hline
4&1&2&*&1&2&*&1\\\hline
\end{array}$
\caption{$\vec{t}_2=010212$ and $\vec{\mathsf{b}}_0=010210$}
\end{subfigure}
\caption{$2$-folding tables}
\label{tab:fold2}
\end{table}
\end{Exa}

The folding table $\mathbb{F}_i(\vec{w})$ can be used to compute $\cQ_i(p)$ for all $p\in\cP_i(\vec{w},u)$ with $u\in W_0^i$ as follows (see \cite{GP:17} for more details). We begin an excursion through the table $\mathbb{F}_i(\vec{w})$ starting at the first cell on row $\ell(u)+1$ (the row corresponding to~$u\in W_0^i$) with a counter $Z$ starting at $Z=1$. At each step we move to a cell strictly to the right of the current cell and modify $Z$ according to the following rules. Suppose we are currently at the $N^{th}$ cell of row~$r$, and this cell contains the symbol $x\in\{-,*,1,2,3,4\}$. Let $j\in\{0,1,2\}$ denote the header entry of the $N^{th}$ column.
 \begin{enumerate}
\item If $x=-$ then we move to the $(N+1)^{st}$ cell of row $r$ and $Z$ remains unchanged.
 \item If $x=*$ then we move to the $(N+1)^{st}$ cell of row $r$ and replace $Z$ by $Z'$ where
 $$
Z'=\begin{cases} Z\times(-\sq_j^{-1})&\text{if either $i\in\{1,2\}$, or if $i=3$ and $j=2$}\\
Z\times \sq_0&\text{if $i=3$ and $j=0$}.
\end{cases}
 $$
 \item If $x=k\in\{1,2,3,4\}$ then we have two options:
 \begin{enumerate}
 \item we can move to the $(N+1)^{st}$ cell of row $r$ and leave $Z$ unchanged, or
 \item we can move to the $(N+1)^{st}$ cell of row $k$ and replace $Z$ by $Z\times(\sq_j-\sq_j^{-1})$. 
  \end{enumerate}
  \end{enumerate}
  The set of all such excursions through the table is naturally in bijection with the set of $i$-folded alcove paths $\mathcal{P}_i(\vec{w},u)$, and the final value of the counter $Z$ at the end of the excursion is $\cQ_i(p)$. Moreover, the final exiting row gives the value of $\theta^i(p)$. It may help to note that cases (1), (2), (3)(a) and (3)(b) correspond to a positive crossing, bounce, negative crossing, and fold respectively. 
\medskip

Suppose that $\vec{w}=\vec{t}_{\omega_1}^{\,m}\cdot \vec{t}_{\omega_2}^{\,n}$ where $m,n\in\mathbb{N}$, and let $u\in W_0^i$. Then $\mathbb{F}_i(\vec{w})$ is the concatenation of $m$ copies of the $i$-folding table of $\vec{t}_{\omega_1}$ followed by $n$ copies of the $i$-folding table of $\vec{t}_{\omega_2}$ (for this observation to hold it is important that $t_{\omega_1}$ and $t_{\omega_2}$ are translations). Thus the process described above may be regarded as ``$m$ passes through the $\vec{t}_{\omega_1}$ table, followed by $n$ passes through the $\vec{t}_{\omega_2}$ table'' in an obvious way.

\subsection{The cell $\Gamma_2$}\label{sec:Ga2}

The cell $\Gamma_2$ is stable on each of the following regions:
\begin{align*}
R_{1}&=\{(r_1,r_2)\in\mathbb{Q}_{>0}^2\mid r_2<r_1,\,r_2<2-r_1\}&
R_{2}&=\{(r_1,r_2)\in\mathbb{Q}_{>0}^2\mid r_2<r_1,\,r_2>2-r_1\}\\
R_3=R_{1,2}&=\{(r_1,r_2)\in\mathbb{Q}_{>0}^2\mid r_2<r_1,\,r_2=2-r_1\}&
R_4=R_{1,1'}&=\{(r_1,r_2)\in\mathbb{Q}_{>0}^2\mid r_2=r_1,\,r_2<2-r_1\}\\
R_5=R_{2,2'}&=\{(r_1,r_2)\in\mathbb{Q}_{>0}^2\mid r_2=r_1,\,r_2>2-r_1\}&
R_6=P_2&=\{(r_1,r_2)\in\mathbb{Q}_{>0}^2\mid r_2=r_1,\,r_2=2-r_1\}.
\end{align*}
To explain this notation, notice that $R_1$ and $R_2$ are open regions, and $R_{1,2}$ is the boarder between these regions. Moreover, $R_{1,1'}$ is the boarder between the regions $R_1$ and $R_{1'}$, where $R_{j'}$ denotes the $\sigma$-dual of $R_j$. Similarly $R_{2,2'}$ if the boarder between the regions $R_{2}$ and~$R_{2'}$. 
\medskip

We begin by describing the cell $\Gamma_2$ in each of the above regions and setting up notation for the statement of the main theorems. In order to obtain the nicest leading matrices possible, we will need, in each case, to adjust the fundamental domains and the bases of $\cM_2$. To do so, we define fundamental domains as ordered sets instead of sets.

\medskip

The cases $(r_1,r_2)\in R_{j}$ with $j=1,2$ are ``generic'', and admit cell factorisations where
\begin{align*}
\sw_j=\begin{cases}
101&\text{if $j=1$,}\\
2&\text{if $j=2$,}
\end{cases}
\quad
\st_j=\begin{cases}
2101&\text{if $j=1$,}\\
1012&\text{if $j=2$,}
\end{cases}\quad\text{and}\quad
\sB_j=\begin{cases}
(e,2,21,210)&\text{if $j=1$,}\\
(e,1,10,101)&\text{if $j=2$}.
\end{cases}
\end{align*}
For each $j=1,2$ let $z_j\in \sB_j$ be such that $\sB_j'=\{z_j^{-1}u\mid u\in\sB_j\}$ is a fundamental domain for the action of $\tau_2$ on $\cU_2$ with $z_j^{-1}e$ on the negative side of each hyperplane separating $z_j^{-1}e$ from $z_j^{-1}u$ with $u\in\sB_j$. Specifically, $z_1=21$ and $z_2=1$. Let $\cB'_j=(\xi_2\otimes X_{z_j^{-1}u}\mid u\in \sB'_j)$ be the basis associated to the fundamental domain $\sB'_j$. Thus
\begin{align*}
\cB'_j&=\begin{cases}
(\xi_2\otimes X_{12},\,\xi_2\otimes X_{1},\,\xi_2\otimes X_{e},\,\xi_2\otimes X_{0})&\text{if $j=1$}\\
(\xi_2\otimes X_{1},\,\xi_2\otimes X_{e},\,\xi_2\otimes X_{0},\,\xi_2\otimes X_{01})&\text{if $j=2$}.
\end{cases}
\end{align*}
The fundamental domain $\sB_1'$ is depicted in the third example in Figure~\ref{fig:paths}. 

\medskip

The region $R_3=R_{1,2}$ is ``non-generic'', and does not admit a cell factorisation. However we have
\begin{align*}
\Gamma_2(R_{1,2})=\Gamma_2(R_2)\cup\{\sw_1\}.
\end{align*} 
Thus we can use cell factorisation in $\Gamma_2(R_2)$ to describe all elements of $\Gamma_2(R_{1,2})\backslash\{\sw_1\}$, and hence the expressions $\su_w$, $\sv_w$, and $\tau_w$ are defined for $w\in \Gamma_2(R_{1,2})\backslash\{\sw_1\}$. We extend this definition by setting 
$$
\su_{\sw_1}=\sv_{\sw_1}=101\quad\text{and}\quad \tau_{\sw_1}=-1.
$$ 
We set $\sB_3=\sB_2$, $\sB'_3=\sB'_2$ and $\cB'_3=\cB'_2$. 

\medskip

The regions $R_4=R_{1,1'}$ and $R_5=R_{2,2'}$ may be considered ``generic'' in a certain sense. Indeed these cases admit a generalised cell factorisation using the extended affine Weyl group. We have
$$
\Gamma_2(R_{j})=W\cap \{u^{-1}\sw_j\st_j^kv\mid u,v\in\sB_j,\,k\geq 0\}
$$
where
\begin{align*}
\sw_j=\begin{cases}
121&\text{if $j=4$}\\
0&\text{if $j=5$}
\end{cases}
\quad
\st_j=\begin{cases}
01\sigma&\text{if $j=4$}\\
12\sigma&\text{if $j=5$}
\end{cases}\quad\text{and}\quad
\sB_j=\begin{cases}
(e,0,\sigma,0\sigma)&\text{if $j=4$}\\
(e,1,\sigma,1\sigma)&\text{if $j=5$}.
\end{cases}
\end{align*}
If $w=u^{-1}\sw_j\st_j^kv$ with $u,v\in\sB_j$ and $k\geq 0$ we write $\su_w=u$, $\sv_w=v$ and $\tau_w=k$. Let $\cB'_4$ and $\cB'_5$ be the bases of $\cM_2$ associated to the fundamental domain $\sB'_4=(e,0,\sigma,0\sigma)$ and $\sB'_5=(1\si,\si,1,e)$. Thus
\begin{align*}
\cB'_j&=\begin{cases}
(\xi_2\otimes X_e,\,\xi_2\otimes X_0,\,\xi_2\otimes X_{\sigma},\,\xi_2\otimes X_{0\sigma})&\text{if $j=4$}\\
(\xi_2\otimes X_{1\sigma},\,\xi_2\otimes X_{\sigma},\,\xi_2\otimes X_{1},\,\xi_2\otimes X_{e})&\text{if $j=5$}.
\end{cases}
\end{align*}

Finally, the region $R_6=P_2$ is truely ``non-generic'', and exhibits rather remarkable behaviour. Every element of $\Ga_2\backslash\{2,12,212,010\}$ can be written in the form $w=ut_{\omega_1}^kv$ with $k\geq 0$ and $u\in\{e,1,21,121\}$ and $v\in\{e,0,01,012\}$, and moreover every element of this form with the exception of $e=et_{\omega_1}^0e$ lies in $\Ga_2$. The following indexing of the elements of $\Ga_2$ will help with the statement of the main theorem. Let $(u_i)=(e,1,21,121)$ and $(v_j)=(e,0,01,012)$. Then, for $k\geq 0$, we define
$
w_{i,j}^k=u_it_{\omega_1}^kv_j$ for all $(i,j)\notin\{(1,1),(1,2),(2,4)\}$, and
\begin{align*}
w_{1,1}^k=u_1t_{\omega_1}^{k+1}v_1,\quad w_{1,2}^k=u_1t_{\omega_1}^{k+1}v_2\quad\text{and}\quad 
w_{2,4}^k=\begin{cases}
12&\text{if $k=0$}\\
u_2t_{\omega_1}^{k-1}v_4&\text{if $k\geq 1$.}
\end{cases}
\end{align*}
Then 
$$
\Gamma_2=\{w_{i,j}^k\mid 1\leq i,j\leq 4,\,k\geq 0\}\cup\{0,2,212,010\}.
$$

The main theorems of this section are the following three results. In the first two theorems, the elementary matrix $E_{u,v}$ ($u,v\in \sB_j$) denotes the matrix with $1$ at position $(k,\ell)$ where $k,\ell$ is the position of $u$ and $v$ in the ordered set $\sB_j$, and $0$ everywhere else. 

\begin{Th}\label{thm:gamma2part1}
Let $(r_1,r_2)\in R_j$, with $j=1,2,3$. Then $\pi_2$, equipped with the basis $\cB'_j$, satisfies $\B{1}$--$\B{5}$ for the cell $\Ga_2=\Gamma_2(r_1,r_2)$, with $\ba_{\pi_2}=\tilde{\ba}({\Gamma_2})$. Moreover, for $j=1,2$ the leading matrices of $\pi_2$ are
$$
\fc_{\pi_2}(w;\sB'_j)=(-1)^j\fs_{\tau_w}(\zeta)E_{\su_w,\sv_w}\quad\text{for $w\in \Gamma_2$}
$$
where $\fs_k(\zeta)$ is the Schur function of type $A_1$. In the case $j=3$ we have, for $w\in \Gamma_2$,
\begin{align*}
\fc_{\pi_2}(w;\sB'_3)&=\ff_{\tau_w}^{\su_w,\sv_w}(\zeta)E_{\su_w,\sv_w}&\text{where}&&\ff_k^{u,v}(\zeta)&=
\begin{cases}
\fs_k(\zeta)-\fs_{k-1}(\zeta)&\text{if $u,v\neq 101$}\\
\fs_k(\zeta)-\fs_{k-1}(\zeta)-\zeta^{-k-1}&\text{if $u\neq 101$ and $v=101$}\\
\fs_k(\zeta)-\fs_{k-1}(\zeta)-\zeta^{k+1}&\text{if $u=101$ and $v\neq 101$}\\
\fs_k(\zeta)-\fs_{k+1}(\zeta)&\text{if $u=v=101$},
\end{cases}
\end{align*}
where $\fs_k(\zeta)$ is the Schur function of type $A_1$ and we set $\fs_{-1}(\zeta)=0$. 
\end{Th}

\begin{Th}\label{thm:gamma2part2}
Let $(r_1,r_2)\in R_j$ with $j=4,5$. Then $\pi_2$, equipped with the basis $\cB'_j$, satisfies $\B{1}$--$\B{5}$ for the cell $\Ga_2=\Gamma_2(r_1,r_2)$, with $\ba_{\pi_2}=\tilde{\ba}({\Gamma_2})$. Moreover the leading matrices of $\pi_2$ are, for $w\in \Gamma_2$,
$$
\fc_{\pi_2}(w;\sB'_j)=(-1)^{j+1}\fs_{\tau_w}(\zeta^{1/2})E_{\su_w,\sv_w}.
$$
\end{Th}

\begin{Th}\label{thm:gamma2part3}
Let $(r_1,r_2)\in R_6$. Then $\tilde{\pi}_2=\pi_2\oplus\pi_5\oplus \pi_6$, equipped with the standard $W_0^2$-basis for the $\pi_2$ component, satisfies $\B{1}$--$\B{5}$ for the cell $\Gamma_2=\Gamma_2(R_6)$. Moreover, the leading matrices are as follows (for $k\geq 0$):
\begin{align*}
\fc(w_{11}^k)&=\zeta^kE_{41}+\zeta^{-k-1}E_{43}&
\fc(w_{12}^k)&=(\zeta^{k+1}+\zeta^{-k-1})E_{44}&
\fc(w_{13}^k)&=\zeta^{-k-1}E_{41}+\zeta^kE_{43}\\
\fc(w_{14}^k)&=(\zeta^k+\zeta^{-k-1})E_{42}&
\fc(w_{21}^k)&=\zeta^kE_{11}+\zeta^{-k}E_{33}&
\fc(w_{22}^k)&=\zeta^{k+1}E_{14}+\zeta^{-k}E_{34}\\
\fc(w_{23}^k)&=\zeta^{k+1}E_{13}+\zeta^{-k-1}E_{31}&
\fc(w_{24}^k)&=\zeta^kE_{12}+\zeta^{-k}E_{32}&
\fc(w_{31}^k)&=\zeta^kE_{21}+\zeta^{-k}E_{23}\\
\fc(w_{32}^k)&=(\zeta^{k+1}+\zeta^{-k})E_{24}&
\fc(w_{33}^k)&=\zeta^{-k-1}E_{21}+\zeta^{k+1}E_{23}&
\fc(w_{34}^k)&=(\zeta^{k+1}+\zeta^{-k-1})E_{22}\\
\fc(w_{41}^k)&=\zeta^{-k}E_{13}+\zeta^kE_{31}&
\fc(w_{42}^k)&=\zeta^{-k}E_{14}+\zeta^{k+1}E_{34}&
\fc(w_{43}^k)&=\zeta^{-k-1}E_{11}+\zeta^{k+1}E_{33}\\
\fc(w_{44}^k)&=\zeta^{-k-1}E_{22}+\zeta^{k+1}E_{32}&
\fc({2})&=E_{22}+E_{66}&
\fc({0})&=E_{44}+E_{55}\\
\fc({212})&=E_{22}-E_{66}&
\fc({010})&=E_{44}-E_{55}.
\end{align*}
\end{Th}

The proof of the above theorems will be given towards the end of this section. We first analyse the slices of the matrices $\pi_2(T_w;W_0^2)$. This in turn requires, by Theorem~\ref{thm:pii}, a rather detailed analysis of $2$-folded alcove paths. Each $w\in W$ can be written uniquely as 
$$
w=ut_2^mt_1^nv\quad\text{with $u\in W_0$, $v\in \sB_0$, and $m,n\in\mathbb{Z}$}
$$
(where we write $t_1=t_{\omega_1}$ and $t_2=t_{\omega_2}$) and necessarily $\ell(w)=\ell(u)+n\ell(t_1)+m\ell(t_2)+\ell(v)$. We choose and fix the reduced expressions for each $u\in W_0$, $v\in\sB_0$, and $t_1,t_2$, which are lexicographically minimal. Thus 
$$
\vec{w}_0=1212,\quad \vec{t}_1=0121,\quad\text{and}\quad \vec{t}_2=010212,
$$
and the expressions $\vec{v}$ for $v\in\sB_0$ are the prefixes of $\vec{\mathsf{b}}_0=010210$, along with the element $012$ (see Example~\ref{ex:ft}). These choices give a distinguished reduced expression for each $w\in W$, namely
$$
\vec{w}=\vec{u}\cdot \vec{t}_2^{\,m}\cdot\vec{t}_1^{\,n}\cdot\vec{v}
$$
with the reduced expressions for each component chosen as above. We fix this choice throughout this section. If $p$ is a path of type $\vec{w}=\vec{u}\cdot \vec{t}_2^{\,m}\cdot\vec{t}_1^{\,n}\cdot\vec{v}$ we write 
\begin{align}\label{eq:decompose}
p=p_0\cdot p^0\quad\text{where $p_0$ is of type $\vec{u}$ and $p^0$ is of type $\vec{t}_2^{\,m}\cdot\vec{t}_1^{\,n}\cdot\vec{v}$}.
\end{align}
To efficiently record $2$-folded alcove paths we will use the notation $\hat{i}$ to denote an $i$-fold, and $\check{i}$ to denote an $i$-bounce. Thus, for example, $p=2\hat{1}\check{0}\hat{1}21$ is a $2$-folded alcove path whose second and fourth steps are $1$-folds, and whose third step is a $0$-bounce (for example, this is a valid $2$-folded alcove path starting at $1$). 
\medskip

The main theorems of this section will follow from the following combinatorial theorem.

\begin{Th}\label{thm:Gamma2Paths}
Let $p$ be an $2$-folded alcove path of type $\vec{w}=\vec{u}\cdot \vec{t}_2^{\,\ell}\cdot \vec{t}_1^{\,k}\cdot \vec{v}$ with $u\in W_0$, $v\in\sB_0$ and $k,\ell\geq 0$ starting at $u_0\in W_0^2$. Then $\exp(\cQ_2(p))\preceq \mathbf{x}$ for some $\mathbf{x}\in \{(1,0,0),(0,1,0),(0,0,1),(2,-1,0),(2,0,-1)\}$. Moreover, the paths $p$ with $\exp(\cQ_2(p))=\mathbf{x}$ for some $\mathbf{x}\in\{(1,0,0),(0,1,0),(0,0,1),(2,-1,0),(2,0,-1)\}$ are precisely the paths $p=p_0\cdot p^0$ with $\mathrm{end}(p_0)=\mathrm{start}(p^0)$ where $p_0$ is listed in Table~\ref{tab:pi2paths0} and $p^0$ is listed in Table~\ref{tab:pi2paths}.
\end{Th}

\begin{proof}
Write $p=p_0\cdot p^0$ as in~(\ref{eq:decompose}). We claim that:
\begin{enumerate}
\item if $\mathrm{end}(p_0)=e$ then $\exp(\cQ_2(p_0))\preceq\mathbf{x}$ for some $\mathbf{x}\in \{(1,0,0),(0,1,0),(2,-1,0)\}$;
 \item if $\mathrm{end}(p_0)=1$ then $\exp(\cQ_2(p_0))\preceq\mathbf{x}$ for some $\mathbf{x}\in \{(1,0,0),(0,1,0)\}$;
 \item if $\mathrm{end}(p_0)=12$ then $\exp(\cQ_2(p_0))\preceq (1,0,0)$;
 \item if $\mathrm{end}(p_0)=121$ then $\exp(\cQ_2(p_0))\preceq(0,0,0)$,
\end{enumerate}
and moreover, the paths where equality is attained are listed in Table~\ref{tab:pi2paths0} (the $*$ in rows 17 and 18 will be explained later).
\medskip

\begin{table}[H]
$$\renewcommand{\arraystretch}{1.2}
\begin{array}{|c|c||c|c|c|c|}\hline
\text{row}&u&\mathrm{start}(p_0)&p_0&\exp(\cQ_2(p_0))&\mathrm{end}(p_0)\\\hline\hline
1&e&121&e&(0,0,0)&121\\\hline
2&1&e&\hat{1}&(1,0,0)&e\\\hline
3&1&12&\hat{1}&(1,0,0)&12\\\hline
4&1&12&1&(0,0,0)&121\\\hline
5&2&1&\hat{2}&(0,1,0)&1\\\hline
6&12&e&1\hat{2}&(0,1,0)&1\\\hline
7&12&12&\hat{1}2&(1,0,0)&1\\\hline
8&21&1&\hat{2}1&(0,1,0)&e\\\hline
9&21&1&2\hat{1}&(1,0,0)&12\\\hline
10&21&1&21&(0,0,0)&121\\\hline
11&121&e&1\hat{2}1&(0,1,0)&e\\\hline
12&121&e&\hat{1}\check{2}\hat{1}&(2,-1,0)&e\\\hline
13&121&e&12\hat{1}&(1,0,0)&12\\\hline
14&121&e&121&(0,0,0)&121\\\hline
15&121&12&\hat{1}21&(1,0,0)&e\\\hline
16&212&1&2\hat{1}2&(1,0,0)&1\\\hline
17^*&1212&e&12\hat{1}2&(1,0,0)&1\\\hline
18^*&1212&e&\hat{1}\check{2}1\hat{2}&(1,0,0)&1\\\hline
\end{array}
$$
\caption{Optimal $p_0$ parts}
\label{tab:pi2paths0}
\end{table}

To establish the claim we note that the paths listed obviously have the stated exponents. One now constructs all paths $p_0$ of type $\vec{u}$ with $u\in W_0$ starting at some $u_0\in \{e,1,12,121\}$, and verifies the claim directly. For example, the paths starting and ending at $e$ are precisely the following:
$$
e,\quad \hat{1},\quad \check{2},\quad \hat{1}\check{2},\quad \check{2}\hat{1},\quad 1\hat{2}1,\quad \hat{1}\check{2}\hat{1},\quad \check{2}\hat{1}\check{2},\quad 1\hat{2}1\check{2},\quad \hat{1}\check{2}\hat{1}\check{2}
$$
and each of these paths has exponent bounded by some element of $\{(1,0,0),(0,1,0),(2,-1,0)\}$ with equality in the second, sixth, and seventh cases (listed in rows $2$, $11$ and $12$ of the table). 

\medskip

Let $\mathbb{E}=\{(1,0,0),(0,1,0),(0,0,1),(2,-1,0),(2,0,-1)\}$. We claim that if $p=p_0\cdot p^0$ then $\exp(\cQ_2(p))\preceq\mathbf{x}$ for some $\mathbf{x}\in \mathbb{E}$, and moreover the paths attaining equality are precisely the concatenations of paths $p_0$ from Table~\ref{tab:pi2paths0} with paths $p^0$ in Table~\ref{tab:pi2paths} with $\mathrm{end}(p_0)=\mathrm{start}(p^0)$. The proof of this claim occupies the remainder of the proof. 
When combining two paths it is useful to note the obvious fact that if $\mathbf{x}'\preceq \mathbf{x}$ and $\mathbf{y}'\preceq \mathbf{y}$ then $\mathbf{x}'+\mathbf{y}'\preceq \mathbf{x}+\mathbf{y}$.

\begin{table}[h!]
$$
\renewcommand{\arraystretch}{1.1}
\begin{array}{|c|c||c|c|c|c|c|c|}\hline
\text{row}&\vec{x}&\mathrm{start}(p^0)&p^0&\exp(\cQ_2(p^0))&\mathrm{wt}^2(p^0)&\theta^2(p^0)&\text{conditions}\\\hline\hline
1&t_1^k&e&\vec{t}_1^{\,k}&(0,0,0)&k&e&k\geq1\\\hline
2&&12&\check{0}\hat{1}21\cdot \vec{t}_1^{\,k-1}&(1,0,-1)&k-1&e&k\geq 1\\\hline
3&&121&\vec{t}_1^{\,k-1}\cdot 012\hat{1}&(1,0,0)&-k&12&k\geq 1\\\hline
4&&121&\vec{t}_1^{\,m}\cdot \hat{0}121\cdot \vec{t}_1^{\,n}&(0,0,1)&n-m&e&m+n=k-1\geq 0\\\hline
5&&121&\vec{t}_1^{\,m}\cdot 0\hat{1}\check{2}\hat{1}\cdot \vec{t}_1^{\,n}&(2,-1,0)&n-m-1&e&m+n=k-1\geq 0\\\hline
6&&121&\vec{t}_1^{\,m}\cdot 01\hat{2}1\cdot \vec{t}_1^{\,n}&(0,1,0)&n-m-1&e&m+n=k-1\geq 0\\\hline
7&&121&\vec{t}_1^{\,m}\cdot 012\hat{1}\cdot \check{0}\hat{1}21\cdot \vec{t}_1^{\,n}&(2,0,-1)&n-m-1&e&m+n=k-2\geq 0\\\hline
8&t_1^k\cdot 0&e&\vec{t}_1^{\,k}\cdot 0&(0,0,0)&k+1&121&k\geq 1\\\hline
9&&12&\check{0}\hat{1}21\cdot \vec{t}_1^{\,k-1}\cdot 0&(1,0,-1)&k&121&k\geq 1\\\hline
10&&121&\vec{t}_1^{\,k}\cdot\hat{0}&(0,0,1)&-k&121&k\geq 1\\\hline
11&&121&\vec{t}_1^{\,m}\cdot \hat{0}121\cdot \vec{t}_1^{\,n}\cdot 0&(0,0,1)&n-m+1&121&m+n=k-1\geq 0\\\hline
12&&121&\vec{t}_1^{\,m}\cdot 0\hat{1}\check{2}\hat{1}\cdot \vec{t}_1^{\,n}\cdot 0&(2,-1,0)&n-m&121&m+n=k-1\geq 0\\\hline
13&&121&\vec{t}_1^{\,m}\cdot 01\hat{2}1\cdot \vec{t}_1^{\,n}\cdot 0&(0,1,0)&n-m&121&m+n=k-1\geq 0\\\hline
14&&121&\vec{t}_1^{\,m}\cdot 012\hat{1}\cdot \check{0}\hat{1}21\cdot \vec{t}_1^{\,n}\cdot 0&(2,0,-1)&n-m&121&m+n=k-2\geq 0\\\hline
15&t_1^k\cdot 01&e&\vec{t}_1^{\,k}\cdot 01&(0,0,0)&k+1&12&k\geq 1\\\hline
16&&12&\check{0}\hat{1}21\cdot \vec{t}_1^{\,k-1}\cdot 01&(1,0,-1)&k&12&k\geq 1\\\hline
17&&121&\vec{t}_1^{\,k}\cdot\hat{0}1&(0,0,1)&-k&12&k\geq 1\\\hline
18&&121&\vec{t}_1^{\,k}\cdot 0\hat{1}&(1,0,0)&-k-1&e&k\geq 1\\\hline
19&&121&\vec{t}_1^{\,k-1}\cdot 012\hat{1}\cdot\check{0}\hat{1}&(2,0,-1)&-k&12&k\geq 1\\\hline
20&&121&\vec{t}_1^{\,m}\cdot \hat{0}121\cdot \vec{t}_1^{\,n}\cdot 01&(0,0,1)&n-m+1&12&m+n=k-1\geq 0\\\hline
21&&121&\vec{t}_1^{\,m}\cdot 0\hat{1}\check{2}\hat{1}\cdot \vec{t}_1^{\,n}\cdot 01&(2,-1,0)&n-m&12&m+n=k-1\geq 0\\\hline
22&&121&\vec{t}_1^{\,m}\cdot 01\hat{2}1\cdot \vec{t}_1^{\,n}\cdot 01&(0,1,0)&n-m&12&m+n=k-1\geq 0\\\hline
23&&121&\vec{t}_1^{\,m}\cdot 012\hat{1}\cdot \check{0}\hat{1}21\cdot \vec{t}_1^{\,n}\cdot 01&(2,0,-1)&n-m&12&m+n=k-2\geq 0\\\hline
24&t_1^k\cdot 012&e&\vec{t}_1^{\,k}\cdot 012&(0,0,0)&k+1&1&k\geq 1\\\hline
25&&12&\check{0}\hat{1}21\cdot \vec{t}_1^{\,k-1}\cdot 012&(1,0,-1)&k&1&k\geq 1\\\hline
26&&121&\vec{t}_1^{\,k}\cdot 01\hat{2}&(0,1,0)&-k-1&1&k\geq 1\\\hline
27&&121&\vec{t}_1^{\,k-1}\cdot 012\hat{1}\cdot \check{0}\hat{1}2&(2,0,-1)&-k&1&k\geq 1\\\hline
28&&121&\vec{t}_1^{\,m}\cdot \hat{0}121\cdot \vec{t}_1^{\,n}\cdot 012&(0,0,1)&n-m+1&1&m+n=k-1\geq 0\\\hline
29&&121&\vec{t}_1^{\,m}\cdot 0\hat{1}\check{2}\hat{1}\cdot \vec{t}_1^{\,n}\cdot 012&(2,-1,0)&n-m&1&m+n=k-1\geq 0\\\hline
30&&121&\vec{t}_1^{\,m}\cdot 01\hat{2}1\cdot \vec{t}_1^{\,n}\cdot 012&(0,1,0)&n-m&1&m+n=k-1\geq 0\\\hline
31&&121&\vec{t}_1^{\,m}\cdot 012\hat{1}\cdot \check{0}\hat{1}21\cdot \vec{t}_1^{\,n}\cdot 012&(2,0,-1)&n-m&1&m+n=k-2\geq 0\\\hline
32^*&t_1^k\cdot 010&121 &\vec{t}_1^{\,k}\cdot 0\hat{1}0 &(1,0,0) &-k &121 &k\geq 1 \\\hline
33^*&&121 &\vec{t}_1^{\,k-1}\cdot 012\hat{1}\cdot \check{0}1\hat{0} &(1,0,0) &-k-1 &121 &k\geq 0 \\\hline
34&e&e &e &(0,0,0) &0 &e & \\\hline
35&&1 &e &(0,0,0) &0 &1 & \\\hline
36&&12 &e &(0,0,0) &0 &12 & \\\hline
37&0&e &0 &(0,0,0) &1 &121 & \\\hline
38&&121 &\hat{0} &(0,0,1) &0 &121 & \\\hline
39&01&e &01 &(0,0,0) &1 &12 & \\\hline
40&&12 &\check{0}\hat{1} &(1,0,-1) &0 &12 & \\\hline
41&&121 &\hat{0}1 &(0,0,1) &0 &12 & \\\hline
42&&121 &0\hat{1} &(1,0,0) &-1 &e & \\\hline
43&012&e &012 &(0,0,0) &1 &1 & \\\hline
44&&12 &\check{0}\hat{1}2 &(1,0,-1) &0 &1 & \\\hline
45&&121 &01\hat{2} &(0,1,0) &-1 &12 & \\\hline
46&010&121 &0\hat{1}0 &(1,0,0) &0 &121 & \\\hline
47&&12 &\check{0}1\hat{0} &(0,0,0) &0 &121 & \\\hline
\end{array}
$$
\caption{Optimal $p^0$ parts}
\label{tab:pi2paths}
\end{table}
\medskip

The folding tables for the elements $\vec{t}_1=0121$ and $\vec{t}_2=010212$ are shown in Table~\ref{tab:fold2}. The following observation will be used frequently: If a pass of either the $\vec{t}_1$ or $\vec{t}_2$ table is completed on a row containing at least one $*$, and if no folds are made in this pass, then
\begin{align}\label{eq:suboptimal}
\exp(\cQ_2(p^0))\prec \exp(\cQ_2(p')),
\end{align}
where $p'$ is the path obtained from $p^0$ by removing this copy of $\vec{t}_1$ or $\vec{t}_2$. Thus such paths necessarily have strictly dominated exponents, and can therefore be discarded in the following analysis. 
\medskip

The claim follows from the following four points.
\begin{enumerate}
\item Suppose that $\mathrm{start}(p^0)=e$. Since every entry of the first row of the $2$-folding table for $\vec{t}_1$ is $-$, and since every entry of the $2$-folding table of $\vec{t}_2$ is either $-$ or $*$, it is clear that $\exp(\cQ_2(p^0))\preceq (0,0,0)$. Thus, combined with the paths from Table~\ref{tab:pi2paths0} we have $\exp(\cQ_2(p))\preceq \mathbf{x}$ for some $\mathbf{x}\in\{(1,0,0),(0,1,0),(2,-1,0)\}\subset\mathbb{E}$. Moreover we have equality if and only if the $p^0$ part has no bounce, and therefore equality holds if and only if $\ell=0$ and $v\in\{e,0,01,012\}$, giving the paths
$$
p^0=\vec{t}_1^{\,k}\cdot \vec{v}\quad\text{for some $k\geq 0$ and $v\in \{e,0,01,012\}$}.
$$
These paths are listed on rows $1/34$, $8/37$, $15/39$ and $24/43$ of Table~\ref{tab:pi2paths} (it is convenient to separate the cases $k>0$ and $k=0$, and this is indicated by the notation $i/j$ for the table rows).
\item Suppose that $\mathrm{start}(p^0)=1$. Writing $p^0=p_1\cdot p_2$ where $p_1$ is of type ${\vec{t}_2}^{\,\ell}$ and $p_1$ is of type $\vec{t}_1^{\,k}\cdot \vec{v}$, we have 
$$
\exp(\cQ_2(p^0))=(0,-\ell,-\ell)+\exp(\cQ_2(p_2)).
$$ 
It is clear that $\exp(\cQ_2(p_2))\preceq (1,-1,-1)$ or $\exp(\cQ_2(p_2))\preceq(0,0,0)$ (depending on whether $k>0$ and the possible fold on the $4$th place of $\vec{t}_1$ is taken). Therefore
$$
\exp(\cQ_2(p^0))\preceq (1,-\ell-1,-\ell-1)\quad\text{or}\quad \exp(\cQ_2(p^0))\preceq (0,-\ell,-\ell). 
$$
Thus $\exp(\cQ_2(p^0))\preceq (1,-1,-1)$ or $\exp(\cQ_2(p^0))\preceq (0,0,0)$. In the first case, combining the contribution from $p_0$ we have $\exp(\cQ_2(p))\preceq (2,-1,-1)\prec (2,-1,0)\in\mathbb{E}$ or $\exp(\cQ_2(p))\preceq (1,0,-1)\prec(2,0,-1)\in\mathbb{E}$, and so the combined path is sub-optimal. In the second case we have $\exp(\cQ_2(p))\preceq (1,0,0)\in\mathbb{E}$ or $\exp(\cQ_2(p))\preceq (0,1,0)\in\mathbb{E}$, with equality if and only if $k=\ell=0$ and $v=e$. This (trivial) path is listed on row $35$ of Table~\ref{tab:pi2paths}.

\item Suppose that $\mathrm{start}(p^0)=12$. We first claim that if $\ell>0$ then $\exp(\cQ_2(p))\prec\mathbf{x}$ for some $\mathbf{x}\in\mathbb{E}$. By the observation made in~(\ref{eq:suboptimal}) it suffices to assume that if $\ell>0$ then a fold is made in the first pass of the $\vec{t}_2$ table. Thus the first part of the path is one of the following:
\begin{align*}
&p_a=\check{0}\hat{1}\check{0}21\check{2}&&p_b=\check{0}1\hat{0}\check{2}12&&p_c=\check{0}10\check{2}\hat{1}\check{2}&&p_d=\check{0}10\check{2}1\hat{2},
\end{align*}
with exponents $(1,-1,-2)$, $(0,-1,0)$, $(1,-2,-1)$, and $(0,0,-1)$ respectively. The paths $p_a$ and $p_c$ exit on row~$1$ of the $2$-folding table, and paths $p_b$ and $p_c$ exit on row~$2$ of the table. Since no positive contributions occur on the first row of any of the tables we have (using the first claim) $\exp(\cQ_2(p))\preceq (1,0,0)+(1,-1,-2)=(2,-1,-2)\prec(2,-1,0)\in\mathbb{E}$ in the case $p_a$, and $\exp(\cQ_2(p^0))\preceq (1,-2,-1)+(1,0,0)=(2,-2,-1)\prec(2,0,-1)\in\mathbb{E}$ in the case $p_c$. The only possible positive contribution on row $2$ of the folding tables comes from the $1$-fold in the $\vec{t}_1$ table, however accessing this fold comes at the cost of both a $0$-bounce and a $2$-bounce. Hence in case $p_b$ we have either $\exp(\cQ_2(p))\preceq(0,-1,0)+(1,0,0)\prec(1,0,0)\in\mathbb{E}$ or $\exp(\cQ_2(p))\preceq (1,-2,-1)+(1,0,0)\prec (2,0,-1)\in\mathbb{E}$, and in case $p_d$ we have either $\exp(\cQ_2(p))\preceq (0,0,-1)+(1,0,0)\prec(1,0,0)\in\mathbb{E}$ or $\exp(\cQ_2(p))\preceq (1,-1,-2)+(1,0,0)\prec(2,-1,0)\in\mathbb{E}$. This establishes the claim.
\medskip

Thus we may assume that $\ell=0$, and so $p^0$ has type $\vec{t}_1^{\,k}\cdot \vec{v}$ for some $k\geq 0$ and some $v\in \sB_0$. If $k>0$ then by the observation above we may assume that a fold is made in the first pass of the $\vec{t}_1$ table. Thus the first part of the path is necessarily $\check{0}\hat{1}21$, which has exponent $(1,0,-1)$ and exits on row~$1$ of the folding table. Any further $\vec{t}_1$ factors will have no effect on the exponent, and the final $\vec{v}$ factor can have contribution at most $(0,0,0)$, and this occurs if and only if $v\in\{e,0,01,012\}$. Thus the paths
$$
p^0= \check{0}\hat{1}21\cdot \vec{t}_1^{\,n}\cdot \vec{v}\quad\text{for $n\geq 0$ and $v\in\{e,0,01,012\}$}
 $$
 (starting at $21$) all have exponent precisely $(1,0,-1)$, and when composed with an optimal $p_0$ path we have $\exp(\cQ_2(p))=(2,0,-1)\in\mathbb{E}$. These paths are listed on rows $2$, $9$, $16$ and $25$ of Table~\ref{tab:pi2paths}. 
 \medskip
 
 If $k=0$ then $p^0$ has type $\vec{v}$ for some $v\in\sB_0$. By direct observation these paths have exponent bounded by either $(0,0,0)$ or $(1,0,-1)$. The only paths with exponent $(0,0,0)$ are the empty path $e$ and the path $\check{0}1\hat{0}$, and the paths with exponent $(1,0,-1)$ are precisely $\check{0}\hat{1}$ and $\check{0}\hat{1}2$. Appended with an optimal $p_0$ path we therefore obtain paths with exponents $(1,0,0)$ and $(2,0,-1)$. These paths are listed on rows $36$, $47$, $40$, and $44$ of Table~\ref{tab:pi2paths}. 

\item Suppose that $\mathrm{start}(p^0)=121$. A very similar argument to the case $\mathrm{start}(p^0)=12$ shows that if $\ell>0$ then $\exp(\cQ_2(p))\prec \mathbf{x}$ for some $\mathbf{x}\in\mathbb{E}$. Thus we may assume that $\ell=0$. Thus $p^0$ has type $\vec{t}_1^{\,k}\cdot \vec{v}$ for some $k\geq 0$ and $v\in \sB_0$. Since the $4$th row of the $2$-folding table of $\vec{t}_1$ contains no bounces, one may begin by making any number $k_1\leq k$ passes through the folding table with no folds. 
\medskip

If $k_1=k$ then the exponent of $p^0$ is equal to the exponent of the $\vec{v}$ part of $p^0$. The possible paths of type $\vec{v}$, $v\in\sB_0$, starting on row~$4$ are as follows:
\begin{align*}
&e&&0&&\hat{0}&&01&&\hat{0}1&&0\hat{1}&&01\check{0}&&\hat{0}1\check{0}&&0\hat{1}0&&01\check{0}2\\
&\hat{0}1\check{0}2&&0\hat{1}0\check{2}&&01\check{0}\hat{2}&&01\check{0}21&&\hat{0}1\check{0}21&&0\hat{1}0\check{2}1&&01\check{0}\hat{2}1&&01\check{0}2\hat{1}&&01\check{0}210&&\hat{0}1\check{0}210\\
&0\hat{1}0\check{2}1\check{0}&&01\check{0}\hat{2}10&&01\check{0}2\hat{1}\check{0}&&01\check{0}21\hat{0}&&012&&\hat{0}12&&0\hat{1}\check{2}&&01\hat{2}.
\end{align*}
Thus when appended with an optimal $p_0$ part we have $\exp(\cQ_2(p))=(0,0,0)+\exp(\cQ_2(p^0))\preceq \mathbf{x}$ for some $\mathbf{x}\in\mathbb{E}$, with equality precisely in the following cases of $p^0$:
\begin{align*}
&\vec{t}_1^{\,k_1}\cdot \hat{0},\quad\vec{t}_1^{\,k_1}\cdot \hat{0}1,\quad \vec{t}_1^{\,k_1}\cdot 0\hat{1},\quad \vec{t}_1^{\,k_1}\cdot 0\hat{1}0,\quad\text{and}\quad \vec{t}_1^{\,k_1}\cdot 01\hat{2}.
\end{align*}
These paths are listed in rows $10/38$, $17/41$, $18/42$, $32^*/46$, and $26/45$ of Table~\ref{tab:pi2paths} (the $*$ will be explained later in Remark \ref{rem:cancel}, and again it is convenient to split the $k_1=0$ and $k_1>0$ cases). 
\medskip

If $k_1<k$ then we assume that the $(k_1+1)$-st pass of the $\vec{t}_1$ table has a fold. The possibilities on this pass are
\begin{align*}
p_a=&\hat{0}121\quad\text{exponent $(0,0,1)$, exit row $1$,}\\
p_b=&0\hat{1}\check{2}\hat{1}\quad\text{exponent $(2,-1,0)$, exit row $1$,}\\
p_c=&01\hat{2}1\quad\text{exponent $(0,1,0)$, exit row $1$,}\\
p_d=&0\hat{1}\check{2}1\quad\text{exponent $(1,-1,0)$, exit row $2$,}\\
p_e=&012\hat{1}\quad\text{exponent $(1,0,0)$, exit row $3$.}
\end{align*}
The paths $p_a$, $p_b$, and $p_c$ exiting on row~$1$ can be followed by any number of $\vec{t}_1$ factors, and then an element $v\in\{e,0,01,012\}$ (any other elements $v\in\sB_0$ will decrease the exponent). Thus the paths
$$
\vec{t}_1^{\,k_1}\cdot p'\cdot \vec{t}_1^{\,k_2}\cdot \vec{v}\quad\text{with $k_1,k_2\geq 0$, $p'\in\{p_a,p_b,p_c\}$ and $v\in\{e,0,01,012\}$}
$$
have exponents $(0,0,1)$ for $p'=p_a$, $(2,-1,0)$ for $p'=p_b$, and $(0,1,0)$ for $p'=p_c$. These paths are listed in rows $4$, $11$, $20$, $28$ (for $p'=p_a$), $5$, $12$, $21$, $29$ (for $p'=p_b$), and $6$, $13$, $22$, $30$ (for $p'=p_c$). 
\medskip

Consider the path $p_d$. If $k_1+1<k$ then there are further passes through the $\vec{t}_1$ table, and by the observation above there must be a fold on the next pass. Thus $p^0$ starts with $\vec{t}_1^{\,k_1}\cdot p_d\cdot \check{0}1\check{2}\hat{1}$, which has exponent $(2,-2,-1)$ and exits on row $1$. Since $(2,-2,-1)\prec(2,0,-1)$ and no positive contributions can be obtained from row~$1$ it follows that in fact $k_1+1=k$. Thus $p^0$ is of the form $\vec{t}_{1}^{\,k-1}\cdot p_d\cdot p''$ for some path $p''$ of type $\vec{v}$ with $v\in\sB_0$. However it is clear that such a path has $\exp(\cQ_2(p^0))\prec (2,0,-1)$, and so $p_d$ does not lead to any optimal paths. 
\medskip

Consider the path $p_e$, which exits on row~$3$. Suppose that $k_1+1<k$. Applying the analysis of the $\mathrm{start}(p^0)=12$ case we see that
$$
p^0=\vec{t}_1^{\,k_1}\cdot p_e\cdot \check{0}\hat{1}21\cdot \vec{t}_1^{\,k_2}\cdot \vec{v}\quad\text{with $k_1,k_2\geq 0$ and $v\in\{e,0,01,012\}$}
$$
are the only paths with exponent $(1,0,0)+(1,0,-1)=(2,0,-1)$. These paths are listed in rows $7$, $14$, $23$ and $31$ of Table~\ref{tab:pi2paths}. If $k_1+1=k$ then the paths
$$
p^0=\vec{t}_1^{\,k_1}\cdot p_e\cdot p'\quad\text{with $p'\in\{\check{0}\hat{1},\check{0}\hat{1}2\}$}
$$
(listed in rows $19$ and $27$) are the only paths with exponent $(2,0,-1)$, and the paths
$$
p^0=\vec{t}_1^{\,k_1}\cdot p_e\cdot p'\quad\text{with}\quad p'\in\{e, \check{0}1\hat{0}\}
$$
(listed in rows $3$ and $33^*$) are the only paths with exponent $(1,0,0)$. 
\end{enumerate}

The theorem now follows by combining Tables~\ref{tab:pi2paths0} and~\ref{tab:pi2paths}.
\end{proof}

\begin{Rem}\label{rem:cancel}
We note that the paths in rows $17^*$ and $18^*$ from Table~\ref{tab:pi2paths0}, and rows $32^*$ and $33^*$ from Table~\ref{tab:pi2paths}, while giving maximal exponent paths, do not contribute to maximal exponents in matrix entries due to cancellations. Let us explain this further.
\medskip

Let $p_0=12\hat{1}2$ and $p_0'=\hat{1}\check{2}1\hat{2}$ be the the paths on rows~$17^*$ and $18^*$ of Table~\ref{tab:pi2paths0}, and suppose that $p^0$ is a path of type $\vec{t}_2^{\,\ell}\cdot \vec{t}_1^{\,k}\cdot\vec{v}$ with $k,l\geq 0$ and $v\in \sB_0$, with $\mathrm{start}(p^0)=\mathrm{end}(p_0)=\mathrm{end}(p_0')=1$. Let $p=p_0\cdot p^0$ and $p'=p_0'\cdot p^0$. Note that these paths are of the same type $w=1212t_2^{\ell}t_1^{k}v$, and they have the same start and end alcove. In particular, for any fundamental domain $\sB$, after using $\tau_2$ to move the start alcove of both paths into $\sB$ (if required) we have $\mathrm{start}(p)=\mathrm{start}(p')=u$, $\mathrm{wt}_{\sB}^2(p)=\mathrm{wt}_{\sB}^2(p')=k$, and $\theta_{\sB}^2(p)=\theta_{\sB}^2(p')=u'$, say. The combined contribution to the matrix $\pi_2(T_w;\sB)$ from these two paths is in the $(u,u')$-entry, and it is given by 
\begin{align*}
(\cQ_2(p)+\cQ_2(p'))\zeta^{k}&=\cQ_2(p^0)(\cQ_2(p_0)+\cQ_2(p_0'))\zeta^k=\cQ_2(p^0)(\sq_1-\sq_1^{-1}-\sq_2^{-1}(\sq_1-\sq_1^{-1})(\sq_2-\sq_2^{-1}))\zeta^{k}\\
&=\cQ_2(p^0)(-\sq_1\sq_2^{-2}-\sq_1^{-1}+\sq_1^{-1}-\sq_1^{-1}\sq_2^{-2})\zeta^k
\end{align*}
Note that the leading terms have cancelled, and hence each remaining exponent $\mathbf{x}$ satisfies $\mathbf{x}\prec\exp(\cQ_2(p))$. A similar comment applies to the paths on rows $32^*$ and $33^*$ from Table~\ref{tab:pi2paths}. Thus, for the purpose of computing optimal terms in matrix entries, the paths from these rows can be ignored.
\medskip

The cancellations outlined above turn out to be the only ``generic'' cancellations of leading terms that occur for paths in the tables. However, as we see below, cancellations can and do occur after specialising, where leading terms for one maximal exponent can cancel with leading terms from another maximal exponent when the exponents lead to equal degrees on specialisation.
\end{Rem}

\begin{Cor}\label{cor:partial}
We have $\mathbb{M}(\pi_2)=\{(1,0,0),(0,1,0),(0,0,1),(2,-1,0),(2,0,-1)\}$. 
\end{Cor}

\begin{proof}
Let $\mathbb{E}=\{(1,0,0),(0,1,0),(0,0,1),(2,-1,0),(2,0,-1)\}$. We have shown in Theorem~\ref{thm:Gamma2Paths} that if $p$ is a $2$-folded alcove path then $\exp(\cQ_2(p))\preceq \mathbf{x}$ for some $\mathbf{x}\in\mathbb{E}$. It then follows from Theorem~\ref{thm:pii} that if $\mathbf{y}$ is an exponent appearing with nonzero coefficient in some matrix entry of some $\pi_2(T_w;\sB)$ then $\mathbf{y}\preceq\mathbf{x}$ for some $\mathbf{x}\in\mathbb{E}$. Thus to show that $\mathbb{M}(\pi_2)=\mathbb{E}$ it is sufficient to show that each $\mathbf{x}\in\mathbb{E}$ does indeed appear as an exponent in some matrix entry of some matrix $\pi_2(T_w;\sB)$ (note -- it is a priori not sufficient to show that there exist $2$-folded alcove paths with these exponents, because it is possible for leading terms to cancel in the matrix entries). To this end we use Theorem~\ref{thm:pii} to see that, in the standard $W_0^2$-basis, we have $[\pi_2(T_1)]_{e,e}=\sq_1-\sq_1^{-1}$ (exponent $(1,0,0)$), $[\pi_2(T_2)]_{1,1}=\sq_2-\sq_2^{-1}$ (exponent $(0,1,0)$), $[\pi_2(T_0)]_{121,121}=\sq_0-\sq_0^{-1}$ (exponent $(0,0,1)$), and 
\begin{align*}
[\pi_2(T_1T_2T_1)]_{e,e}&=-\sq_1^{2}\sq_2^{-1}+\sq_2+\sq_2^{-1}-\sq_1^{-2}\sq_2^{-1}&[\pi_2(T_1T_0T_1)]_{12,12}&=-\sq_1^{2}\sq_0^{-1}+\sq_0+\sq_0^{-1}-\sq_2^{-2}\sq_0^{-1}
\end{align*}
giving exponents $(2,-1,0)$ and $(2,0,-1)$ respectively. Thus $\mathbb{M}(\pi_2)=\mathbb{E}$. 
\end{proof}

We can now prove Theorems~\ref{thm:gamma2part1}--\ref{thm:gamma2part3}. 

\begin{proof}[Proof of Theorems~\ref{thm:gamma2part1}--\ref{thm:gamma2part3}]
Consider the case $(r_1,r_2)\in R_1$. By specialising we have that $\deg(\cQ_2(p))$ is bounded by each integer $xa+yb+zc$ with $(x,y,z)\in\mathbb{M}(\pi_2)$ (see (\ref{eq:degreeconv})), and thus $\deg(\cQ_2(p))$ is bounded by $2a-c$ with equality if and only if $\exp(\cQ_2(p))=(2,0,-1)$. We now compute the $(2,0,-1)$ slice of $\pi_2(T_w)$. 
\medskip

We first find all paths with exponent $(2,0,-1)$. These paths are obtained by choosing a path $p_0$ from Table~\ref{tab:pi2paths0}, and $p^0$ from Table~\ref{tab:pi2paths}, with $\mathrm{end}(p_0)=\mathrm{start}(p^0)$ and with exponents summing to $(2,0,-1)$. Explicitly these paths are as follows:
\begin{enumerate}
\item The paths starting at $e$ are $p=p_0\cdot p^0$ with either $\text{row}(p_0)=13$ and $\text{row}(p^0)\in\{2,9,16,25,40,44\}$, or $\text{row}(p_0)=14$ and $\text{row}(p_0)\in\{7,14,19,23,27,31\}$.  
\item The paths starting at $1$ are $p=p_0\cdot p^0$ with either $\text{row}(p_0)=9$ and $\text{row}(p^0)\in\{2,9,16,25,40,44\}$, or $\text{row}(p_0)=10$ and $\text{row}(p_0)\in\{7,14,19,23,27,31\}$. 
\item The paths starting at $12$ are $p=p_0\cdot p^0$ with either $\text{row}(p_0)=3$ and $\text{row}(p^0)\in\{2,9,16,25,40,44\}$, or $\text{row}(p_0)=4$ and $\text{row}(p_0)\in\{7,14,19,23,27,31\}$. 
\item The paths starting at $121$ are $p=p_0\cdot p^0$ with $\text{row}(p_0)=1$ and $\text{row}(p_0)\in\{7,14,19,23,27,31\}$.
\end{enumerate}
Each of these paths can be rewritten in the form $u^{-1}\sw_1\st_1^Nv$ for some $u,v\in \sB_1$ and $N\geq 0$ (recall that $\st_1=2101$, and note that $t_1=t_{\omega_1}=0121$). This shows that if $\exp(\cQ_2(p))=(2,0,-1)$ then $w\in \Ga_2(R_1)$. 
\medskip

The paths above combine to give all paths of the form
\begin{enumerate}
\item $(21)^{-1}\cdot\hat{1}\check{0}\hat{1}\cdot\st_1^N\cdot v$ and $(21)^{-1}\cdot 101\cdot \st_1^{N-k-1}\cdot 2\hat{1}\check{0}\hat{1}\cdot \st_1^k\cdot v$ with $v\in\sB_1$ and $0\leq k\leq N-1$.
\item $(2)^{-1}\cdot\hat{1}\check{0}\hat{1}\cdot\st_1^N\cdot v$ and $(2)^{-1}\cdot 101\cdot \st_1^{N-k-1}\cdot 2\hat{1}\check{0}\hat{1}\cdot \st_1^k\cdot v$ with $v\in\sB_1$ and $0\leq k\leq N-1$.
\item $(e)^{-1}\cdot\hat{1}\check{0}\hat{1}\cdot\st_1^N\cdot v$ and $(e)^{-1}\cdot 101\cdot \st_1^{N-k-1}\cdot 2\hat{1}\check{0}\hat{1}\cdot \st_1^k\cdot v$ with $v\in\sB_1$ and $0\leq k\leq N-1$.
\item $(210)^{-1}\cdot\hat{1}\check{0}\hat{1}\cdot\st_1^N\cdot v$ and $(210)^{-1}\cdot 101\cdot \st_1^{N-k-1}\cdot 2\hat{1}\check{0}\hat{1}\cdot \st_1^k\cdot v$ with $v\in\sB_1$ and $0\leq k\leq N-1$.
\end{enumerate}
Using the action of $\tau_2$ on $\cU_2$ we consider the paths in point 4 to start at $0=\tau_2\cdot 121$. Then, with respect to the fundamental domain $\sB_1'=z_1^{-1}\sB_1=\{e,1,12,0\}$ the paths in each of the points have weights $N$ or $2k-N$, and $\theta_{\sB_1'}^2(p)=v$ in all cases. Then
$$
\fc_{\pi_2}^{(2,0,-1)}(u^{-1}\sw_1\st_1^Nv)=-\left(\zeta^N+\sum_{k=0}^{N-1}\zeta^{2k-N}\right)E_{u,v}=-\fs_N(\zeta)E_{u,v},
$$
where the minus sign comes from the fact that $\cQ_2(p)=(-\sq_0)^{-1}(\sq_1-\sq_1^{-1})^2$ has leading term $-\sq_0^{-1}\sq_1^2$. 
\medskip

This calculation shows that $\exp(\cQ_2(p))=(2,0,-1)$ if and only if $w\in \Ga_2(R_1)$. It follows that $\B{2}$ and $\B{3}$ hold for the representation $\pi_2$ equipped with the basis associated to $\sB_1'$, with $\ba_{\pi_2}=2a-c$. Then, by Theorem~\ref{thm:sumslices} we have
$$
\fc_{\pi_2}(w;\sB'_1)=\fc_{\pi_2}^{(2,0,-1)}(w;\sB'_1).
$$
It is then clear that $\B{4}$ holds (by linear independence of Schur characters), and the formula 
$$
\fc_{\pi_2}(\su_w^{-1}\sw_1\su_w;\sB'_1)\fc_{\pi_2}(w;\sB'_1)=(-\fs_0(\zeta)E_{\su_w,\su_w})(-\fs_{\tau_w}(\zeta)E_{\su_w,\sv_w})=-\fc_{\pi_2}(w;\sB'_1)
$$
verifies $\B{5}$. 
\medskip

The case $(r_1,r_2)\in R_2$ is very similar -- one first identifies the paths with exponent $(0,1,0)$, and then rewrites these paths in the cell factorisation $u^{-1}\sw_2\st_2^Nv$ with $u,v\in\sB_2$. Next one adjusts the start of the paths according to the fundamental domain $\sB_2'=z_2^{-1}\sB_2=\{e,1,0,01\}$ (paths starting at $121$ now start at $0$, and those starting at $12$ now start at $01$). Since $\cQ_2(p)=\sq_2-\sq_2^{-1}$ has leading term $+\sq_2$ for all such paths we finally obtain $+\fs_N(\zeta)$. 
\medskip

In fact, all other cases are similar (although somewhat more complicated). For example, consider the non-generic case $(r_1,r_2)\in R_3$, where $c=2a-b$ and $c<b$. One proceeds as above, however note that on specialising the maximum value of $xa+yb+zc$ for $\bx \in\mathbb{M}(\pi_2)$ is $\max\{a,b,c,2a-c,2a-c\}=c$ attained at $\mathbf{x}=(0,0,1)$ and $\mathbf{x}=(2,-1,0)$. 
One checks, directly from Theorem~\ref{thm:Gamma2Paths}, that if $p$ is of type $\vec{w}$ with $\exp(\cQ_2(p))\in\{(0,0,1),(2,-1,0)\}$ then $w\in\Ga_2(R_3)$. We then compute the sum of slices with respect to the adjusted basis $\cB'_3$ associated to $\sB'_3$
$$
\fc_{\pi_2}^{(0,0,1)}(w;\sB'_3)+\fc_{\pi_2}^{(2,-1,0)}(w;\sB'_3),
$$
and it turns out that this sum is precisely as stated in Theorem~\ref{thm:gamma2part1}. It follows that $\exp(\cQ_2(p))\in\{(0,0,1),(2,-1,0)\}$ if and only if $w\in\Ga_2(R_3)$. Hence $\B{2}$ and $\B{3}$ hold, and Theorem~\ref{thm:sumslices} shows that the above sum of slices equals $\fc_{\pi_2}(w;\sB'_3)$. Axiom $\B{4}$ readily follows. To verify axiom $\B{5}$ let $u_0=101$ and set $d_{u_0}=\sw_1$ and $d_{u}=u^{-1}\sw_2u$ for all $u\in \sB_3\backslash\{u_0\}$ (these turn out to be the Duflo involutions, see Theorem~\ref{thm:Duflo}). Note that $\fc_{\pi_2}(d_u)=-E_{u,u}$ for all $u\in\sB'_3$. Then, for $w\in \Ga_2(R_3)$ we have $\fc_{\pi_2}(d_{\su_w};\sB'_3)\fc_{\pi_2}(w;\sB'_3)=-\fc_{\pi_2}(w;\sB'_3)$, and hence $\B{5}$ holds. 
\medskip

We omit the details for the ``generic'' cases in Theorem~\ref{thm:gamma2part2} which involve the extended affine Weyl group -- the general approach is similar to the above. Thus consider the most intricate case of all -- the equal parameter case of Theorem~\ref{thm:gamma2part3}. In this case, quite remarkably, the maximum value of $xa+yb+zc$ is $a$, attained at all $\mathbf{x}\in\mathbb{M}(\pi_i)$. One checks directly from Theorem~\ref{thm:Gamma2Paths} that if $p$ is of type $\vec{w}$ with $\exp(\cQ_2(p))\in\mathbb{M}(\pi_2)$ then either $w\in\Ga_2(R_6)$ or $p^0$ is on row $32^*$ or $33^*$. However, as explained in Remark~\ref{rem:cancel}, the paths on rows $32^*$ and $33^*$ may be discarded (as their leading terms cancel one another). We now compute the sum of all slices:
$$
\sum_{\mathbf{x}\in\mathbb{M}(\pi_2)}\fc_{\pi_2}^{\mathbf{x}}(w;W_0^2)
$$
with respect to the standard basis. A rather miraculous calculation (with many cancellations occurring) shows that this sum of slices is precisely as stated in Theorem~\ref{thm:gamma2part3}. This computation can be read immediately off the tables in Theorem~\ref{thm:Gamma2Paths}, because we work in the standard basis and thus no modifications or conversions are required; however one must be rather careful with signs. For example, let us compute the sum of slices for $w=w_{3,3}^k$. We look through the tables to find all paths of type $w_{3,3}^k=21\cdot t_{\omega_1}^k\cdot 01$. These paths are listed in Table~\ref{tab:equal}.

\begin{table}[H]
$$
\begin{array}{|c|c|c|c|c|c|c|c|c|}\hline
\text{$p_0$ row}&\text{$p^0$ row}&\text{start}(p)&p&\exp(\cQ_2(p))&\text{coeff}&\mathrm{wt}^2(p)&\theta^2(p)&\text{conditions}\\\hline\hline
8&15/39&1&\hat{2}1t_1^k01&(0,1,0)&+1&k+1&12&k\geq 0\\\hline
9&16&1&2\hat{1}\check{0}\hat{1}21t_1^{k-1}01&(2,0,-1)&-1&k&12&k\geq 1\\\hline
9&40&1&2\hat{1}\check{0}\hat{1}&(2,0,-1)&-1&0&12&\\\hline
10&17/41&1&21t_1^k\hat{0}1&(0,0,1)&+1&-k&12&k\geq 0\\\hline
10&18/42&1&21t_1^k0\hat{1}&(1,0,0)&+1&-k-1&e&k\geq 0\\\hline
10&19&1&21t_1^{k-1}012\hat{1}\check{0}\hat{1}&(2,0,-1)&-1&-k&12&k\geq 1\\\hline
10&20&1&21t_1^m\hat{0}121t_1^n01&(0,0,1)&+1&n-m+1&12&m+n=k-1\geq 0\\\hline
10&21&1&21t_1^m0\hat{1}\check{2}\hat{1}t_1^n01&(2,-1,0)&-1&n-m&12&m+n=k-1\geq 0\\\hline
10&22&1&21t_1^m01\hat{2}1t_1^n01&(0,1,0)&+1&n-m&12&m+n=k-1\geq 0\\\hline
10&23&1&21t_1^m012\hat{1}\check{0}\hat{1}21t_1^n01&(2,0,-1)&-1&n-m&12&m+n=k-2\geq 0\\\hline
\end{array}
$$
\caption{Paths for $w_{3,3}^k$}
\label{tab:equal}
\end{table}
Therefore, with respect to the standard basis,
\begin{align*}
\fc_{\pi_2}^{(1,0,0)}(w_{3,3}^k)&=\zeta^{-k-1}E_{2,1}\\
\fc_{\pi_2}^{(0,1,0)}(w_{3,3}^k)&=\bigg(\zeta^{k+1}+\sum_{n=0}^{k-1}\zeta^{2n-k+1}\bigg)E_{2,3}=\big(\zeta^{k+1}+\fs_{k-1}(\zeta)\big)E_{2,3}\\
\fc_{\pi_2}^{(0,0,1)}(w_{3,3}^k)&=\bigg(\zeta^{-k}+\sum_{n=0}^{k-1}\zeta^{2n-k+2}\bigg)E_{2,3}=\fs_k(\zeta)E_{2,3}\\
\fc_{\pi_2}^{(2,-1,0)}(w_{3,3}^k)&=-\bigg(\sum_{n=0}^{k-1}\zeta^{2n-k+1}\bigg)E_{2,3}=-\fs_{k-1}(\zeta)E_{2,3}\\
\fc_{\tilde{\pi}_2}^{(2,0,-1)}(w_{3,3}^k)&=-\bigg(\zeta^k+\zeta^{-k}+\sum_{n=0}^{k-2}\zeta^{2n-k+2}\bigg)E_{2,3}=-\fs_k(\zeta)E_{2,3}.
\end{align*}
Thus the sum of slices is
$$
\sum_{\mathbf{x}\in\mathbb{M}(\pi_2)}\fc_{\pi_2}^{\mathbf{x}}(w;W_0^2)=\zeta^{-k-1}E_{2,1}+\zeta^{k+1}E_{2,3}.
$$
Note the remarkable cancellations that have occurred. The remaining formulae for the sum of slices for each $w=w_{i,j}^k$ follow very similarly. Then $\B{2}$ and $\B{3}$ follow for the representation $\pi_2$, and it is easy to see that $\B{2}$ and $\B{3}$ also hold for $\tilde{\pi}_2$. Verification of $\B{4}$ for $\tilde{\pi}_2$ is as follows (note that obviously $\B{4}$ fails for $\pi_2$, as $\fc_{\pi_2}(0)=E_{4,4}=\fc_{\pi_2}(010)$). Suppose that 
\begin{align}\label{eq:li}
\sum_{w\in\Ga_2}a_w\fc_{\tilde{\pi}_2}(w)=0\quad\text{for some $a_w\in\mathbb{Z}$ (finitely many of which are nonzero)}.
\end{align}
Write $a_{ij}^k=a_{w_{ij}^k}$. Consider the $(1,1)$-entry of~(\ref{eq:li}). This gives
$$
\sum_{k\geq 0}a_{21}^k\zeta^k+\sum_{k\geq 0}a_{43}^k\zeta^{-k-1}=0.
$$
Since each power of $\zeta$ appears at most once, we have $a_{21}^k=a_{43}^k=0$ for all $k\geq 0$. Similarly, by considering the $(1,2)$, $(1,3)$, $(1,4)$ and $(2,1)$ entries of~(\ref{eq:li}) gives $$a_{24}^k=a_{23}^k=a_{41}^k=a_{22}^k=a_{42}^k=a_{31}^k=a_{33}^k=0.$$
The $(2,2)$-entry gives 
$$
a_{s_2}+a_{s_2s_1s_2}+\sum_{k\geq 0}a_{44}^k\zeta^{-k-1}+\sum_{k\geq 0}a_{34}^k(\zeta^{k+1}+\zeta^{-k-1})=0
$$
Thus $a_{34}^k=0$ for all $k\geq 0$ (considering the powers $\zeta^{k+1}$), and then it follows that $a_{44}^k=0$ for all $k\geq 0$ (considering the powers $\zeta^{-k-1}$) and thus $a_{s_2}+a_{s_2s_1s_2}=0$. Now considering the $(6,6)$-entry we have $a_{s_2}-a_{s_2s_1s_2}=0$, and hence $a_{s_2}=a_{s_2s_1s_2}=0$. Continuing in this way we see that $a_{ij}^k=0$ for all $i,j,k$, and hence $\B{4}$ holds. 
\medskip

To verify $\B{5}$, note directly from the formulae for $\fc_{\tilde{\pi}_2}(w_{ij}^k)$ that 
$$
\fc_{\tilde{\pi}_2}(s_j)\fc_{\tilde{\pi}_2}(w)=\fc_{\tilde{\pi}_2}(w)\quad\text{for all $w$ in the right cell of $s_j$ (with $j=0,1,2$)}
$$
(we note that the elements $s_j$, $j=0,1,2$, turn out to be the Duflo involutions, see Theorem~\ref{thm:Duflo}). The proof is now complete. 
\end{proof}

\subsection{The cell $\Gamma_1$}\label{sec:Ga1}

The analysis of this cell is similar to (and in fact considerably easier than) the $\Ga_2$ case.  
\medskip

The stable regions for $\Gamma_1$ (with $r_2\leq r_1$) are as follows.
\begin{align*}
R_{1}&=\{(r_1,r_2)\in\mathbb{Q}_{>0}^2\mid r_2\leq r_1,\,r_2<1-r_1\}&
R_{2}&=\{(r_1,r_2)\in\mathbb{Q}_{>0}^2\mid r_2\leq r_1,\,r_2>1-r_1,\,r_2>r_1-1\}\\
R_{3}&=\{(r_1,r_2)\in\mathbb{Q}_{>0}^2\mid r_2<r_1-1\}&
R_4=R_{1,2}&=\{(r_1,r_2)\in\mathbb{Q}_{>0}^2\mid r_2\leq r_1,\,r_2=1-r_1\}\\
R_5=R_{2,3}&=\{(r_1,r_2)\in\mathbb{Q}_{>0}^2\mid r_2=r_1-1\}.
\end{align*}
The regimes $(r_1,r_2)\in R_{j}$ with $j=1,2,3$ are ``generic'', and admit cell factorisations
where
\begin{align*}
\sw_j=\begin{cases}
1&\text{if $j=1$}\\
02&\text{if $j=2$}\\
212&\text{if $j=3$}
\end{cases}
\quad
\st_j=\begin{cases}
021&\text{if $j=1$}\\
102&\text{if $j=2$}\\
012&\text{if $j=3$}
\end{cases}\quad\text{and}\quad
\sB_j=\begin{cases}
(e,2,0,02)&\text{if $j=1$}\\
(e,1,10,12)&\text{if $j=2$}\\
(e,0,01,010)&\text{if $j=3$}
\end{cases}
\end{align*}
If $w=u^{-1}\sw_j\st_j^kv$ with $u,v\in\sB_j$ and $k\geq 0$ we write, as usual, $\su_w=u$, $\sv_w=v$, and $\tau_w=k$.
\medskip

For each $j=1,2,3$ let $z_j\in \sB_j$ be such that $\sB_j'=\{z_j^{-1}u\mid u\in\sB_j\}$ is a fundamental domain for the action of $\tau_1$ on $\cU_1$ with $z_j^{-1}e$ on the negative side of each hyperplane separating $z_j^{-1}e$ from $z_j^{-1}u$ with $u\in\sB_j$. Specifically, $z_j=2,12,e$ in the cases $j=1,2,3$. Let $\cB'_j=(\xi_1\otimes X_{z_j^{-1}u}\mid u\in \sB'_j)$ be the basis associated to the fundamental domain $\sB'_j$. Thus
\begin{align*}
\cB'_j&=\begin{cases}
(\xi_1\otimes X_2,\,\xi_1\otimes X_e,\,\xi_1\otimes X_{20},\,\xi_1\otimes X_{0})&\text{if $j=1$}\\
(\xi_1\otimes X_{21},\,\xi_1\otimes X_{2},\,\xi_1\otimes X_{20},\,\xi_1\otimes X_{e})&\text{if $j=2$}\\
(\xi_1\otimes X_{e},\,\xi_1\otimes X_{0},\,\xi_1\otimes X_{01},\,\xi_1\otimes X_{010})&\text{if $j=3$}.
\end{cases}
\end{align*}
The fundamental domain~$\sB_1'$ is depicted in the second example in Figure~\ref{fig:paths}.
\medskip

The regimes $R_{1,2}$ and $R_{2,3}$ are ``non-generic'', and do not admit cell factorisations. We have
\begin{align*}
\Gamma_1(R_{1,2})=\Gamma_1(R_1)\cup\{\sw_2\}\quad\text{and}\quad \Gamma_1(R_{2,3})=\Gamma_1(R_{2})\cup\{\sw_3\}.
\end{align*} 
Thus we can use cell factorisation in $\Gamma_1(R_1)$ to describe all elements of $\Gamma_1(R_{1,2})\backslash\{\sw_2\}$, and hence the expressions $\su_w$, $\sv_w$, and $\tau_w$ are defined for $w\in \Gamma_1(R_{1,2})\backslash\{\sw_2\}$. We extend this definition by setting 
$$
\su_{\sw_2}=\sv_{\sw_2}=02\quad\text{and}\quad \tau_{\sw_2}=-1.
$$ 
Similarly we can use cell factorisation in $\Gamma_1(R_2)$ to describe all elements of $\Gamma_1(R_{2,3})\backslash\{\sw_3\}$, and hence the expressions $\su_w$, $\sv_w$, and $\tau_w$ are defined for $w\in \Gamma_1(R_{2,3})\backslash\{\sw_3\}$. We extend this definition by setting 
$$
\su_{\sw_3}=\sv_{\sw_3}=12\quad\text{and}\quad \tau_{\sw_3}=-1.
$$

The main theorem of this section is the following. To conveniently state the theorem we will write $R_4=R_{1,2}$, $R_5=R_{2,3}$, $\sB'_4=\sB'_1$, $\cB'_4=\cB'_1$, $\sB'_5=\sB'_2$ and $\cB'_5=\cB'_2$. Moreover, we let $\mathsf{b}_4=02$ and $\mathsf{b}_5=12$. The elementary matrix $E_{u,v}$ ($u,v\in \sB_j$) denotes the matrix with $1$ at position $(k,\ell)$ where $k,\ell$ is the position of $u$ and $v$ in the ordered set $\sB_j$, and $0$ everywhere else.

\begin{Th}\label{thm:Ga1}
Let $(r_1,r_2)\in R_j$, with $1\leq j\leq 5$. Then $\pi_1$, equipped with the basis $\cB'_j$, satisfies $\B{1}$--$\B{5}$ for the cell $\Ga_1=\Gamma_1(r_1,r_2)$, with $\ba_{\pi_1}=\tilde{\ba}(\Gamma_1)$. Moreover, for $j=1,2,3$ the leading matrices of $\pi_1$ are
$$
\fc_{\pi_1}(w;\sB'_j)=\fs_{\tau_w}(\zeta)E_{\su_w,\sv_w}\quad\text{for $w\in \Gamma_1$},
$$
where $\fs_k(\zeta)$ is the Schur function of type $A_1$. In the cases $j=4,5$ we have, for $w\in \Gamma_1$,
\begin{align*}
\fc_{\pi_1}(w;\sB'_j)&=\ff_{\tau_w}^{\su_w,\sv_w}(\zeta)E_{\su_w,\sv_w}&\text{where}&&\ff_k^{u,v}(\zeta)&=
\begin{cases}
\fs_k(\zeta)\pm\fs_{k-1}(\zeta)&\text{if $u,v\neq \mathsf{b}_j$}\\
\fs_k(\zeta)\pm\fs_{k-1}(\zeta)\pm\zeta^{-k-1}&\text{if $u\neq \mathsf{b}_j$ and $v=\mathsf{b}_j$}\\
\fs_k(\zeta)\pm\fs_{k-1}(\zeta)\pm\zeta^{k+1}&\text{if $u=\mathsf{b}_j$ and $v\neq \mathsf{b}_j$}\\
\fs_k(\zeta)\pm\fs_{k+1}(\zeta)&\text{if $u,v=\mathsf{b}_j$}
\end{cases}
\end{align*}
with the $+$ sign for $j=4$, and the $-$ sign for $j=5$, and where $\fs_{-1}(\zeta)=0$. 
\end{Th}

\begin{proof}
The proof of Theorem~\ref{thm:Ga1} is similar to the proof of Theorem~\ref{thm:gamma2part1}, and we will simply make some comments and omit the details. One first establishes an analogue of Theorem~\ref{thm:Gamma2Paths} using the $1$-folding tables for $\vec{t}_1$ and $\vec{t}_2$ given in Table~\ref{tab:fold1}. 

\begin{table}[H]
\renewcommand{\arraystretch}{1.2}  
\begin{subfigure}{.5\linewidth}
\centering
$\begin{array}{|c||c|c|c|c|}\hline
&0&1&2&1 \\\hline\hline
1&-&-&-&* \\\hline
2&-&*&-&- \\\hline
3&1&*&1&2 \\\hline
4&2&1&2&* \\\hline
\end{array}$
\caption{$\vec{t}_1=0121$}
\end{subfigure}
\begin{subfigure}{.5\linewidth}
   \centering
   $\begin{array}{|c||c|c|c|c|c|c||c|}
   \hline
   & 0 &1 &0 &2&1&2 &0\\\hline
   \hline
1&-&-&-&-&-&-&-\\\hline
2&-&*&1&-&*&1&-\\\hline
3&1&*&-&1&*&-&1\\\hline
4&2&1&3&2&1&3&2\\\hline
\end{array}$
\caption{$\vec{t}_2=010212$ and $\vec{\mathsf{b}}_0=010210$}
\end{subfigure}
\caption{$1$-folding tables}
\label{tab:fold1}
\end{table}

In particular one shows that $\exp(\cQ_1(p))\preceq \mathbf{x}$ for some $\mathbf{x}\in \{(1,0,0),(0,1,1),(-1,2,0),(-1,0,2)\}$. Then, as in Corollary~\ref{cor:partial} we see that
\begin{align*}
\mathbb{M}(\pi_1)=\{(1,0,0),(0,1,1),(-1,2,0),(-1,0,2)\}.
\end{align*}
Next one classifies the paths $p$ for which $\exp(\cQ_1(p))=\mathbf{x}$ for some $\mathbf{x}\in\mathbb{M}(\pi_1)$. Theorem~\ref{thm:Ga1} now follows as in the $\Ga_2$ case.  
\end{proof}

\subsection{The cell $\Gamma_3$}\label{sec:Ga3}

Again, the analysis of this cell is similar to (and considerably easier than) the $\Ga_2$ case.
\medskip

The cell $\Gamma_3$ is stable in the following regions:
\begin{align*}
R_{1}&=\{(r_1,r_2)\in\mathbb{Q}_{>0}^2\mid r_1-2<r_2<r_1\}\\
R_{2}&=\{(r_1,r_2)\in\mathbb{Q}_{>0}^2\mid r_2<r_1-2\}\\
R_3=R_{1,2}&=\{(r_1,r_2)\in\mathbb{Q}_{>0}^2\mid r_2=r_1-2\}.
\end{align*}
The parameters $(r_1,r_2)\in R_{1}\cup R_2$ are generic for the cell $\Gamma_3$, and we have a cell factorisation 
where
$$
\sw_j=\begin{cases}
0101&\text{if $j=1$}\\
02&\text{if $j=2$}
\end{cases}\quad 
\st_j=\begin{cases}
2101&\text{if $j=1$}\\
1012&\text{if $j=2$}
\end{cases}\quad \quad\text{and}\quad
\sB_j=\begin{cases}
(e,2,21,210)&\text{if $j=1$}\\
(e,1,10,101)&\text{if $j=2$}.
\end{cases}
$$
For each $j=1,2$ let $z_j\in \sB_j$ be such that $\sB'_j=\{z_j^{-1}u\mid u\in\sB_j\}$ is a fundamental domain for the action of $\tau_2$ on $\cU_2$ with $z_j^{-1}e$ on the negative side of each hyperplane separating $z_j^{-1}e$ from $z_j^{-1}u$ with $u\in\sB_j$. Specifically, $z_j=21,1$ in the cases $j=1,2$. Let $\cB'_j=(\xi_1\otimes X_{z_j^{-1}u}\mid u\in \sB'_j)$ be the basis associated to the fundamental domain $\sB'_j$. Thus
$$
\cB'_j=\begin{cases}
(\xi_3\otimes X_{12},\xi_3\otimes X_{1},\xi_3\otimes X_{e},\xi_3\otimes X_{0})&\text{if $j=1$}\\
(\xi_3\otimes X_{1}, \xi_3\otimes X_{e},  \xi_3\otimes X_{0},  \xi_3\otimes X_{01})&\text{if $j=2$}.
\end{cases}
$$
The fundamental domain $\sB_1'$ is depicted in the third example of Figure~\ref{fig:paths}.
\medskip

The regime $R_3=R_{1,2}$ is non-generic for $\Gamma_3$, and there is no cell factorisation. However we note that
$$
\Gamma_3(R_{1,2})=\Gamma_3(R_2)\cup\{\sw_1\}.
$$
Thus we use the cell factorisation in $\Gamma_3(R_2)$ to describe the elements of $\Gamma_3(R_3)$, with the extension of notation 
$$
\su_{\sw_1}=\sv_{\sw_1}=101\quad\text{and}\quad \tau_{\sw_1}=-1.
$$
We also set $\sB'_3=\sB'_2$ and $\cB'_3=\cB'_2$. In the following theorem, the elementary matrix $E_{u,v}$ ($u,v\in \sB_j$) denotes the matrix with $1$ at position $(k,\ell)$ where $k,\ell$ is the position of $u$ and $v$ in the ordered set $\sB_j$, and $0$ everywhere else.

\begin{Th}\label{thm:Ga3}
Let $(r_1,r_2)\in R_j$, with $j=1,2,3$. Then $\pi_3$, equipped with the basis $\cB'_j$, satisfies $\B{1}$--$\B{5}$ for the cell $\Ga_3=\Gamma_3(r_1,r_2)$, with $\ba_{\pi_3}=\tilde{\ba}(\Gamma_3)$. Moreover, for $j=1,2$ the leading matrices of $\pi_3$ are
$$
\fc_{\pi_3}(w;\sB'_j)=\fs_{\tau_w}(\zeta)E_{\su_w,\sv_w}\quad\text{for $w\in \Gamma_3$}
$$
where $\fs_k(\zeta)$ is the Schur function of type $A_1$. In the case $j=3$ we have, for $w\in \Gamma_3$,
\begin{align*}
\fc_{\pi_3}(w;\sB'_3)&=\ff_{\tau_w}^{\su_w,\sv_w}(\zeta)E_{\su_w,\sv_w}&\text{where}&&\ff_k^{u,v}(\zeta)&=
\begin{cases}
\fs_k(\zeta)+\fs_{k-1}(\zeta)&\text{if $u,v\neq 101$}\\
\fs_k(\zeta)+\fs_{k-1}(\zeta)+\zeta^{-k-1}&\text{if $u\neq 101$ and $v=101$}\\
\fs_k(\zeta)+\fs_{k-1}(\zeta)+\zeta^{k+1}&\text{if $u=101$ and $v\neq 101$}\\
\fs_k(\zeta)+\fs_{k+1}(\zeta)&\text{if $u=v=101$},
\end{cases}
\end{align*}
where $\fs_{-1}(\zeta)=0$. 
\end{Th}

\begin{proof}
Again the proof of Theorem~\ref{thm:Ga3} is similar to the proof of Theorem~\ref{thm:gamma2part1}, however the presence of a positive contribution $\sq_0^{+1}$ to $\cQ_3(p)$ from the bounces on the ``top'' wall of the strip $\cU_2$ requires some additional arguments, which we now outline. Since the cell $\Gamma_3$ only occurs in the regime $r_2<r_1$ the key idea is to include the relation $(0,0,1)\prec (0,1,0)$ in the partial order on $\mathbb{Z}^3$. This turns out to be most useful in the form $(0,-1,1)\prec (0,0,0)$ which should be interpreted as saying that the combined contribution to exponent by performing both a bounce on the top of the strip and a bounce on the bottom of the strip is negative. 
\medskip

The $3$-folding tables of $\vec{t}_1$ and $\vec{t}_2$ are as in Table~\ref{tab:fold2}. Note that each row that contains at least one $*$ entry in fact contains precisely one $*$ in a $0$-headed column and one $*$ in a $2$-headed column. This fact makes the critical observation (\ref{eq:suboptimal}) remain true: If a pass of either the $\vec{t}_1$ or $\vec{t}_2$ table is completed on a row containing at least one $*$, and if no folds are made in this pass, then we have
\begin{align*}
\exp(\cQ_3(p))=\exp(\cQ_3(p'))+(0,-1,1)\prec\exp(\cQ_3(p')),
\end{align*}
where $p'$ is the path obtained from $p$ by removing this copy of $\vec{t}_1$ or $\vec{t}_2$. Thus such paths necessarily have strictly dominated exponents.
\medskip

Incorporating the above observations into the analysis one readily establishes an analogue of Theorem~\ref{thm:Gamma2Paths}. Specifically, for each $3$-folded alcove path $p$ we have $\exp(\cQ_3(p))\preceq\mathbf{x}$ for some $\mathbf{x}\in \{((2,0,2),(0,1,1)\}$. Then, as in Corollary~\ref{cor:partial} we see that
\begin{align*}
\mathbb{M}(\pi_3)=\{(2,0,2),(0,1,1)\},
\end{align*}
and the paths with $\exp(\cQ_3(p))=\mathbf{x}$ for some $\mathbf{x}\in\mathbb{M}(\pi_3)$ are easily classified. Theorem~\ref{thm:Ga3} follows.
\end{proof}

The proof of Theorem~\ref{thm:balancedsystems} is now complete. Moreover, we have explicit formulae for the leading matrices for all cells. Using these formulae we can easily verify conjecture~$\conj{8}$. 

\begin{Cor}\label{cor:P8}
Conjecture $\conj{8}$ holds for all choices of parameters. 
\end{Cor}

\begin{proof}
Suppose that $x,y,z\in W$ and $\gamma_{x,y,z^{-1}}\neq 0$. It follows that $x,y,z\in\Gamma$ for some $\Ga\in\tsc$ (see Theorem~\ref{thm:afn}). Then $\gamma_{x,y,z^{-1}}$ is the coefficient of $\fc_{\pi_{\Ga}}(z;\sB_{\Ga})$ in the expansion of $\fc_{\pi_{\Ga}}(x;\sB_{\Ga})\fc_{\pi_{\Ga}}(y;\sB_{\Ga})$. Suppose that $\Ga$ admits a cell factorisation. Then by the explicit formulae from Theorem~\ref{thm:finite} and Section~\ref{sec:infinite} we have $\fc_{\pi_{\Ga}}(w;\sB_{\Ga})=\mathfrak{f}_w\,E_{\su_w,\sv_w}$ for some constant or Schur function $\mathfrak{f}_w\neq 0$. Then
$$
\fc_{\pi_{\Ga}}(x;\sB_{\Ga})\fc_{\pi_{\Ga}}(y;\sB_{\Ga})=\mathfrak{f}_x\mathfrak{f}_yE_{\su_x,\sv_x}E_{\su_y,\sv_y}=\delta_{\sv_x,\su_y}\mathfrak{f}_x\mathfrak{f}_yE_{\su_x,\sv_y}.
$$
Thus if $\gamma_{x,y,z^{-1}}\neq 0$ we have $\sv_x=\su_y$ (that is, $x^{-1}\sim_{\cR} y$), and moreover $\su_z=\su_x$ (that is, $z\sim_{\cR} x$) and $\sv_z=\sv_y$ (that is, $z^{-1}\sim_{\cR} y^{-1}$). Hence $\conj{8}$ follows in this case. 
\medskip

If $\Ga$ does not admit a cell factorisation then the result follows by more direct computation using the explicit formulae for the leading matrices, and we omit the easy details.
\end{proof}


\section{The asymptotic Plancherel formula}\label{sec:7}

At this stage we have computed Lusztig's $\ba$-function, and proved conjectures $\conj{4}$, $\conj{8}$, $\conj{9}$, $\conj{10}$, $\conj{11}$, $\conj{12}$, and $\conj{14}$ (see Corollaries~\ref{cor:conj14}, \ref{cor:afn}, and~\ref{cor:P8}). In this section we prove the remaining conjectures. With the exception of $\conj{15}$, all of these conjectures follow from a remarkable property (Theorem~\ref{thm:muprime}) of Opdam's Plancherel formula which ensures that there is a descent to an ``asymptotic Plancherel formula'' on Lusztig's asymptotic algebra $\mathcal{J}$. This asymptotic Plancherel formula ensures that $\conj{7}$ holds (since we obtain an inner product on $\mathcal{J}$), and moreover allows us to prove $\conj{1}$ and compute the Duflo involutions. Conjectures $\conj{2}$, $\conj{3}$, $\conj{5}$, $\conj{6}$, and $\conj{13}$ all follow. Conjecture~$\conj{15}$ is of a slightly different flavour, and uses an additional ingredient due to Xie~\cite{Xie:15} (see Theorem~\ref{thm:conj15}). 

\subsection{The Plancherel formula}

Since the Plancherel Theorem is inherently an analytic concept, we regard $\cH$ as an algebra over $\mathbb{C}$ by specialising $\sq\to q$ for some real number $q>1$ and extending scalars from $\mathbb{Z}$ to $\mathbb{C}$. We write $\cH_{\mathbb{C}}$ for this specialised algebra. Let $\pi_i$, $i=0,1,\ldots,13$ be the specialisations of the representations $\pi_i$ defined earlier. Now we regard $\zeta\in(\mathbb{C}^{\times})^2$ for the representation $\pi_0=\pi_0^{\zeta}$ and $\zeta\in\mathbb{C}^{\times}$ for the representations $\pi_i=\pi_i^{\zeta}$ with $i=1,2,3$. Write $\chi_i^{\zeta}$ for the character of $\pi_i^{\zeta}$ for $i=0,1,2,3$, and write $\chi_i$ for the character of $\pi_i$ for $i=4,5,\ldots,13$. 
\medskip

Define an involution $*$ on $\cH_{\mathbb{C}}$ and the \textit{canonical trace functional} $\mathrm{Tr}:\cH_{\mathbb{C}}\to\mathbb{C}$ by
$$
\left(\sum_{w\in W}a_wT_w\right)^*=\sum_{w\in W}\overline{a_w}\,T_{w^{-1}}\quad\text{and}\quad \mathrm{Tr}\left(\sum_{w\in W}a_wT_w\right)=a_e
$$
where now $\overline{a_w}$ denotes complex conjugation. An induction on $\ell(v)$ shows that $\mathrm{Tr}(T_uT_v^*)=\delta_{u,v}$ for all $u,v\in W$, and hence $\mathrm{Tr}(h_1h_2)=\mathrm{Tr}(h_2h_1)$ for all $h_1,h_2\in \cH_{\mathbb{C}}$. It follows that
$
(h_1,h_2)=\mathrm{Tr}(h_1h_2^*)
$
defines an Hermitian inner product on $\cH_{\mathbb{C}}$. Let $\|h\|_2=\sqrt{(h,h)}$ be the $\ell^2$-norm. The algebra $\cH_{\mathbb{C}}$ acts on itself by left multiplication, and the corresponding operator norm is $\|h\|=\sup\{\|hx\|_2\colon x\in\cH_{\mathbb{C}},\|x\|_2\leq 1\}$. Let $\overline{\cH}_{\mathbb{C}}$ denote the completion of $\cH_{\mathbb{C}}$ with respect to this norm. Thus $\overline{\cH}_{\mathbb{C}}$ is a non-commutative $C^*$-algebra. The irreducible representations of $\overline{\cH}_{\mathbb{C}}$ are the (unique) extensions of the irreducible representations of $\cH_{\mathbb{C}}$ that are continuous with respect to the $\ell^2$-operator norm, and it turns out that these are the irreducible ``tempered'' representations of $\cH_{\mathbb{C}}$ (see \cite[\S2.7 and Corollary~6.2]{Op:04}). In particular, every irreducible representation of $\overline{\cH}_{\mathbb{C}}$ is finite dimensional (since every irreducible representation of $\cH_{\mathbb{C}}$ has degree at most $|W_0|$), and it follows from the general theory of traces on ``liminal'' $C^*$-algebras that there exists a unique positive Borel measure $\mu$ on $\irreps$, called the \textit{Plancherel measure}, such that (see \cite[\S8.8]{Dix:77})
$$
\mathrm{Tr}(h)=\int_{\irreps}\chi_{\pi}(h)\,d\mu(\pi)\quad\text{for all $h\in\overline{\cH}_{\mathbb{C}}$}.
$$
The Plancherel measure has been computed in general by Opdam~\cite{Op:04}. We now recall the explicit formulation in type~$\tilde{C}_2$ obtained by the second author in~\cite[\S 4.7]{Par:14}. 
\medskip

Define rational functions $c_j(\zeta)$, $j=0,1,2,3$, by
\begin{align*}
c_0(\zeta)&=\frac{(1-q^{-2a}\zeta_1^{-1})(1-q^{-2a}\zeta_1^{-1}\zeta_2^{-2})(1-q^{-b-c}\zeta_1^{-1}\zeta_2^{-1})(1+q^{-b+c}\zeta_1^{-1}\zeta_2^{-1})(1-q^{-b-c}\zeta_2^{-1})(1+q^{-b+c}\zeta_2^{-1})}{(1-\zeta_1^{-1})(1-\zeta_1^{-1}\zeta_2^{-2})(1-\zeta_1^{-2}\zeta_2^{-2})(1-\zeta_2^{-2})}\\
c_1(\zeta)&=\frac{(1+q^{-a-b-c}\zeta^{-1})(1-q^{-a-b+c}\zeta^{-1})(1+q^{a-b-c}\zeta^{-1})(1-q^{a-b+c}\zeta^{-1})}{(1-\zeta^{-2})(1-q^{2a}\zeta^{-2})}\\
c_2(\zeta)&=\frac{(1+q^{-b+c}\zeta^{-1})(1-q^{-2a-b-c}\zeta^{-1})(1-q^{-2a+b+c}\zeta^{-1})}{(1-\zeta^{-2})(1-q^{b+c}\zeta^{-1})}\\
c_3(\zeta)&=\frac{(1-q^{-b-c}\zeta^{-1})(1+q^{-2a-b+c}\zeta^{-1})(1+q^{-2a+b-c}\zeta^{-1})}{(1-\zeta^{-2})(1+q^{b-c}\zeta^{-1})},
\end{align*}
%
%
and constants $C_j$, $j=0,1,2,\ldots,8$ by
\begin{align*}
C_0&=\frac{1}{8q^{4a+4b}}&C_1&=\frac{q^{2a}-1}{2q^{2a+4b}(q^{2a}+1)}\\
C_2&=\frac{q^{2b+2c}-1}{2q^{4a+2b}(q^{2b}+1)(q^{2c}+1)}&C_3&=\frac{q^{2b}-q^{2c}}{2q^{4a+2b}(q^{2b}+1)(q^{2c}+1)}\\
C_4&=\frac{(q^{2a+2b+2c}-1)(q^{4a+2b+2c}-1)}{(q^{2a}+1)(q^{2b}+1)(q^{2c}+1)(q^{2a+2b}+1)(q^{2a+2c}+1)}
&C_5&=\frac{(q^{2a-2b-2c}-1)(q^{4a-2b-2c}-1)}{(q^{2a}+1)(q^{-2b}+1)(q^{-2c}+1)(q^{2a-2b}+1)(q^{2a-2c}+1)}\\
C_6&=\frac{(q^{2b}-q^{2c})(q^{2b+2c}-1)}{(q^{2a+2b}+1)(q^{2a+2c}+1)(q^{-2a+2b}+1)(q^{-2a+2c}+1)}
&C_7&=\frac{(q^{2a+2b-2c}-1)(q^{4a+2b-2c}-1)}{(q^{2a}+1)(q^{2b}+1)(q^{-2c}+1)(q^{2a+2b}+1)(q^{2a-2c}+1)}\\
C_8&=\frac{(q^{2a-2b+2c}-1)(q^{4a-2b+2c}-1)}{(q^{2a}+1)(q^{-2b}+1)(q^{2c}+1)(q^{2a-2b}+1)(q^{2a+2c}+1)}.
\end{align*}

The explicit formulation of the Plancherel formula for $\tilde{C}_2$ from \cite[\S 4.7]{Par:14} is as follows. Let $\mathbb{T}=\{\zeta\in\mathbb{C}\,:\, |\zeta|=1\}$ denote the circle group with normalised Haar measure $d\zeta$ (thus $\int_{\mathbb{T}}f(\zeta)d\zeta=\frac{1}{2\pi}\int_0^{2\pi}f(e^{i\theta})d\theta$).

\begin{Th}\label{thm:plancherel} Let $h\in\overline{\cH}_{\mathbb{C}}$. If $r_2\leq r_1$ then
\begin{align*}
\Tr(h)&=|C_0|\iint_{\mathbb{T}^2}\frac{\chi_0^{\zeta}(h)}{|c_0(\zeta)|^2}\,d\zeta+|C_1|\int_{\mathbb{T}}\frac{\chi_1^{\zeta}(h)}{|c_1(\zeta)|^2}\,d\zeta+|C_2|\int_{\mathbb{T}}\frac{\chi_2^{\zeta}(h)}{|c_2(\zeta)|^2}\,d\zeta+|C_3|\int_{\mathbb{T}}\frac{\chi_3^{\zeta}(h)}{|c_3(\zeta)|^2}\,d\zeta\\
&\qquad+|C_4|\chi_4(h)+|C_5|\chi'(h)+|C_6|\chi_{12}(h)+|C_7|\chi_5(h)+|C_8|\chi''(h)
\end{align*}
where
\begin{align*}
\chi'&=\begin{cases}
\chi_7&\text{if $r_1+r_2>2$}\\
\chi_{10}&\text{if $1<r_1+r_2<2$}\\
\chi_9&\text{if $r_1+r_2<1$}
\end{cases}&\chi''&=\begin{cases}
\chi_8&\text{if $r_1-r_2>2$}\\
\chi_{11}&\text{if $1<r_1-r_2<2$}\\
\chi_{6}&\text{if $r_1-r_2<1$.}
\end{cases}
\end{align*}
If $r_2>r_1$ then the Plancherel Theorem is obtained from the $r_2<r_1$ formula by applying $\sigma$ to all representation, constants, and $c$-functions. The defining regions in $(\chi'')^{\sigma}$ are $r_2-r_1>2$, $1<r_2-r_1<2$, and $r_2-r_1<1$.
\end{Th}

\begin{proof}
See \cite[Section~4.7]{Par:14} for the case $r_1\neq r_2$, and \cite[Section~4.4]{Par:14} for the case $r_1=r_2$.  
\end{proof}

\subsection{The Plancherel formula and cell decomposition}

In this section we make an observation comparing the cell decomposition and the Plancherel formula in type $\tilde{C}_2$. This observation was conjectured in \cite{GP:17} to hold in arbitrary affine type, and here we confirm this conjecture for type~$\tilde{C}_2$. 
\medskip

It is convenient to group the representations that appear under the integral signs in the Plancherel formula into classes $\Pi_0=\{\pi_0^{\zeta}\mid \zeta\in\mathbb{T}^2\}$ and $\Pi_i=\{\pi_i^{\zeta}\mid \zeta\in\mathbb{T}\}$ for $i=1,2,3$. The remaining representations (the ``point masses'') are taken to be in their own classes: $\Pi_j=\{\pi_j\}$ for $4\leq j\leq 12$. For each choice $(r_1,r_2)$ (with $r_1=b/a$ and $r_2=c/a$ as usual) let $\sfp(r_1,r_2)$ denote the set of classes that appear with nonzero coefficient in the Plancherel formula. For example, if $(r_1,r_2)\in A_1$ then 
\begin{align}\label{eq:piex}
\sfp(r_1,r_2)=\{\Pi_0,\Pi_1,\Pi_2,\Pi_3,\Pi_4,\Pi_5,\Pi_6,\Pi_9,\Pi_{12}\}.
\end{align}

Let $\rho_0,\ldots,\rho_{13}$ denote the representations of the balanced system of cell representations constructed in Theorem~\ref{thm:balancedsystems}. Thus typically $\rho_j=\pi_j$, with only the following exceptions: In equal parameters we have $\rho_2=\pi_2^{\zeta}\oplus\pi_5\oplus \pi_6$, for $(r_1,r_2)\in \{(r,1)\mid r\geq 1\}$ we have $\rho_{13}=\pi_5\oplus \pi_7\oplus \pi_{12}$, and for $(r_1,r_2)\in A_{2,3}$ we have $\rho_{13}=\pi_6\oplus \pi_{12}\oplus \pi_{10}$.

\begin{Prop}
For each choice $(r_1,r_2)\in\mathbb{Q}_{>0}^2$ there is a well defined surjective map $\Omega:\sfp(r_1,r_2)\to\tsc(r_1,r_2)$ given by 
$$
\Omega(\Pi_j)=\Ga_j\quad\text{if $\pi_j$ is a submodule of $\rho_j$ (as representations of $\cH_g$)}.
$$
Moreover, on each open region $A_j$ the map $\Omega$ is bijective. 
\end{Prop}

\begin{proof}
This is by direct observation for each parameter regime. For example, consider $(r_1,r_2)\in A_1$. In this case $\sfp(r_1,r_2)$ is as in~(\ref{eq:piex}), and from Figure~\ref{fig:partition} we have $\tsc(r_1,r_2)=\{\Ga_0,\Ga_1,\Ga_2,\Ga_3,\Ga_4,\Ga_5,\Ga_6,\Ga_9,\Ga_{12}\}$, and the result follows in this case.
\medskip

For another example, consider $(r_1,r_2)\in A_{2,3}$. Thus $r_1=1$ and $0<r_2<1$. Then $$\sfp(r_1,r_2)=\{\Pi_0,\Pi_1,\Pi_2,\Pi_3,\Pi_4,\Pi_5,\Pi_6,\Pi_{10},\Pi_{12}\}\quad\text{and}\quad \tsc(r_1,r_2)=\{\Ga_0,\Ga_1,\Ga_2,\Ga_3,\Ga_4,\Ga_5,\Ga_{13}\}.
$$ 
Thus $\Omega(\Pi_j)=\Ga_j$ for $j\in\{0,1,2,3,4,5\}$, and $\Omega(\Pi_6)=\Omega(\Pi_{10})=\Omega(\Pi_{12})=\Ga_{13}$ (recall that $\rho_{13}=\pi_6\oplus \pi_{12}\oplus\pi_{10}$). 
\medskip

As a final example, consider $(r_1,r_2)=(1,1)=P_2$ (equal parameters). In this case
$$
\sfp(r_1,r_2)=\{\Pi_0,\Pi_1,\Pi_2,\Pi_4,\Pi_5,\Pi_6\}\quad\text{and}\quad \tsc(r_1,r_2)=\{\Ga_0,\Ga_1,\Ga_2,\Ga_4\}.
$$
Since $\rho_2=\pi_2^{\zeta}\oplus \pi_5\oplus \pi_6$ we have $\Omega(\Pi_j)=\Ga_j$ for $j\in\{0,1,4\}$ and  $\Omega(\Pi_2)=\Omega(\Pi_5)=\Omega(\Pi_6)=\Ga_2$. All remaining cases are similar.
\end{proof}

We will sometimes write $\Omega(\pi)$ in place of $\Omega(\Pi)$ if $\pi$ is a member of the class $\Pi$.

\subsection{The asymptotic Plancherel formula}

Each rational function $f(\sq)=a(\sq)/b(\sq)$ can be written as $f(\sq)=\sq^{-N}a'(\sq^{-1})/b'(\sq^{-1})$ with $N\in\mathbb{Z}$  where $a'(\sq^{-1})$ and $b'(\sq^{-1})$ are  polynomials in $\sq^{-1}$ nonvanishing at $\sq^{-1}=0$. The integer $N$ in this expression is uniquely determined, and is called the \textit{$\sq^{-1}$-valuation} of $f$, written $\nu_{\sq}(f)=N$. For example, $\nu_{\sq}((\sq^2+1)(\sq^3+1)/(\sq^7-\sq+1))=2$.

\begin{Def}
Let $\Pi$ be a class of representations appearing in the Plancherel Theorem. Consider the coefficient in the Plancherel formula of a generic character $\chi_{\pi}$ with $\pi\in\Pi$ as a rational function $C=C(\sq)$ in $\sq$ by setting $q=\sq$. The \textit{$\sq^{-1}$-valuation} of $\Pi$ is defined to be $\nu_{\sq}(\Pi)=\nu_{\sq}(C(\sq))$. We also write $\nu_{\sq}(\pi)=\nu_{\sq}(\Pi)$ for any $\pi\in\Pi$. 
\end{Def}

Recall that we have seen that Lusztig's $\ba$-function is constant on two-sided cells, and thus we may write $\ba(\Ga)$ for the value of $\ba(w)$ for any $w\in\Ga$. Moreover the values of the $\ba$-function are given in Table~\ref{tab:afunction} (and the discussion immediately following the table; see Corollary~\ref{cor:afn}). The following remarkable property of the Plancherel measure has an analogue in the finite dimensional case where the Plancherel measure is replaced by the ``generic degrees'' of the Hecke algebra (see \cite[Chapter~11]{GP:00} and~\cite{Geck:11}).

\begin{Th}\label{thm:muprime}
For each classes $\Pi$ appearing in the Plancherel formula in type $\tilde{C}_2$ we have
$
\nu_{\sq}(\Pi)=2\ba(\Omega(\Pi)).
$
\end{Th}

\begin{proof}
If $\nu_{\sq}(f(\sq))=N$ then we write 
$
f(\sq)\sim C\sq^{-N}
$
where $C$ is the specialisation at $\sq^{-1}=0$ of $\sq^{\alpha}f(\sq)$. Thus $C\sq^{-\alpha}$ is the ``leading term'' of $f(\sq)$ when $f(\sq)$ is expressed as a Laurent power series in $\sq^{-1}$. Then we compute, directly from Theorem~\ref{thm:plancherel},
\begin{align*}
\frac{|C_0|}{|c_0(\zeta)|^2}&\sim\frac{1}{8}\times\begin{cases}
q^{-4a-4b}|(1-\zeta_1^{-1})(1-\zeta_1^{-1}\zeta_2^{-2})(1-\zeta_1^{-2}\zeta_2^{-2})(1-\zeta_2^{-2})|^2&\text{if $r_2<r_1$}\\
q^{-4a-4c}|(1-\zeta_1^{-1})(1-\zeta_1^{-1}\zeta_2^{-2})(1-\zeta_1^{-2}\zeta_2^{-2})(1-\zeta_2^{-2})|^2&\text{if $r_2>r_1$}\\
q^{-4a-4b}|(1-\zeta_1^{-1})(1-\zeta_1^{-1}\zeta_2^{-2})(1-\zeta_1^{-1}\zeta_2^{-1})(1-\zeta_2^{-1})|^2&\text{if $r_1=r_2$}
\end{cases}
\end{align*}
showing that $\nu_{\sq}(\Pi_0)=2\ba(\Ga_0)$ for all choices of parameters. Similarly we compute\newline
\begin{minipage}{0.5\linewidth}
\vspace{0pt}
\begin{align*}
\frac{|C_1|}{|c_1(\zeta)|^2}&\sim\frac{1}{2}\times\begin{cases}
q^{2a-4c}|1-\zeta^{-2}|^2&\text{if $(r_1,r_2)\in A$}\\
q^{-2b-2c}|1-\zeta^{-2}|^2&\text{if $(r_1,r_2)\in B$}\\
q^{2a-4b}|1-\zeta^{-2}|^2&\text{if $(r_1,r_2)\in C$}\\
q^{-2a}|1-\zeta^{-2}|^2&\text{if $(r_1,r_2)\in D$}\\
q^{-2b-2c}|1+\zeta^{-1}|^2&\text{if $(r_1,r_2)\in E$}\\
q^{2a-4b}|1+\zeta^{-1}|^2&\text{if $(r_1,r_2)\in F$}\\
q^{-2b-2c}|1-\zeta^{-1}|^2&\text{if $(r_1,r_2)\in G$}
\end{cases}
\end{align*}
\end{minipage}
\begin{minipage}{0.5\linewidth}
\vspace{20pt}
\begin{tikzpicture}[scale=1]
\draw (0,0)--(3,0);
\draw (0,0)--(0,3);
\draw (0,1)--(1,0);
\draw (0,1)--(2,3);
\draw (1,0)--(3,2);
\node at (0.5,2.5) {\small{$A$}};
\node at (1.5,1.5) {\small{$B$}};
\node at (2.5,0.5) {\small{$C$}};
\node at (0.25,0.25) {\small{$D$}};
\node at (1,2) {\small{$E$}};
\node at (2,1) {\small{$F$}};
\node at (0.5,0.5) {\small{$G$}};
\node at (1,-0.25) {\small{$1$}};
\node at (-0.25,1) {\small{$1$}};
\end{tikzpicture}
\end{minipage}

\begin{minipage}{0.5\linewidth}
\vspace{0pt}
\begin{align*}
\frac{|C_2|}{|c_2(\zeta)|^2}&\sim\frac{1}{2}\times\begin{cases}
q^{-2c}|1-\zeta^{-2}|^2&\text{if $(r_1,r_2)\in A$}\\
q^{-2b}|1-\zeta^{-2}|^2&\text{if $(r_1,r_2)\in B$}\\
q^{-4a+2c}|1-\zeta^{-2}|^2&\text{if $(r_1,r_2)\in C$}\\
q^{-4a+2b}|1-\zeta^{-2}|^2&\text{if $(r_1,r_2)\in D$}\\
q^{-2b}|1-\zeta^{-1}|^2&\text{if $(r_1,r_2)\in E$}\\
q^{-4a+2c}|1+\zeta^{-1}|^2&\text{if $(r_1,r_2)\in F$}\\
q^{-4a+2c}|1-\zeta^{-1}|^2&\text{if $(r_1,r_2)\in G$}\\
q^{-4a+2b}|1+\zeta^{-1}|^2&\text{if $(r_1,r_2)\in H$}\\
q^{-2a}&\text{if $(r_1,r_2)\in I$}
\end{cases}
\end{align*}
\end{minipage}
\begin{minipage}{0.5\linewidth}
\vspace{20pt}
\begin{tikzpicture}[scale=1.25]
\draw (0,0)--(2.5,0);
\draw (0,0)--(0,2.5);
\draw (0,0)--(2.25,2.25);
\draw (0,2)--(2,0);
\node at (-0.25,2) {\small{$2$}};
\node at (2,-0.25) {\small{$2$}};
\node at (1,1.5) {\small{$A$}};
\node at (1.5,1.5) {\small{$E$}};
\node at (1.5,1) {\small{$B$}};
\node at (1.5,0.5) {\small{$F$}};
\node at (1,0.5) {\small{$C$}};
\node at (0.5,0.5) {\small{$G$}};
\node at (0.5,1) {\small{$D$}};
\node at (0.5,1.5) {\small{$H$}};
\node at (1,1) {\small{$I$}};
\end{tikzpicture}
\end{minipage}

\begin{minipage}{0.5\linewidth}
\vspace{0pt}
\begin{align*}
\frac{|C_3|}{|c_3(\zeta)|^2}&\sim\frac{1}{2}\times\begin{cases}
q^{-2b-2c}|1-\zeta^{-2}|^2&\text{if $(r_1,r_2)\in A$}\\
q^{-4a-4b}|1-\zeta^{-2}|^2&\text{if $(r_1,r_2)\in B$}\\
q^{-4a-4c}|1-\zeta^{-2}|^2&\text{if $(r_1,r_2)\in C$}\\
q^{-2b-2c}|1-\zeta^{-2}|^2&\text{if $(r_1,r_2)\in D$}\\
q^{-4a-4b}|1-\zeta^{-1}|^2&\text{if $(r_1,r_2)\in E$}\\
q^{-4a-4c}|1-\zeta^{-1}|^2&\text{if $(r_1,r_2)\in F$}
\end{cases}
\end{align*}
\end{minipage}
\begin{minipage}{0.5\linewidth}
\vspace{20pt}
\begin{tikzpicture}[scale=1]
\draw (0,0)--(3,0);
\draw (0,0)--(0,3);
\draw (0,2)--(1,3);
\draw[style = dashed] (0,0)--(3,3);
\draw (2,0)--(3,1);
\node at (0.25,2.75) {\small{$A$}};
\node at (1,2) {\small{$B$}};
\node at (2,1) {\small{$C$}};
\node at (2.5,0.5) {\small{$F$}};
\node at (2.75,0.25) {\small{$D$}};
\node at (0.5,2.5) {\small{$E$}};
\node at (2,-0.25) {\small{$2$}};
\node at (-0.25,2) {\small{$2$}};
\end{tikzpicture}
\end{minipage}
and thus $\nu_{\sq}(\Pi_i)=2\ba(\Omega(\Pi_i))$ for all $i=1,2,3$ and all choices of parameters.

For the point masses we have\newline
\begin{minipage}{0.5\linewidth}
\vspace{0pt}
\begin{align*}
|C_5|&\sim\begin{cases}
q^{-4a+2b}&\text{if $(r_1,r_2)\in A$}\\
q^{-2a}&\text{if $(r_1,r_2)\in B\cup E$}\\
q^{-4a+2c}&\text{if $(r_1,r_2)\in C$}\\
q^{-2c}&\text{if $(r_1,r_2)\in D$}\\
q^{-2b}&\text{if $(r_1,r_2)\in F$}\\
q^{-2b-2c}&\text{if $(r_1,r_2)\in G$}\\
q^{-2a}/2&\text{if $(r_1,r_2)\in H\cup I$}\\
q^{-2c}/2&\text{if $(r_1,r_2)\in J$}\\
q^{-2b}/2&\text{if $(r_1,r_2)\in K$}
\end{cases}
\end{align*}
\end{minipage}
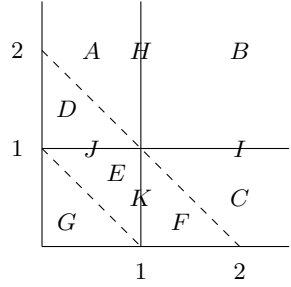
\begin{minipage}{0.5\linewidth}
\vspace{20pt}
\begin{tikzpicture}[scale=1.3]
\draw (0,0)--(2.5,0);
\draw (0,0)--(0,2.5);
\draw[style = dashed] (0,2)--(2,0);
\draw[style = dashed] (0,1)--(1,0);
\draw (1,0)--(1,2.5);
\draw (0,1)--(2.5,1);
\node at (0.5,2) {\small{$A$}};
\node at (2,2) {\small{$B$}};
\node at (2,0.5) {\small{$C$}};
\node at (0.25,1.4) {\small{$D$}};
\node at (1.4,0.25) {\small{$F$}};
\node at (0.75,0.75) {\small{$E$}};
\node at (0.25,0.25) {\small{$G$}};
\node at (1,2) {\small{$H$}};
\node at (2,1) {\small{$I$}};
\node at (0.5,1) {\small{$J$}};
\node at (1,0.5) {\small{$K$}};
\node at (2,-0.25) {\small{$2$}};
\node at (-0.25,2) {\small{$2$}};
\node at (1,-0.25) {\small{$1$}};
\node at (-0.25,1) {\small{$1$}};
\end{tikzpicture}
\end{minipage}
\newline
\begin{minipage}{0.5\linewidth}
\vspace{0pt}
\begin{align*}
|C_6|&\sim\begin{cases}
q^{-2a}&\text{if $(r_1,r_2)\in A\cup D$}\\
q^{-2b}&\text{if $(r_1,r_2)\in B$}\\
q^{-2c}&\text{if $(r_1,r_2)\in C$}\\
q^{-4a+2c}&\text{if $(r_1,r_2)\in E$}\\
q^{-4a+2b}&\text{if $(r_1,r_2)\in F$}\\
q^{-2a}/2&\text{if $(r_1,r_2)\in G\cup H$}\\
q^{-4a+2c}/2&\text{if $(r_1,r_2)\in I$}\\
q^{-4a+2b}/2&\text{if $(r_1,r_2)\in J$}
\end{cases}
\end{align*}
\end{minipage}
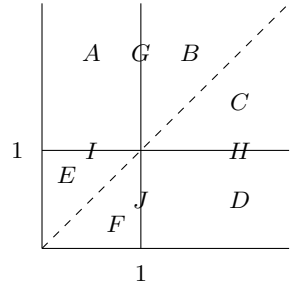
\begin{minipage}{0.5\linewidth}
\vspace{20pt}
\begin{tikzpicture}[scale=1.3]
\draw (0,0)--(2.5,0);
\draw (0,0)--(0,2.5);
\draw[style = dashed] (0,0)--(2.5,2.5);
\draw (1,0)--(1,2.5);
\draw (0,1)--(2.5,1);
\node at (0.5,2) {\small{$A$}};
\node at (1.5,2) {\small{$B$}};
\node at (2,1.5) {\small{$C$}};
\node at (2,0.5) {\small{$D$}};
\node at (0.25,0.75) {\small{$E$}};
\node at (0.75,0.25) {\small{$F$}};
\node at (1,2) {\small{$G$}};
\node at (2,1) {\small{$H$}};
\node at (0.5,1) {\small{$I$}};
\node at (1,0.5) {\small{$J$}};
\node at (1,-0.25) {\small{$1$}};
\node at (-0.25,1) {\small{$1$}};
\end{tikzpicture}
\end{minipage}
\newline
\begin{minipage}{0.5\linewidth}
\vspace{0pt}
\begin{align*}
|C_7|&\sim\begin{cases}
q^{-4a-4b}&\text{if $(r_1,r_2)\in A$}\\
q^{-2b-2c}&\text{if $(r_1,r_2)\in B$}\\
q^{2a-4c}&\text{if $(r_1,r_2)\in C$}\\
q^{-2c}&\text{if $(r_1,r_2)\in D$}\\
q^{2a-4c}/2&\text{if $(r_1,r_2)\in E$}
\end{cases}
\end{align*}
\end{minipage}
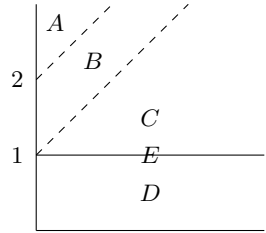
\begin{minipage}{0.5\linewidth}
\vspace{20pt}
\begin{tikzpicture}[scale=1]
\draw (0,0)--(3,0);
\draw (0,0)--(0,3);
\draw[style = dashed] (0,2)--(1,3);
\draw[style = dashed] (0,1)--(2,3);
\draw (0,1)--(3,1);
\node at (0.25,2.75) {\small{$A$}};
\node at (0.75,2.25) {\small{$B$}};
\node at (1.5,1.5) {\small{$C$}};
\node at (1.5,0.5) {\small{$D$}};
\node at (1.5,1) {\small{$E$}};
\node at (-0.25,2) {\small{$2$}};
\node at (-0.25,1) {\small{$1$}};
\phantom{\node at (1,-0.25) {\small{$1$}};}
\end{tikzpicture}
\end{minipage}
\newline
\begin{minipage}{0.5\linewidth}
\vspace{0pt}
\begin{align*}
|C_8|&\sim\begin{cases}
q^{-4a-4c}&\text{if $(r_1,r_2)\in A$}\\
q^{-2b-2c}&\text{if $(r_1,r_2)\in B$}\\
q^{2a-4b}&\text{if $(r_1,r_2)\in C$}\\
q^{-2b}&\text{if $(r_1,r_2)\in D$}\\
q^{2a-4b}/2&\text{if $(r_1,r_2)\in E$}
\end{cases}
\end{align*}
\end{minipage}
\begin{minipage}{0.5\linewidth}
\vspace{20pt}
\begin{tikzpicture}[scale=1]
\draw (0,0)--(0,3);
\draw (0,0)--(3,0);
\draw[style = dashed] (2,0)--(3,1);
\draw[style = dashed] (1,0)--(3,2);
\draw (1,0)--(1,3);
\node at (2.75,0.25) {\small{$A$}};
\node at (2.25,0.75) {\small{$B$}};
\node at (1.5,1.5) {\small{$C$}};
\node at (0.5,1.5) {\small{$D$}};
\node at (1,1.5) {\small{$E$}};
\node at (2,-0.25) {\small{$2$}};
\node at (1,-0.25) {\small{$1$}};
\phantom{\node at (-0.25,1) {\small{$1$}};}
\end{tikzpicture}
\end{minipage}
and the result follows (by comparison with Table~\ref{tab:afunction}).
\end{proof}

\begin{Def} Using Theorem~\ref{thm:muprime} we define the \textit{asymptotic Plancherel measure} on $\irreps$ by
$$
d\mu'(\pi)=\lim_{q\to\infty} q^{2\ba(\Omega(\Pi))}d\mu(\pi)\quad\text{for all $\pi\in\Pi$}.
$$
\end{Def}

\begin{Th}\label{thm:asymptoticplancherel}
For $r_2\leq r_1$ the asymptotic Plancherel measure is as follows. The case $r_2>r_1$ may be obtained by applying $\sigma$. For the infinite cells we have
\begin{align*}
\mu'(\pi_0^{\zeta})&=\frac{1}{8}\times\begin{cases}
|(1-\zeta_1^{-1})(1-\zeta_1^{-1}\zeta_2^{-2})(1-\zeta_1^{-2}\zeta_2^{-2})(1-\zeta_2^{-2})|^2&\text{if $r_2\neq r_1$}\\
|(1-\zeta_1^{-1})(1-\zeta_1^{-1}\zeta_2^{-2})(1-\zeta_1^{-1}\zeta_2^{-1})(1-\zeta_2^{-1})|^2&\text{if $r_2=r_1$}
\end{cases}\\
\mu'(\pi_1^{\zeta})&=\frac{1}{2}\times\begin{cases}
|1-\zeta^{-2}|^2&\text{if $r_2\neq r_1-1$ and $r_2\neq 1-r_1$}\\
|1-\zeta^{-1}|^2&\text{if $r_2=1-r_1$}\\
|1+\zeta^{-1}|^2&\text{if $r_2=r_1+1$}
\end{cases}\\
\mu'(\pi_2^{\zeta})&=\frac{1}{2}\times\begin{cases}
|1-\zeta^{-2}|^2&\text{if $r_2\neq r_1$ and $r_2\neq 2-r_1$}\\
|1-\zeta^{-1}|^2&\text{if $r_2=r_1$ and $r_2\neq 2-r_1$}\\
|1+\zeta^{-1}|^2&\text{if $r_2\neq r_1$ and $r_2=2-r_1$}\\
1&\text{if $(r_1,r_2)=(1,1)$}
\end{cases}\\
\mu'(\pi_3^{\zeta})&=\frac{1}{2}\times\begin{cases}
|1-\zeta^{-2}|^2&\text{if $r_2\neq r_1$ and $r_2=r_1-2$}\\
|1-\zeta^{-1}|^2&\text{if $r_2\neq r_1$ and $r_2=r_1-2$}\\
0&\text{if $r_2=r_1$}
\end{cases}
\end{align*}
and for the square integrable representations we have $\mu'(\pi_4)=1$ for all $(r_1,r_2)$, and 
\begin{align*}
\mu'(\pi_5)&=\begin{cases}
1&\text{if $r_2\neq 1$}\\
\frac{1}{2}&\text{if $r_2=1$}
\end{cases}&
\mu'(\pi_6)&=\begin{cases}
1&\text{if $r_1\neq 1$}\\
\frac{1}{2}&\text{if $r_1=1$}
\end{cases}\\
\mu'(\pi_7)&=\begin{cases}
1&\text{if $r_2>2-r_1$ and $r_2\neq 1$}\\
\frac{1}{2}&\text{if $r_2>2-r_1$ and $r_2=1$}\\
0&\text{if $r_2\leq 2-r_1$}
\end{cases}&
\mu'(\pi_8)&=\begin{cases}
1&\text{if $r_2<r_1-2$}\\
0&\text{if $r_2\geq r_1-2$}
\end{cases}\\
\mu'(\pi_9)&=\begin{cases}
1&\text{if $r_2<1-r_1$}\\
0&\text{otherwise}
\end{cases}&
\mu'(\pi_{10})&=\begin{cases}
1&\text{if $1-r_1<r_2<2-r_1$ and $r_1\neq 1$}\\
\frac{1}{2}&\text{if $1-r_1<r_2<2-r_1$ and $r_1=1$}\\
0&\text{otherwise}
\end{cases}\\
\mu'(\pi_{11})&=\begin{cases}
1&\text{if $r_1-2<r_2<r_1-1$}\\
0&\text{otherwise}
\end{cases}&
\mu'(\pi_{12})&=\begin{cases}
1&\text{if $r_2<r_1$, $r_1\neq 1$ and $r_2\neq 1$}\\
\frac{1}{2}&\text{if $r_2<r_1$ and either $r_1=1$ or $r_2=1$}\\
0&\text{otherwise}
\end{cases}
\end{align*}
\end{Th}

\begin{proof}
This follows directly from the computations made in the proof of Theorem~\ref{thm:muprime}.
\end{proof}

\subsection{Conjecture \conj{1}}

We can now prove that \conj{1} holds for $\tilde{C}_2$, following the technique of~\cite{GP:17}.

\begin{Th}\label{thm:P1} Lusztig's conjecture \conj{1} holds for $\tilde{C}_2$ for all choices of parameters.
\end{Th}

\begin{proof}
Recall that $\Delta(w)$ is defined by $P_{e,w}=n_w\sq^{-\Delta(w)}+\text{(strictly smaller powers of $\sq$)}$, where $n_w\neq 0$. We are required to prove that $\ba(w)\leq\Delta(w)$. This is equivalent to showing that
$$
\lim_{q\to\infty}q^{\ba(w)}P_{e,w}(q)<\infty,
$$
where we write $P_{e,w}(q)$ for the specialisation of $P_{e,w}$ at $\sq=q$. By the Plancherel Theorem we have
\begin{align*}
q^{\ba(w)}P_{e,w}(q)=q^{\ba(w)}\mathrm{Tr}(C_w)&=\int_{\irreps}q^{\ba(w)}\chi_{\pi}(C_w)\,d\mu(\pi).
\end{align*}
Suppose that $w$ is in the two-sided cell~$\Ga$, and hence $\ba(w)=\ba(\Ga)$. Since the representations $\pi_i$ satisfy $\B{1}$ and $\B{2}$ for their respective cell $\Omega(\pi_i)$, it follows that the integral above is over only those classes of representations $\pi\in\Pi\in\Omega^{-1}(\Gamma')$ with $\Ga\geq_{\cLR}\Gamma'$. For each such class of representations the Plancherel measure is, by Theorem~\ref{thm:muprime}, of the form 
$$
d\mu(\pi)=q^{-2\ba(\Ga')}(1+\mathcal{O}(q^{-1}))d\mu'(\pi)
$$ 
where $d\mu'$ is the asymptotic Plancherel measure. Thus the integrand (with respect to the asymptotic Plancherel measure) is $q^{\ba(\Ga)-\ba(\Ga')}\mathrm{tr}(\fc_{\pi}(w))(1+\mathcal{O}(q^{-1}))$. Since $\Gamma\geq_{\cLR}\Gamma'$ we have $\ba(\Ga')\geq \ba(\Ga)$ (by $\conj{4}$) and thus the power of $q$ in the integrand is at most~$0$. It is clear from the explicit $\tilde{C}_2$ Plancherel Theorem that the limit may be passed under the integral sign, and the result follows. 
\end{proof}

\subsection{Duflo involutions and conjecture $\conj{6}$}

In this section we compute the Duflo elements for each cell. We recall that for $r_2\leq r_1$ all cells admit a cell factorisation (perhaps within the extended affine Weyl group) with the exceptions
\begin{align*}
&\Ga_1\quad\text{in the case $(r_1,r_2)\in R_{1,2}^{1}=\{(r,r')\mid r'\leq r,\,r'=1-r\}$},\\
&\Ga_1\quad\text{in the case $(r_1,r_2)\in R_{2,3}^{1}=\{(r,r')\mid r'=r-1\}$},\\
&\Ga_2\quad\text{in the case $(r_1,r_2)\in R_{1,2}^{2}=\{(r,r')\mid r'<r,\,r'=2-r\}$},\\
&\Ga_3\quad\text{in the case $(r_1,r_2)\in R_{1,2}^{3}=\{(r,r')\mid r'=r-2\}$},\\
&\Ga_2\quad\text{in the case $(r_1,r_2)=(1,1)$},\\
&\Ga_{13}\,\,\,\,\text{in all cases in which this cell appears.}
\end{align*}

\begin{Th}\label{thm:Duflo}
Let $\Ga\in\tsc$. The Duflo elements~$\cD_{\Ga}=\cD\cap\Ga$ are as follows. If $\Gamma$ admits a cell factorisation then
$$
\cD_{\Ga}=\{u^{-1}\sw_{\Ga} u\mid u\in\sB_{\Ga}\}.
$$
If $\Ga$ does not admit a cell factorisation then (in the local notation of the relevant subsection~\ref{sec:Ga2}--\ref{sec:Ga3})
\begin{align*}
\cD_{\Ga_1}&=\{\sw_2\}\cup\{u^{-1}\sw_1u\mid u\in\sB_1\backslash\{02\}\}&&\text{if $(r_1,r_2)\in R_{1,2}^{1}$}\\
\cD_{\Ga_1}&=\{\sw_3\}\cup\{u^{-1}\sw_2u\mid u\in\sB_2\backslash\{12\}\}&&\text{if $(r_1,r_2)\in R_{2,3}^{1}$}\\
\cD_{\Ga_2}&=\{\sw_1\}\cup\{u^{-1}\sw_2u\mid u\in\sB_2\backslash\{101\}\}&&\text{if $(r_1,r_2)\in R_{1,2}^{2}$}\\
\cD_{\Ga_3}&=\{\sw_1\}\cup\{u^{-1}\sw_2u\mid u\in\sB_2\backslash\{101\}\}&&\text{if $(r_1,r_2)\in R_{1,2}^{3}$}\\
\cD_{\Ga_2}&=\{0,1,2\}&&\text{if $(r_1,r_2)=(1,1)$}\\
\cD_{\Ga_{13}}&=\{0,1\}&&\text{if $(r_1,r_2)\in A_{4,5}\cup A_{7,8}\cup A_{9,10}\cup P_4\cup P_5$}\\
\cD_{\Ga_{13}}&=\{1,2,010\}&&\text{if $(r_1,r_2)\in A_{2,3}$}.
\end{align*}
\end{Th}

\begin{proof}
Let $n_w'$ be the coefficient of $\sq^{-\ba(w)}$ in $P_{e,w}$. Thus $w\in\cD$ if and only if $n_w'\neq 0$ (and in this case $n_w'=n_w$). Moreover, from the asymptotic Plancherel formula we have (see the proof of Theorem~\ref{thm:P1})
\begin{align}\label{eq:nw}
n_w'=\int_{\Omega^{-1}(\Gamma)}\mathrm{tr}(\fc_{\pi}(w))\,d\mu'(\pi)\quad\text{if $w\in\Gamma$}.
\end{align}
In particular, for $w\in\Gamma_i$ with $4\leq i\leq 12$ we have, using Theorem~\ref{thm:finite}, 
$$
n_w'=\mathrm{tr}(\fc_{\pi_i}(w))d\mu'(\pi_i)=\pm\mathrm{tr}(E_{\su_w,\sv_w})d\mu'(\pi_i),
$$
and thus $n_w'\neq 0$ if and only if $w\in\{u^{-1}\sw_{\Ga_i}u\mid u\in\sB_{\Ga_i}\}$ as claimed. 
\medskip

For the infinite cells admitting a cell factorisation the analysis is as follows. Consider the lowest two-sided cell~$\Ga_0$. If $r_2\neq r_1$ then Theorem~\ref{thm:Ga0} and the asymptotic Plancherel formula give (for $w\in\Ga_0$) 
$$
n_w'=\frac{1}{8}\int_{\mathbb{T}^2}\fs_{\tau_w}(\zeta)\mathrm{tr}(E_{\su_w,\sv_w})|(1-\zeta_1^{-1})(1-\zeta_1^{-1}\zeta_2^{-2})(1-\zeta_1^{-2}\zeta_2^{-2})(1-\zeta_2^{-2})|^2\,d\zeta_1 d\zeta_2.
$$
It is well known that the Schur functions~$\fs_{\lambda}(\zeta)$ defined in equation~(\ref{eq:schur}) are orthonormal with respect to the measure $\frac{1}{8}|(1-\zeta_1^{-1})(1-\zeta_1^{-1}\zeta_2^{-2})(1-\zeta_1^{-2}\zeta_2^{-2})(1-\zeta_2^{-2})|^2\,d\zeta_1 d\zeta_2$, and it follows that $n_w'=0$ unless $\tau_w=0$ and $\su_w=\sv_w$, in which case $n_w'=1$. Hence the result in this case. If $r_2=r_1$ then the analysis is similar, since the Schur functions $\fs_{\lambda}'(\zeta)$ defined in~(\ref{eq:schurdual}) are orthonormal with respect to the asymptotic Plancherel measure $\frac{1}{8}|(1-\zeta_1^{-1})(1-\zeta_1^{-1}\zeta_2^{-2})(1-\zeta_1^{-1}\zeta_2^{-1})(1-\zeta_2^{-1})|^2\,d\zeta_1 d\zeta_2$. 
\medskip

Now consider the case $\Ga_2$ with $(r_1,r_2)\in R_2=\{(r,r')\mid r'<r,\,r'>2-r\}$. In this case we have (see Section~\ref{sec:Ga2})
$$
\Ga_2=\{u^{-1}\sw_2\st_2^kv\mid u,v\in\sB_2,\,k\geq 0\}
$$
where $\sw_2=2$, $\st_2=1012$, and $\sB_2=\{e,1,10,101\}$. Using the formula in Theorem~\ref{thm:gamma2part1} we have $\fc_{\pi_2^{\zeta}}(w;\sB_2')=\fs_{\tau_w}(\zeta)E_{\su_w,\sv_w}$ for $w\in\Ga_2$, where $\fs_k$ is the Schur function of type~$A_1$. The asymptotic Plancherel formula gives 
$$
n_w'=\frac{1}{2}\int_{\mathbb{T}}\fs_{\tau_w}(\zeta)\mathrm{tr}(E_{\su_w,\sv_w})|1-\zeta^{-2}|^2\,d\zeta,
$$
and since the Schur functions of type $A_1$ are orthonormal with respect to the measure $\frac{1}{2}|1-\zeta^{-2}|^2\,d\zeta$ it follows that $n_w'\neq 0$ if and only if $\su_w=\sv_w$ and $\tau_w=0$. Thus $w\in\{u^{-1}\sw_2 u\mid u\in\sB_2\}$ as claimed. 
\medskip

The remaining cases admitting a cell factorisation are similar. However there are slight modifications in the cases $\Ga_2$ with $r_2=r_1\neq 1$ where we have $\fc_{\pi_2^{\zeta}}(w)=\pm \fs_{\tau_w}(\zeta^{1/2})E_{\su_w,\sv_w}$. Here the asymptotic Plancherel measure is $\frac{1}{2}|1-(\zeta^{1/2})^{-2}|^2$ and so the same analysis applies. 
\medskip

We now consider the cells that do not admit cell factorisations. Consider the case $\Ga_2$ with $(r_1,r_2)\in R_{1,2}^2$. Here we have 
$$
\Ga_2=\{\sw_1\}\cup \{u^{-1}\sw_2\st_2^kv\mid u,v\in\sB_2,\,k\geq 0\}
$$
where $\sw_1=101$ and $\sw_2$, $\st_2$, and $\sB_2$ are as above. Recall that we extend the cell factorisation in $\Ga_2(R_2)$ to the element $\sw_1$ by setting $\tau_{\sw_1}=-1$ and $\su_{\sw_1}=\sv_{\sw_1}=101$. The asymptotic Plancherel formula, along with Theorem~\ref{thm:gamma2part1}, gives
\begin{align*}
n_w'&=\frac{1}{2}\int_{\mathbb{T}}\mathfrak{f}_{\tau_w}^{\su_w,\sv_w}\mathrm{tr}(E_{\su_w,\sv_w})|1+\zeta^{-1}|^2\,d\zeta=\frac{1}{2}\delta_{\su_w,\sv_w}\int_{\mathbb{T}}\mathfrak{f}_{\tau_w}^{\su_w,\su_w}\,|1+\zeta^{-1}|^2\,d\zeta
\end{align*}
where
$$
\mathfrak{f}_{k}^{u,u}=\begin{cases}
\fs_k(\zeta)-\fs_{k-1}(\zeta)=\fs_{2k}(-\zeta^{1/2})&\text{if $u\neq 101$}\\
\fs_k(\zeta)-\fs_{k+1}(\zeta)=-\fs_{2k+2}(-\zeta^{1/2})&\text{if $u=101$}.
\end{cases}
$$
Since the elements $\fs_{2k}(-\zeta^{1/2})$ are orthonormal with respect to the measure $|1+\zeta^{-1}|^2\,d\zeta$ the result follows. The first $4$ cases listed at the beginning of this section are similar. 
\medskip

Now consider the cell $\Ga_2$ in the equal parameter case. We have $\Omega^{-1}(\Ga_2)=\Pi_2\cup \Pi_5\cup\Pi_6$, and so Theorem~\ref{thm:gamma2part3} and the asymptotic Plancherel formula give (for $w\in\Ga_2$)
$$
n_w'=\frac{1}{2}\int_{\mathbb{T}}\mathrm{tr}(\fc_{\pi_2^{\zeta}}(w))\,d\zeta+\frac{1}{2}\mathrm{tr}(\fc_{\pi_5}(w))+\frac{1}{2}\mathrm{tr}(\fc_{\pi_6}(w)),
$$
where the matrices $\fc_{\pi_2^{\zeta}}(w)$ are obtained from the matrices in Theorem~\ref{thm:gamma2part3} by removing the $5$th and $6$th rows and columns. Since $\int_{\mathbb{T}}\zeta^k\,d\zeta=\delta_{k,0}$ we obtain 
$$
\frac{1}{2}\int_{\mathbb{T}}\mathrm{tr}(\fc_{\pi_2^{\zeta}}(w))\,d\zeta=\begin{cases}
1&\text{if $w=1$}\\
\frac{1}{2}&\text{if $w\in\{0,2,010,212\}$}\\
0&\text{otherwise}
\end{cases}
$$
(note that $1=w_{21}^0$). Moreover, we have
$$
\fc_{\pi_5}(w)=\begin{cases}
1&\text{if $w=0$}\\
-1&\text{if $w=010$}\\
0&\text{otherwise}
\end{cases}\quad\text{and}\quad \fc_{\pi_6}(w)=\begin{cases}
1&\text{if $w=2$}\\
-1&\text{if $w=212$}\\
0&\text{otherwise.}
\end{cases}
$$
For example, in the case of $\fc_{\pi_5}(w)$, the above claim follows from the fact that $\gamma_0(w)-\gamma_1(w)-\gamma_2(w)\leq 1$ with equality if and only if $w\in\{0,010\}$, where $\gamma_i(w)$ denotes the number of $i$ generators appearing in any reduced expression of $w$ (note that since the orders $m_{ij}$ of the products $s_is_j$ are even this statistic is well defined). 
\medskip

Putting these facts together gives
$$
n_w'=\begin{cases}
1&\text{if $w\in\{0,1,2\}$}\\
0&\text{otherwise},
\end{cases}
$$
and hence the result for this cell.
\medskip

Finally we consider the finite cell $\Gamma_{13}$. There are two regimes: 
$$
\Omega^{-1}(\Ga_{13})=\begin{cases}\{\pi_5,\pi_7,\pi_{12}^B\}&\text{if $(r_1,r_2)\in R_1=A_{4,5}\cup A_{7,8}\cup A_{9,10}\cup P_4\cup P_5$}\\
\{\pi_6,\pi_{12}^A,\pi_{10}^B\}&\text{if $(r_1,r_2)\in R_2=A_{2,3}$}.
\end{cases}
$$
The asymptotic Plancherel measure is a sum of point masses:
\begin{align*}
d\mu'(\Omega^{-1}(\Ga_{13}))&=\frac{1}{2}\times \begin{cases}
\delta_{\pi_5}+\delta_{\pi_7}+\delta_{\pi_{12}}&\text{if $(r_1,r_2)\in R_1$}\\
\delta_{\pi_6}+\delta_{\pi_{12}}+\delta_{\pi_{10}}&\text{if $(r_1,r_2)\in R_2$},
\end{cases}
\end{align*}
and thus
$
n_w'=\frac{1}{2}\mathrm{tr}(\fc_{\pi_{13}}(w)).
$
The result follows using the formulae for the leading matrices from Theorem~\ref{thm:finite}.\end{proof}

\begin{Cor}\label{cor:P6}
Conjecture $\conj{6}$ holds for all choices of parameters. 
\end{Cor}

\begin{proof}
Using the explicit descriptions in Theorem~\ref{thm:Duflo} it is clear that the elements of $\cD$ are involutions.
\end{proof}

\subsection{An inner product on $\mathcal{J}$ and conjectures $\conj{2}$, $\conj{3}$, $\conj{5}$, $\conj{7}$ and $\conj{13}$}

In this section we endow Lusztig's asymptotic algebra $\mathcal{J}_{\Gamma}$ with a natural inner product inherited from the Plancherel Theorem (a kind of \textit{asymptotic Plancherel Theorem}). As a consequence we obtain a proof of conjectures $\conj{2}$, $\conj{3}$, $\conj{5}$, $\conj{7}$, and~$\conj{13}$.
\medskip

Recall that we have proved in Theorem~\ref{thm:afn} that for each $\Gamma\in\tsc$ we have that Lusztig's asymptotic algebra is isomorphic to the $\mathbb{Z}$-algebra $\mathcal{J}_{\Gamma}$ spanned by the leading matrices $\{\fc_{\pi_{\Gamma},w}\mid w\in \Gamma\}$. We thus identify Lusztig's asymptotic algebra with this concrete algebra, with $J_w\leftrightarrow \fc_{\pi_{\Gamma},w}$. Define an involution $*$ on $\mathcal{J}_{\Gamma}$ by linearly extending $J_w^*=J_{w^{-1}}$.  

\begin{Th}\label{thm:innerp}
Let $\Gamma\in\tsc$. The formula
$$
\langle g_1,g_2\rangle_{\Ga}=\int_{\Omega^{-1}(\Gamma)}\mathrm{tr}(g_1g_2^*)\,d\mu'(\pi)\quad\text{for $g_1,g_2\in \mathcal{J}_{\Gamma}$}
$$
defines an inner product on $\mathcal{J}_{\Gamma}$ with $\{J_w\mid w\in \Gamma\}$ an orthonormal basis.
\end{Th}

\begin{proof}
The proof is exactly as in \cite[Theorem~8.14]{GP:17}.
\end{proof}

\begin{Cor}\label{cor:P71}
Conjectures $\conj{2}$, $\conj{3}$, $\conj{5}$, $\conj{7}$, and $\conj{13}$ hold for all choices of parameters. 
\end{Cor}

\begin{proof}
If $x,y,z\in\Gamma$ then
$
\gamma_{x,y,z}=\langle J_xJ_y,J_{z^{-1}}\rangle_{\Gamma}=\langle J_y,J_{x^{-1}}J_{z^{-1}}\rangle_{\Gamma}=\langle J_yJ_z,J_{x^{-1}}\rangle_{\Gamma}=\gamma_{y,z,x},
$ and hence $\conj{7}$ holds. 
\medskip

Conjectures $\conj{2}$, $\conj{3}$, $\conj{5}$, and $\conj{13}$ will follow easily from the following observation. By Theorem~\ref{thm:Duflo} we see that each right cell $\Up$ contains a unique Duflo involution $d_{\Up}\in\cD$. Using the explicit formulae for the leading matrices we compute directly that for all two-sided cells $\Ga$, and all right cells $\Up\subseteq \Ga$, we have
\begin{align}\label{eq:dufloformula}
\fc_{\pi_{\Ga}}(d_{\Up})\fc_{\pi_{\Ga}}(w)=\begin{cases}
\pm\fc_{\pi_{\Ga}}(w)&\text{if $w\in\Up$}\\
0&\text{if $w\notin\Up$}
\end{cases}
\end{align}
where the sign is independent of~$w$ (and thus depends only on $d_{\Up}$). For example, if $\Gamma$ admits a cell factorisation then $d_{\Up}=u^{-1}\sw_{\Ga}u$ for some $u\in\sB_{\Ga}$ and $\fc_{\pi_{\Ga}}(d_{\Up})=\pm E_{u,u}$. For $w\in \Ga$ we have $\fc_{\Ga}(w)=c\,E_{\su_w,\sv_w}$ for some constant or Schur function~$c$, and thus
$$
\fc_{\pi_{\Ga}}(d_{\Up})\fc_{\pi_{\Ga}}(w)=
\pm c\,E_{u,u}E_{\su_w,\sv_w}=\pm\delta_{u,\su_w}\fc_{\pi_{\Ga}}(w).
$$
Since $w\in \Up$ if and only if $\su_{w}=u$ the result follows (note also that if $w\notin\Gamma$ then $\fc_{\pi_{\Ga}}(w)=0$). For the cases where $\Ga$ does not admit a cell factorisation we have in fact already verified the above formulae in most cases in the course of establishing~$\B{5}$ (see for example Theorem~\ref{thm:finite} for the cell $\Ga_{13}$, and the final lines in Section~\ref{sec:Ga2} for the case $\Ga_2$ with equal parameters). 
\medskip

Consider $\conj{2}$. If $\gamma_{x,y,d}\neq 0$ with $d=d_{\Up}$ then $x,y,d\in\Ga$ for some two-sided cell~$\Ga$. Using $\conj{7}$ we have 
$$
\gamma_{x,y,d}=\gamma_{d,x,y}=\langle J_{d}J_x,J_{y^{-1}}\rangle_{\Ga}.
$$
By~(\ref{eq:dufloformula}) we have $x\in\Up$ (otherwise $J_{d}J_x=0$ and so $\gamma_{x,y,d}=0$) and therefore
$J_{d}J_x=\pm J_x$ (recall that $J_w\in\mathcal{J}_{\Ga}$ is identified with $\fc_{\pi_{\Ga}}(w)$). Therefore $\gamma_{x,y,d}=\pm \langle J_x,J_{y^{-1}}\rangle_{\Ga}$, and Theorem~\ref{thm:innerp} forces $y^{-1}=x$. Thus $\conj{2}$ holds. 
\medskip

Consider $\conj{5}$. Note from the previous paragraph that the condition $\gamma_{x,y,d}\neq 0$ forces $x,d\in\Up$ for some right cell $\Up$ and $y=x^{-1}$. Moreover, $\gamma_{x,x^{-1},d}=\gamma_{d,x,x^{-1}}=\langle J_dJ_{x},J_x\rangle_{\Ga}$ where $\Ga$ is the two-sided cell containing $\Up$. Using (\ref{eq:dufloformula}) it follows that $\gamma_{x,x^{-1},d}=\epsilon \langle J_x,J_x\rangle_{\Ga}=\epsilon$ for some $\epsilon\in\{-1,1\}$ independent of~$x$. In particular, taking $x=d$ we have
$$
\epsilon=\gamma_{d,d^{-1},d}=\gamma_{d,d,d^{-1}}=\langle J_d^2,J_d\rangle_{\Ga}=\int_{\Omega^{-1}(\Ga)}\mathrm{tr}(\fc_{\pi_{\Ga}}(d)^3)\,d\mu'(\pi), 
$$
where we have used the fact that $d^2=e$. However, by~(\ref{eq:dufloformula}) we have $\fc_{\pi_{\Ga}}(d)^3=\epsilon \fc_{\pi_{\Ga}}(d)^2=\epsilon^2  \fc_{\pi_{\Ga}}(d)= \fc_{\pi_{\Ga}}(d)$, and hence 
$$
\epsilon=\int_{\Omega^{-1}(\Ga)}\mathrm{tr}( \fc_{\pi_{\Ga}}(d))\,d\mu'(\pi)=n_d,
$$
by (\ref{eq:nw}) and the fact that $n_d'=n_d$ for $d\in\cD$. Hence $\conj{5}$ holds.
\medskip

Conjectures $\conj{3}$ and $\conj{13}$ follow more easily.
\end{proof}

\begin{Rem}
We note that some efficiency could be gained by using the logical dependencies between the conjectures established in \cite[Chapter 14]{bible}. For example, $\conj{1}+\conj{3}\Rightarrow\conj{5}$, and $\conj{2}+\conj{3}+\conj{4}+\conj{5}\Rightarrow\conj{7}$. However we have found it instructive and illustrative to demonstrate each conjecture directly. For example, it is considerably more satisfying to see that $\conj{7}$ is in fact a consequence of an inner product structure on Lusztig's asymptotic algebra rather than a consequence of axioms $\conj{2}$, $\conj{3}$, $\conj{4}$ and $\conj{5}$. 
\end{Rem}

\subsection{Conjecture $\conj{15}$}\label{sec:P15}

In summary, using the explicit decomposition into cells, the calculation of the $\ba$-function, and the asymptotic Plancherel Theorem we have proved conjectures $\conj{1}$--$\conj{14}$ (see Corollaries~\ref{cor:conj14}, \ref{cor:afn}, \ref{cor:P8}, \ref{cor:P6}, \ref{cor:P71} and Theorem~\ref{thm:P1}). The remaining conjecture~$\conj{15}$ has been proved by Xie \cite[Theorem~6.2]{Xie:15} under an assumption on the $\ba$-function. We see below that this assumption is easily checked using the results of this section, and $\conj{15}$ follows.

\begin{Th}\label{thm:conj15}
Conjecture $\conj{15}$ holds for all choices of parameters. 
\end{Th}

\begin{proof}
By \cite[Theorems~6.2 and 6.3]{Xie:15} it is sufficient to verify that $\ba(d)=\deg h_{d,d,d}$ for all $d\in\cD$. This in turn is equivalent to showing that $\gamma_{d,d,d}\neq 0$ for all $d\in\cD$. As we saw in the proof of $\conj{5}$ above (see Corollary~\ref{cor:P71}) we have $\gamma_{d,d,d}=n_d=\pm 1$, and hence the result.
\end{proof}


\bibliographystyle{plain}

\begin{thebibliography}{10}

\bibitem{Bon:09}
C.~{B}onnaf\'e.
\newblock {S}emicontinuity properties of {K}azhdan-{L}usztig cells.
\newblock {\em New Zealand J. Math.}, 39:171--192, 2009.

\bibitem{Bon:17}
C.~{B}onnaf\'e.
\newblock {\em {K}azhdan-{L}usztig cells with unequal parameters}, Algebra and Applications, 24,
\newblock Springer, 2017.



\bibitem{B-I}
C.~Bonnaf\'e and L.~Iancu.
\newblock Left cells in type ${B}_n$ with unequal parameters.
\newblock {\em Represent. Theory}, 7:587--609, 2003.

\bibitem{Dix:77}
J.~{D}ixmier.
\newblock {\em $C^*$-algebras}, volume North-Holland Mathematical Library.
\newblock North-Holland Publishing Co., Amsterdam-New York-Oxford, 1977.

\bibitem{EW:14}
B.~{E}lias and G.~{W}illiamson.
\newblock The {H}odge theory of {S}oergel bimodules.
\newblock {\em Ann. of Math.}, 180(2):1089--1136, 2014.

\bibitem{Geck:02}
M.~{G}eck.
\newblock Constructible characters, leading coefficients and left cells for
  finite {C}oxeter groups with unequal parameters.
\newblock {\em Represent. Theory}, 6:1--30, 2002.

\bibitem{geck1}
M.~Geck.
\newblock Computing {K}azhdan-{L}usztig cells for unequal parameters.
\newblock {\em J. of Algebra $\mathbf{281}$, 342--365}, 2004.

\bibitem{Geck:11}
M.~Geck.
\newblock On {I}wahori-{H}ecke algebras with unequal parameters and {L}usztig's
  isomorphism theorem.
\newblock {\em Pure Appl. Math. Q.}, 7:587--620, 2011.

\bibitem{chevie2}
M.~Geck, G.~Hiss, F.~L{\"u}beck, G.~Malle, and G.~Pfeiffer.
\newblock {CHEVIE} -- {A} system for computing and processing generic character
  tables for finite groups of {L}ie type, {W}eyl groups and {H}ecke algebras.
\newblock {\em Appl. Algebra Engrg. Comm. Comput.}, 7:175--210, 1996.

\bibitem{Goe:07}
U.~{G}\"ortz.
\newblock Alcove walks and nearby cycles on affine flag manifolds.
\newblock {\em J. Algebraic Combin.}, 26:415--430, 2007.

\bibitem{guilhot4}
J.~Guilhot.
\newblock Kazhdan-{L}usztig cells in affine {W}eyl groups of rank 2.
\newblock {\em Int Math Res Notices}, 2010:3422--3462, 2010.

\bibitem{GP:17}
J.~{G}uilhot and J.~{P}arkinson.
\newblock A proof of {L}usztig's conjectures for affine type ${G}_2$ with
  arbitrary parameters.
\newblock {\em Proc. London Math. Soc.}, https://doi.org/10.1112/plms.12211, 2018.

\bibitem{KL1}
D.~A. Kazhdan and G.~Lusztig.
\newblock Representations of {C}oxeter groups and {H}ecke algebras.
\newblock {\em Invent. Math}, 53:165--184, 1979.

\bibitem{KL2}
D.~A. Kazhdan and G.~Lusztig.
\newblock {S}chubert varieties and {P}oincar\'e duality.
\newblock {\em Proc. Sympos. Pure Math.}, 34:185--203, 1980.

\bibitem{Lus1p}
G.~Lusztig.
\newblock Left cells in {W}eyl groups.
\newblock {\em Lecture Notes in Math.}, 1024:99--111, 1983.

\bibitem{bible}
G.~Lusztig.
\newblock {\em {H}ecke algebras with unequal parameters}.
\newblock CRM Monograph Series. Amer. Math. Soc., Providence, RI, 2003.

\bibitem{GP:00}
G.~{P}feiffer M.~{G}eck.
\newblock {\em Characters of finite {C}oxeter groups and {I}wahori-{H}ecke
  algebras}, volume~21 of {\em London Mathematical Society Monographs}.
\newblock Clarendon Press, Oxford, 2000.

\bibitem{chevie}
J.~Michel.
\newblock The development version of the {CHEVIE} package of {GAP}3.
\newblock {\em J. Algebra}, 435:308--336, 2015.

\bibitem{Op:04}
E.~{O}pdam.
\newblock On the spectral decomposition of affine {H}ecke algebras.
\newblock {\em J. Inst. Math. Jussieu}, 3:531--648, 2004.

\bibitem{Par:14}
J.~{P}arkinson.
\newblock On calibrated representations and the {P}lancherel {T}heorem for
  affine {H}ecke algebras.
\newblock {\em J. Algebraic Combin.}, 40:331--371, 2014.

\bibitem{Ram:06}
A.~{R}am.
\newblock Alcove walks, {H}ecke algebras, spherical functions, crystals and
  column strict tableaux.
\newblock {\em Pure and Applied Mathematics Quarterly (special issue in honor
  of Robert MacPherson)}, 2(4):963--1013, 2006.

\bibitem{GAP}
Martin Sch{\"o}nert et~al.
\newblock {\em {GAP} -- {Groups}, {Algorithms}, and {Programming} -- version 3
  release 4 patchlevel 4}.
\newblock Lehrstuhl D f{\"u}r Mathematik, Rheinisch Westf{\"a}lische Technische
  Hoch\-schule, Aachen, Germany, 1997.



\bibitem{SY:15}
J.-Y. Shi and G.~{Y}ang.
\newblock The weighted universal {C}oxeter group and some related conjectures of {L}usztig.
\newblock {\em J. Algebra}, 441:678--694, 2015.



\bibitem{Xie:15}
X.~Xie.
\newblock A decomposition formula for the {K}azhdan-{L}usztig basis of affine
  {H}ecke algebras of rank 2.
\newblock {\em Preprint available at arXiv:1509.05991}, 2015.

\end{thebibliography}

   \end{document}